\documentclass[11pt,a4paper]{amsart}
\usepackage[utf8x]{inputenc}
\usepackage[T1]{fontenc}
\usepackage{tikz}
\usepackage{amssymb}
\usepackage{amsmath}
\usepackage{hyperref}
\usepackage{mystyle}

\usepackage[shortlabels]{enumitem}
\usepackage{calc}
\usepackage{StandardNewCommands}
\usepackage{multirow}
\usetikzlibrary{arrows}
\usepackage{mathtools}
\usepackage{bbm}
\usepackage{numprint}
\usepackage{array}

\usepackage[style=alphabetic, backend=biber, backref=true,
 doi=false,isbn=false,url=false,
 eprint=true, sorting=nyt, maxnames=99]{biblatex}
\usepackage{url}

\usepackage[capitalize]{cleveref}
\crefname{lemma}{Lemma}{Lemmata}
\crefname{proposition}{Proposition}{Propositions}
\crefname{theorem}{Theorem}{Theorems}
\crefname{corollary}{Corollary}{Corollaries}
\crefname{algorithm}{Algorithm}{Algorithms}
\crefname{example}{Examples}{Examples}
\crefname{remark}{Remark}{Remarks}
\crefname{assumption}{Assumption}{Assumptions}
\crefname{enumi}{Step}{Steps}

\newcommand{\Z}{\mathbb{Z}}
\newcommand{\Q}{\mathbb{Q}}
\newcommand{\Qbar}{\overline{\mathbb{Q}}}
\newcommand{\Kbar}{\overline{K}}

\newcommand{\frakp}{\mathfrak{p}}

\newcommand{\lambdabar}{\overline{\lambda}}
\newcommand{\Nm}{\mathrm{Nm}}
\newcommand{\disc}{\mathrm{disc}}
\newcommand{\rank}{\mathrm{rank}}

\newcommand{\OM}{\mathcal{O}_M}
\newcommand{\calO}{\mathcal{O}}
\newcommand{\C}{\mathcal{C}}

\newcommand{\Jac}{\mathrm{Jac}}
\newcommand{\calB}{\mathcal{B}}
\newcommand{\Ht}{\mathrm{ht}}
\newcommand{\im}{\mathrm{im}}
\newcommand{\an}{\mathrm{an}}

\newcommand{\GM}{G_{M}}

\newcommand{\Art}{\mathrm{Art}}

\DeclareMathOperator*{\lcm}{lcm}

\title[Computing residual representations of IM abelian threefolds]{Computing residual representations of abelian threefolds with imaginary multiplication}

\author{Shiva Chidambaram}

\address{University of Wisconsin-Madison,
480 Lincoln Dr, Madison, WI 53706 USA.}
\email{chidambaram3@wisc.edu}
\urladdr{\url{https://people.math.wisc.edu/~chidambaram3/}}

\author{Pip Goodman}

\address{Max Planck Institute for Mathematics, Vivatsgasse 7, 53111 Bonn, Germany}
\email{pip.goodman@ub.edu}
\urladdr{\url{https://pipgoodman.github.io/}}


\date{}

\begin{document}

\maketitle
\begin{abstract}
Let $M$ be an imaginary quadratic field. Let $A$ be a polarised abelian threefold defined over $M$ with geometric endomorphism algebra isomorphic to $M$. We study residual Galois representations attached to $A$. In particular, we describe the endomorphism field and the natural restrictions placed on the residual representations by their endomorphisms. We then devise criteria for the image of residual Galois representations to be large, and provide efficient algorithms suitable for large scale calculations.

Subsequently, we apply this algorithm to millions of curves in several families, whose Jacobians are abelian threefolds with imaginary multiplication, and to Sutherland's dataset of $7$-smooth Picard curves. We produce an explicit example of a Picard curve which appears to have an isogeny of degree $13$ defined over $\mathbb{Q}(\zeta_3)$. The Jacobians of all other curves, in the range of our computation with endomorphism algebra $M$, have mod-$\ell$ image as large as possible for any prime $\ell>7$. This allows us to realise the group $\Gamma\mathrm{U}_3(\ell)$ of unitary semisimilitudes as a Galois group over $\mathbb{Q}$ for all $\ell \not\equiv 1, 25, 121 \pmod{168}$.

Our algorithms also led us to the discovery of several interesting rational families of Picard curves whose generic members appear to have endomorphism algebra of dimension $6$ which is not a CM field. These represent curves on the Picard modular surface and appear to be Shimura curves. Some of them appear to parametrise non-principally polarised abelian surfaces with quaternionic multiplication.
\end{abstract}

\section{Introduction}
Let $M$ be an imaginary quadratic field and fix an algebraic closure $\overline{M}$ of $M$.
In this article we study the residual representations associated to abelian threefolds $A$ defined over $M$, with a fixed polarisation and 
geometric endomorphism algebra isomorphic to $M$, i.e., 
$\End^0(A_{\overline{M}}) \coloneqq \End(A_{\overline{M}}) \otimes_{\ZZ} \QQ \cong M$.

The main contribution of this paper is an algorithm which given sufficiently many characteristic polynomials of Frobenius elements for such an abelian variety returns a small set of primes outside of which the residual mod-$\ell$ representations are as large as possible.
Moreover, the algorithm is sufficiently fast that it may be applied to large datasets of such abelian varieties.

The following theorem is obtained as a consequence. 
Recall that $\GammaU_3(\ell)$ denotes the group of semisimilitudes of a non-degenerate hermitian pairing on a 3 dimensional vector space  over  $\FF_{\ell^2}$ \cite[page 11, page 15 Table 2.1.B]{KleidmanLiebeck}.

\begin{theorem}[$\subseteq$ \cref{thm:inverse_Galois_explicit_curves}]
\label{thm:GammaU_is_GaloisgroupoverQ}
    Let $\ell$ be a prime, not congruent to $1, 25$ or $121$ modulo $168$. Then $\GammaU_3(\ell)$ is a Galois group over $\QQ$.
\end{theorem}

The algorithms developed in  \cref{section:criteriaforfullimage} do not require $\End^0(A_{\overline{M}}) \cong M$ and in fact may be applied to any abelian threefold $A$ with imaginary multiplication by $M$ of signature $(2,1)$ (see \cref{def:signature,def:imaginary_multiplication}).
However, if $\End^0(A_{\overline{M}}) \not \cong M$, then these algorithms will fail.
This turns out to be an effective way to discover such abelian varieties.
Indeed, using our algorithms we were able to discover several families of curves whose Jacobians have endomorphism algebras properly containing $M$.

We let $B_{d}$ denote the quaternion algebra over $\QQ$ with discriminant $d$.
\begin{theorem}
Let $a \in \QQ$ have height at most 20.
Let $f_a$ be one of the polynomials defined in the table below.
Suppose $f_a$ is separable.
Let $A$ be the Jacobian of the smooth affine curve defined by $y^3=f_a(x)$.
Then $\End^0(A_{\Qbar})$ contains the $\QQ$-algebra listed in the same row as $f_a(x)$ below.
\begin{center}

\begin{table}[h!]
\renewcommand{\arraystretch}{1.2}
\begin{center}
 \hspace*{-1.7cm}
\begin{tabular}{|p{13cm}|p{2.4cm}|}
    \hline
    $f_a(x)$ &  
    $\QQ$-algebra \\
    \hline
    $x(x^3+12x^2+ax+64)$ 

    &  $\Q(\zeta_3) \times M_2(\Q)$ \\

    $((x+1)^2 - a(5a+1))((x-1)^2 + 27a(a+1))$ & $\Q(\zeta_3) \times M_2(\Q)$ \\

    $(x^2 - 27a(a-1)(3a^2-3a+32))^2 + 2916a^2(a-1)(x - (a-1)(7a-16))^2$ & $\Q(\zeta_3) \times B_{15}$ \\

    $ (x - 3a)\big( (x + a)^3 - a(17a^3 + 6a^2 + 24a + 8)(x + a) + a^2(27a^4 + 14a^3 + 60a^2 + 24a + 32) \big)$ & $\Q(\zeta_3) \times B_6$ \\

    \hline
\end{tabular}
\end{center}
\end{table}
\end{center}    
\end{theorem}

We believe that for any $a \in \Qbar$ such that $f_a$ is separable, the geometric endomorphism algebra of the Jacobian of $y^3=f_a(x)$ will always contain the corresponding $\QQ$-algebra listed above and moreover be generically equal to it.

Let us now return to our initial motivation for studying this question.
Serre's Open Image Theorem \cite{Serre72} tells us that a given elliptic curve $E$ without complex multiplication defined over a number field $K$ satisfies $\Gal(K(E[\ell])/K) \cong \GL_2(\ell)$ for all but finitely many primes $\ell$.
For an elliptic curve $E/\QQ$ without complex multiplication it is widely believed that $\Gal(\QQ(E[\ell])/\QQ) \cong \GL_2(\ell)$ for all $\ell > 37$.
The truth of this statement is widely known as ``Serre's Uniformity Conjecture''.

Zywina and Sutherland \cite{Zywina_surjectivity,Sut16} have  developed algorithms which, provided with sufficiently many traces of Frobenius for a given elliptic curve $E/\QQ$ without complex multiplication, compute $\Gal(\QQ(E)/\QQ)$ as a subgroup of $\GL_2(\ell)$ (up to conjugation) for any prime $\ell$.
This in turn gives a large amount of data supporting Serre's Uniformity Conjecture.
Moreover, for $\ell \leq 37$, they produce an explicit list of possible small images, i.e., sets of couples $(G,\ell)$ consisting of a group $G$ and a prime $\ell$ such that $G \cong \Gal(\QQ(E[\ell])/\QQ) \subsetneq \GL_2(\ell)$ for some elliptic curve $E$ and conjecture that these are the only non-maximal mod-$\ell$ images possible.
There has been a lot of recent progress on these conjectures coming from studying the associated modular curves, see for example \cite{BDMTV19_Annals,RSZB22,BDMTV23,FL26_7adicGalimg}.

The study of the analogous problem for principally polarised abelian surfaces $A/\QQ$ with $\End(A_{\Qbar}) = \ZZ$ was initiated by Dieulefait \cite{Dieulefait_genus2}.
More specifically, he described an algorithm which given sufficiently many characteristic polynomials of Frobenius, returns a small finite set of primes, outside of which 
$\Gal(\Q(A[\ell]) / \Q) \cong \GSp_4(\ell)$.
His algorithm was recently implemented in SageMath by the authors of \cite{BBKKMSV23}.
This implementation is in turn a key input for the main algorithm of \cite{vBCCK23} which computes the isogeny class of a given principally polarised abelian surface.

It is a natural question to understand what the possible residual images of principally polarised abelian surfaces are.
Determining points on the moduli spaces in question, which have dimension three, seems to be far out of reach from current methods.
We seem, however, to be at the dawn of provably determining rational points on surfaces \cite{Caro_Pasten_Chabauty_for_surfaces}.
It is our hope that the data provided in \cref{sec:results} will aid and inform our understanding of which Picard modular surfaces should have either infinitely or finitely many rational points,
thus identifying interesting testing ground for the determination of rational points on surfaces.

Let us now briefly summarise the structure of the paper.
In \cref{sec:endomorphism_structure}, we analyse the endomorphism structure of abelian threefolds $A$ over a number field $K$ such that $\End^0(A_{\overline{K}}) \cong M$.
\cref{subsec:group_theory_centralisers} covers the group theoretic preliminaries we require for determining the restrictions on the images of the residual representations.
In \cref{sec:residual_reps} we then make these restrictions explicit; they allow us in particular to define mod-$\lambda$ representations and show in \cref{thm:containment_for_rho_lambda} that their images are contained in the group $\CL^\varepsilon_3(\ell)$ 
(see \cref{subsec:Notation} for the notation).
The action of inertia at $\ell$ on $A[\ell]$ is then studied in \cref{subsection:inertia_at_ell}.
This in particular allows us to prove surjectivity results for the determinants of the mod-$\lambda$ representations.
We begin \cref{section:criteriaforfullimage} by analysing the maximal subgroups of $\CL^\varepsilon_3(\ell)$ and then go on to give tests for the containment of the images of mod-$\lambda$ representations in these maximal subgroups.
Finally in \cref{sec:results}, we list the results obtained from our implementation of algorithms developed in \cref{section:criteriaforfullimage}.

Magma \cite{Magma} implementations of our algorithms, and collected data are available in the GitHub repo \cite{CG26_IMGalreps}.

\subsection{Notation}
\label{subsec:Notation}
Our notation for the classical groups is taken to reflect that of \cite{Bray_Holt_Roney-Dougal,KleidmanLiebeck}. 
Let $V$ be a vector space over a field $\FF$ and $\langle, \rangle \colon V \times V \rightarrow \FF$ a bilinear map.
There is a chain of subgroups associated to $\langle, \rangle$:
\[S \leq I \leq C \leq \Gamma\]
where $\Gamma$ is the set of semisimilitudes of $(V,\langle, \rangle)$, i.e., semilinear maps which preserve $\langle, \rangle$ up to a constant; $C$ is the set of similitudes of $(V,\langle, \rangle)$, i.e., the linear maps contained in $\Gamma$; $I$ is the set of isometries of $(V,\langle, \rangle)$, i.e., elements of $C$ preserving $\langle, \rangle$; and $S$ is the kernel of $\det \colon I \rightarrow \FF^*$.

If $\langle, \rangle$ is identically zero, then $S=\SL(V), I=C=\GL(V)$.
If $\langle, \rangle$ is a non-degenerate Hermitian pairing then $S=\SU(V), I = \GU(V), C=\CU(V)$ and $\Gamma =\GammaU(V)$.
If $\langle, \rangle$ is a non-degenerate symplectic pairing then $S=I=\Sp(V) $ and $ C=\GSp(V)$.
We also write $\SL^+=\SL, \GL^+=\CL^+ =\GL$ as well as $\SL^-=\SU, \GL^-=\GU$ and $\CL^-=\CU$.

We write $\C_1,\C_2,\C_3,\C_5,\C_6,\C_8,\mathcal{S}$ for the Aschbacher classes of maximal subgroups as defined in \cite[pp. 378-379]{Bray_Holt_Roney-Dougal}.

\subsection{Acknowledgments:}
We are grateful to Drew Sutherland for helping start our collaboration and his constant support which has been invaluable for this work. 
We also thank him for helpful discussions, and for sharing his dataset of $7$-smooth Picard curves.
We thank Edgar Costa for his help with the endomorphisms package, and his inputs for extending the scope of this work.
We are also grateful to Francesc Fit\'e, Davide Lombardo, Bjorn Poonen and Ramin Takloo-Bighash for helpful conversations.

SC was supported by Simons Foundation grant 550033 during part of this project and also partially supported by NSF grant DMS-2301386.

PG was supported the Deutsche Forschungsgemeinschaft (DFG) project grant STO 299/18-2 (AOBJ: 686837), by the Spanish Ministry of Science and Innovation via the grant “Abelian varieties, L-functions, and rational points” (code PID2022-137605NB-I00) and the Max-Planck-Institut für Mathematik.

The authors thank all the institutions at which they were hosted, whilst they carried out work on this project: Universität Bayreuth, Universitat de Barcelona, Max-Planck-Institut für Mathematik, Massachusetts Institute of Technology, and University of Wisconsin-Madison.

\section{The endomorphism structure}
\label{sec:endomorphism_structure}
Let $g$ be a positive integer, $K$ a number field and $\overline{K}$ a fixed algebraic closure of $K$ and $M/\QQ$ a quadratic extension.
The ultimate goal of this section is to define a class of abelian varieties generalising those of Jacobians of Picard curves.
Recall that the endomorphism algebra of an abelian variety is a semisimple $\QQ$-algebra.
We only wish to consider abelian varieties such that every simple $\QQ$-subalgebra of their geometric endomorphism algebra contains $M$.
For this reason, we make the following definition:

\begin{definition}[Unital algebra embedding]
\label{defintion:unital_algebra_embedding}
    Given $\QQ$-algebras $X, Y$ with multiplicative identities and a map $\iota \colon X \rightarrow Y$.
    We say $\iota$ is a unital algebra embedding, if it is injective, is a $\QQ$-linear map, preserves multiplication and sends the multiplicative identity of $X$ to that of $Y$.
\end{definition}

We shall now employ ideas 
from the theory of CM abelian varieties to study the endomorphism structure 
of a $g$-dimensional abelian variety $A$ over $K$ equipped with a unital algebra embedding
$\iota \colon M \hookrightarrow \End^0(A_{\Kbar})= \End(A_{\Kbar}) \otimes \QQ$.

We first recall some fundamental notions.
For an extension $K'/K$, we let $\Omega^1(A_{K})$ denote the 
space of invariant differential forms of $A_{K'}$.
It is a $g$-dimensional 
$K'$-vector space.
An endomorphism $f$ of $A$ acts on $\Omega^1(A_{K'})$ by 
pullback $f^*$.
Note that pullback reverses composition.
This action extends $\QQ$-linearly to a faithful representation
\begin{align*}
    \pi_{A,K',\an} \colon \End^0(A_{K'}) \rightarrow \End_{K'}(\Omega^1(A_{K'})), \quad f \mapsto f^*,
\end{align*}
which we call the analytic representation.

Let $\varphi, \Bar\varphi$ denote the 
two embeddings $M \xhookrightarrow{} \Kbar$. 
As $\pi_{A,\overline{K},\an} \circ \iota$ is a unital algebra embedding  and $M \otimes_\QQ \Kbar \cong \Kbar \times \Kbar$
the space $\Omega^1(A_{\Kbar})$ 
is the direct sum of the two eigenspaces
\begin{align*}
    \Omega_\varphi = \{\omega : \iota(\alpha)^*\omega = \alpha^{\varphi}\omega ~\forall~ \alpha \in M\}
    \text{ and }
    \Omega_{\Bar\varphi} = \{\omega : \iota(\alpha)^*\omega = \alpha^{\Bar\varphi}\omega ~\forall~ \alpha \in M\},
\end{align*}
where we write $\alpha^\varphi$ for $\varphi(\alpha)$.

\begin{definition}[Signature]
\label{def:signature}
    The tuple $(r,s) = (\dim \Omega_\varphi, \dim \Omega_{\Bar\varphi})$ is called the signature of $(A, \iota)$ with respect to the ordering $(\varphi, \Bar\varphi)$.
    We always assume $\varphi$ and $\Bar \varphi$ are chosen such that $r \geq s$.

    When the existence of $\iota$ is implicit from context, we shall instead refer to the signature of $(A,M)$.
\end{definition}

Equivalently, if we choose an eigenbasis $\omega_1, \dots, \omega_g$ of 
$\Omega^1(A_{\Kbar})$, so that for each $1 \leq i \leq g$, $M$ acts on $\omega_i$ through an embedding 
$\varphi_i : M \xhookrightarrow{} \Kbar$, i.e., 
$\iota(\alpha)^* \omega_i = \alpha^{\varphi_i}\omega_i$ for all $\alpha \in M$, 
then $\varphi$ occurs $r$ times in the multiset 
$\{ \varphi_i | 1 \leq i \leq g \}$ and $\Bar\varphi$ occurs $s$ times.
In other words, the characteristic polynomial of $\iota(\alpha)^*$ is 
$\prod_{i=1}^g (x - \alpha^{\varphi_i}) = (x-\alpha^{\varphi})^r(x-\alpha^{\Bar\varphi})^s$.

\begin{proposition}
\label{prop:endofield}
    Let $M$ be any quadratic field. Let $g \geq 1$ be odd, and $A/K$ be a $g$-dimensional abelian variety 
    equipped with an isomorphism $\iota \colon M \xrightarrow{\sim} \End^0(A_{\Kbar})$.

    Then $M$ is imaginary quadratic, and $MK$ is the minimal extension of 
    $K$ over which all endomorphisms of $A$ are defined.
\end{proposition}

\begin{proof}
    Let $\OO = \End(A_{\Kbar})$, so that $\OO$ is an order in $M$. Let $(r,s)$ with $r \geq s$ be the signature of $A$ with respect to the embeddings $\varphi, \Bar\varphi$ of $M$.
    Since $r+s=g$ is odd, we have $r \neq s$.

    We first show that $M$ is imaginary quadratic. The $2g$-dimensional $\QQ$-vector space $H_1(A(\CC),\QQ)$ carries a natural action of $M$ through $\iota\colon M \cong \End^0(A_{\overline{K}})$. Since $\iota$ is 
    a unital algebra embedding, this makes $H_1(A(\CC),\QQ)$ a vector space over $M$ of dimension $g$, so the multiplicity of each of 
    $\varphi$ and $\Bar\varphi$ in $H_1(A(\CC),\QQ)\otimes_\QQ \CC$ is $g$. On the 
    other hand, by \cite[Prop. 1.2.3]{BL04}, we know that
    there is an isomorphism of $\CC[\End^0(A_{\overline{K}})]$-representations between $H_1(A(\CC),\QQ) \otimes_\QQ \CC$
    and 
    the direct sum of the analytic representation $\pi_{A,\CC,\an}$ and its 
    complex conjugate $\Bar \pi_{A,\CC,\an}$. 
    If $M$ were real quadratic, both $\varphi$ and $\Bar\varphi$ would be real, 
    so that $\Bar \pi_{A,\CC, \an} = \pi_{A,\CC, \an}$.
    This would imply that $\varphi$ and $\Bar{\varphi}$ have multiplicity $2r$ and $2s$ in $H_1(A(\CC),\QQ) \otimes_\QQ \CC$, giving $2r=g=2s$, which is absurd.

    Let $L/K$ be a finite extension. It is enough to prove that all endomorphisms 
    of $A$ are defined over $L$ if and only if $M \subseteq L$.

    Suppose that all endomorphisms of $A$ are defined over $L$. Since $\iota(\alpha)^*$ 
    is $L$-linear on $\Omega^1(A_L)$, its trace lies in $L$. Calculating it 
    over $\Kbar$ using the signature,
    \begin{align*}
        \Tr(\iota(\alpha)^*) = r \alpha^{\varphi} + s \alpha^{\Bar\varphi}
        = (r-s)\alpha^{\varphi} + s \Tr_{M/\QQ}(\alpha).
    \end{align*}
    Since $r \neq s$, we find that $\alpha^\varphi \in L$. 
    This shows that $M = \varphi(M) \subseteq L$.

    Conversely, suppose that $M \subseteq L$. The absolute Galois group $G_L$ 
    acts on $\End^0(A_{\Kbar}) \cong M$ by ring automorphisms, leading to a homomorphism
    \[\varrho\colon  G_L \rightarrow \Gal(M/\QQ)\]
    defined by $\iota(\alpha)^{\sigma} = \iota(\alpha^{\varrho(\sigma)})$. Consider an 
    eigenbasis $\{ \omega_i | 1 \leq i \leq g \}$ of $\Omega^1(A_{\Kbar})$, with 
    corresponding embeddings $\varphi_i$. The group $G_L$ acts on
    $\Omega^1(A_{\Kbar}) = \Omega^1(A_L)\otimes_L \Kbar$ via its action on $\Kbar$, and 
    $\{ \omega_i^\sigma | 1 \leq i \leq g \}$ is another basis for any 
    $\sigma \in G_L$. We have
    \begin{align*}
        \left(\iota(\alpha)^{\sigma}\right)^* \omega_i^{\sigma} = \left(\iota(\alpha)^* \omega_i\right)^{\sigma}
        = \sigma(\alpha^{\varphi_i}) \omega_i^{\sigma}.
    \end{align*}
    Since $\alpha^{\varphi_i} \in \varphi_i(M) = M \subseteq L$, we know 
    $\sigma(\alpha^{\varphi_i}) = \alpha^{\varphi_i}$.
    So the characteristic polynomial of $\left(\iota(\alpha)^{\sigma}\right)^*$ is
    $\prod_i (x - \alpha^{\varphi_i})$. But the characteristic polynomial of
    $(\iota(\alpha^{\varrho(\sigma)}))^*$ is 
    $\prod_i (x - \alpha^{\varrho(\sigma)\varphi_i})$. 
    If $\varrho(\sigma)$ is not the identity, $\alpha^\varphi$ will be a root with 
    multiplicity $r$ of the first polynomial and multiplicity $s$ of the second. 
    Since $r \neq s$, we deduce that $\varrho(\sigma) = 1 \in \Gal(M/\Q)$ for all 
    $\sigma \in G_L$. This shows that all the endomorphisms of $A$ are defined 
    over $L$.
\end{proof}

\begin{definition}[Imaginary multiplication]
\label{def:imaginary_multiplication}
    Let $A/K$ be an abelian variety and $M$ an imaginary quadratic field.
    If there exists a unital algebra embedding $\iota \colon M \hookrightarrow \End^0(A_{KM})$, then we shall say that $A$ has potential imaginary multiplication by $(M, \iota)$.
    
    Furthermore, we say that $A$ has imaginary multiplication by $(M, \iota)$ if $\iota(M) \subseteq \End^0(A)$.
\end{definition}

The next proposition is originally due to Shimura \cite{Shi63} as part of a more general result, and is also given as \cite[Exercise 9.10 (3)]{BL04}.
We provide a concise proof in the specific case we require to aid readability. 

\begin{proposition}
\label{prop:g0_implies_Eg}
Let $M$ be an imaginary quadratic field, and $A/K$ be a $g$-dimensional abelian variety 
with imaginary multiplication by $(M,\iota)$.
If the signature of $(A, \iota)$ is $(g,0)$,
then $A$ is geometrically isogenous to $E^g$ for an elliptic curve $E/M$ with complex multiplication by $M$.
\end{proposition}
\begin{proof}
    We may write $A(\CC) = \CC^g/\Lambda$ some lattice $\Lambda$ in $\CC^g$.
    Let $\pi \coloneqq \pi_{A,\CC,\an}$ denote the analytic representation of $\End(A)$ with respect to this analytic 
    coordinate system. 
    
    After possibly replacing $A$ by an isogenous abelian variety, we may assume 
    $\iota(M) \cap \End(A) = \iota(\OO_M)$, where $\OO_M$ is the maximal order in $M$.    
    To simplify notation, let us identify $\iota(\OO_M)$ with $\OO_M$.
    The lattice $\Lambda$ is a finitely generated torsion-free $\OO_M$-module under $\pi$.
    By the structure theorem for such modules over Dedekind domains \cite[Prop. 10.2.1]{Over_the_nest}, we may write 
    $\Lambda = \pi(I_1)x_1+\pi(I_2)x_2+\ldots \pi(I_g)x_g$ for some ideals $I_n$ of $\OO_M$ and 
    $x_n \in \CC^g$ for $1 \leq n \leq g$.

    Write $x_n = (x_{n,1},x_{n,2},\ldots,x_{n,g})$.
    Since $A$ is of signature $(g,0)$, there is some embedding $\varphi \colon M \hookrightarrow \CC$ 
    such that for any $\alpha \in \OO_M$,
    \[\pi(\alpha)x_n = (\varphi(\alpha)x_{n,1},\varphi(\alpha)x_{n,2},\ldots,\varphi(\alpha)x_{n,g}).\]
    Thus $\pi(\alpha)x_n = \varphi(\alpha)x_n$ where $\varphi(\alpha) \in \CC$ is a scalar.
    As $\RR \otimes \varphi(I_n) \cong \CC$, we see that 
    $\left(\RR \otimes \pi(I_n)x_n\right)/\pi(I_n)x_n \cong \left(\RR \otimes \varphi(I_n)x_n\right)/\varphi(I_n)x_n \cong \CC/\varphi(I_n)$.

    Define a $\CC$-linear map $T\colon \CC^g \rightarrow \CC^g$ by $x_n \mapsto e_n$ (where $e_n$ is $1$ in the 
    $n$-th component and zero elsewhere).
    As $\Lambda$ is a lattice in $\CC^g$, $x_1,x_2,\ldots,x_g$ are linearly independent over $\CC$ and thus $T$ is 
    invertible. It follows that $T$ induces a biholomorphic map
    \[\CC^g/\Lambda \rightarrow \CC/\varphi(I_1) \times \CC/\varphi(I_2) \times \cdots \times \CC/\varphi(I_g).\]
    The $\CC/\varphi(I_n)$ are isogenous elliptic curves with CM by $\OO_M$ (see for example 
    \cite[Chapter II, \S 1]{silvermanII}, or \cite[Chapter II, Thm. 2]{Shimura_CM_book} and its corollary).
\end{proof}

\begin{corollary}
\label{cor:signature_is_(2,1)}
    Let $M$ be an imaginary quadratic field and $A/K$ be an abelian threefold with endomorphism algebra $\End^0(A)$ isomorphic to $M$.
    Then the signature of $(A,M)$ is $(2,1)$.
\end{corollary}

\section{Group theoretic preliminaries}
\label{subsec:group_theory_centralisers}
Let $g \geq 2$ be an integer and $\ell$ a prime number.
In this section, $V$ denotes a $2g$-dimensional vector space over $\FF_\ell$ equipped with a non-degenerate alternating pairing $\langle, \rangle \colon V \times V \rightarrow \FF_\ell$.
We let $\GSp(V)$ denote the similitude group of $\langle, \rangle$.
That is, $\GSp(V)$ is the group of $\FF_\ell$-linear maps $V \rightarrow V$ such that for every $\sigma \in \GSp(V)$, there exists a constant $\chi(\sigma)$ such that for any $v, w \in V$, we have:
\[\langle \sigma v, \sigma w \rangle = \chi(\sigma) \langle v,w\rangle.  \]
It is easy to check that the map $\chi \colon \GSp(V) \rightarrow \FF_\ell^*$ defined by $\sigma \mapsto \chi(\sigma)$ is a homomorphism.
We define $\Sp(V)$ to be the isometry group of $\langle, \rangle$, in other words, the kernel of $\chi$.

We shall now investigate the centralisers of certain elements in $\GSp(V)$ (see Propositions \ref{prop:centralisers_grp_thy_split_case}, \ref{prop:centralisers_grp_thy_inert_case}).
The results presented here are almost certainly well-known to group theorists, however we were unable to find an adequate reference for them.
The reader looking to read more about related questions may wish to consult \cite{Springer_centralisers,wall,OBrienDeFranLiebeck}.

In the following we denote the transpose of a matrix $T$ by $T^t$ and its inverse transpose by $T^{-t}$.

\begin{proposition}
    \label{prop:centralisers_grp_thy_split_case}
    Suppose $\alpha \in \GSp(V)$ has a reducible minimal polynomial $m_\alpha$ of degree 2 with distinct roots, 
    and satisfies $m_\alpha(0) = \chi(\alpha)$. 
    Then the centraliser $C_{\GSp(V)}(\alpha)$ of $\alpha$ in $\GSp(V)$ admits an isomorphism:
    \[\varphi \colon C_{\GSp(V)}(\alpha) \xrightarrow{\sim} \FF_\ell^* \times \GL_g(\ell)\]
    sending $C_{\GSp(V)}(\alpha) \cap \Sp(V)$ to $\{ 1 \} \times \GL_g(\ell)$.
    
    Furthermore, writing $m_\alpha = \phi \phi'$ with $\deg \phi=\deg \phi '=1$ and $V_\phi=\ker(\phi(\alpha))$, $V_{\phi'}=\ker(\phi'(\alpha))$:
    \begin{enumerate}[i)]
        \item \label{part1:splitcentralizer} the natural action of $C_{\GSp(V)}(\alpha)$ on $V$ preserves both $V_\phi$ and $V_{\phi'}$.
        \item \label{part2:splitcentralizer} Both $V_\phi$ and $V_{\phi'}$ are maximal totally isotropic subspaces of $V$.
        \item \label{part3:splitcentralizer} We may take $\varphi$ so that the natural action of $C_{\GSp(V)}(\alpha)$ on $V$ translates via $\varphi$ to the action of $(\mu,T)\in \FF_\ell^* \times \GL_g(\ell) $ on $ V_\phi$ as $\mu T$ and on  $V_{\phi'}$ as $T^{-t}$.

    \end{enumerate}
    Moreover, taking $\varphi$ so that it satisfies $iii)$ and letting $\gamma' \in \GSp(V)$ denote an element which exchanges $V_\phi$ and $V_{\phi'}$, we may extend $\varphi$ to an isomorphism
    \[\langle \gamma', C_{\GSp(V)}(\alpha) \rangle  \xrightarrow{\sim} \left(\FF_\ell^* \times \GL_g(\ell) \right) \rtimes \langle \gamma \rangle\]
    where $\gamma(\mu, T) = (\mu , \mu^{-1}T^{-t})$.    
\end{proposition}

\begin{proof}

    Suppose first that $m_\alpha$ is reducible.
    By assumption we may write $m_\alpha(x) = (x-a)(x-\chi(\alpha)a^{-1})$ for some $a \in \FF_\ell^*$ with $a \neq \chi(\alpha)a^{-1}$.
    Let $v,w  \in V$ be eigenvectors of $\alpha$ with eigenvalues $a,b$ respectively.
    We have 
    \[\chi(\alpha)b^{-1}\langle v, w \rangle=\chi(\alpha)\langle v, \alpha^{-1}w \rangle = \langle \alpha v,w \rangle =a \langle v, w \rangle .\]
    Thus either $\langle v,w \rangle =0$ or $b=\chi(\alpha)a^{-1}$.
    Hence we see that the restriction of $\langle, \rangle$ to the eigenspace $V_a \coloneqq \ker(\alpha -a)$ is trivial.
    Similarly,  the restriction of $\langle, \rangle$ to the eigenspace $V_{\chi(\alpha)a^{-1}} \coloneqq \ker(\alpha - \chi(\alpha)a^{-1})$ is trivial.

    It follows that we may take a basis $e_1, \ldots , e_g$ of $V_a$ such that $\langle e_i,e_j \rangle =0$ for all $i,j$ and a basis of $f_1, \ldots , f_g$ of $V_{\chi(\alpha)a^{-1}}$ such that $\langle f_i,f_j \rangle =0$ and $\langle e_i,f_j \rangle =\delta_{ij}$ for all $i,j$ where $\delta_{ij}=1$ if $i=j$ and $0$ otherwise.

    Let $\sigma \in C_{\GSp(V)}(\alpha)$.
    Note that as $\sigma$ commutes with $\alpha$ it preserves its eigenspaces.
    Write $T \in \GL_g(\ell)$ (resp. $S \in \GL_g(\ell)$) for the matrix obtained by letting $\sigma$ act on the basis $e_1, \ldots , e_g$ (resp. $f_1, \ldots , f_g$).
    The condition that $\sigma$ is a similitude of $\langle, \rangle$ with similitude factor $\chi(\sigma)$ may then be expressed as:
    \[ \begin{pmatrix}
        T^t &0\\
        0 & S^t
    \end{pmatrix} \begin{pmatrix}
        0 &I_g\\
        -I_g & 0
    \end{pmatrix}\begin{pmatrix}
        T &0\\
        0 & S
    \end{pmatrix}= \chi(\sigma)\begin{pmatrix}
        0 &I_g\\
        -I_g & 0
    \end{pmatrix}.\]
    From which it follows that if $\chi(\sigma)=1$, then $S=T^{-t}$.
    In fact, we deduce that the map $\GL_g(\ell) \rightarrow C_{\GSp(V)}(\alpha) \cap \Sp(V)$ defined by $X \mapsto \left(\begin{smallmatrix}
        X &0\\
        0 & X^{-t}
    \end{smallmatrix}\right)$ is an isomorphism.
    Let $\mu \in \FF_\ell^*$.
    Observe that
    \[\begin{pmatrix}
        \mu I_g &0\\
        0 & I_g
    \end{pmatrix} \begin{pmatrix}
        0 &I_g\\
        -I_g & 0
    \end{pmatrix}\begin{pmatrix}
        \mu I_g &0\\
        0 & I_g
    \end{pmatrix}= \mu\begin{pmatrix}
        0 &I_g\\
        -I_g & 0
    \end{pmatrix}.\]
    Since $\GSp(V)$ is generated by $\Sp(V)$ and the linear map defined by $e_i \mapsto \mu e_i$, $f_i \mapsto f_i$, we deduce $C_{\GSp(V)}(\alpha) \cong \FF^*_\ell \times \GL_g(\ell)$.

    We shall now prove the final statement.
    Let us write $\gamma'$ as $\left(\begin{smallmatrix}
        0 & T' \\
        S' & 0
    \end{smallmatrix}\right)$ with respect to the basis $e_1, \ldots ,e_g,f_1, \ldots ,f_g$.
    After possibly multiplying $\gamma'$ by a matrix of the form $\left(\begin{smallmatrix}
        \mu I_g &0\\
        0 & I_g
    \end{smallmatrix}\right)$, we may assume the following equation holds:
    \[ \begin{pmatrix}
        0 & (S')^t\\
        (T')^t & 0
    \end{pmatrix} \begin{pmatrix}
        0 &I_g\\
        -I_g & 0
    \end{pmatrix}\begin{pmatrix}
        0 & T'\\
        S' & 0
    \end{pmatrix}= \begin{pmatrix}
        0 &-I_g\\
        I_g & 0
    \end{pmatrix}\]
    from which we deduce $S' = (T')^{-t}$.
    Observing that conjugation by
    \[\begin{pmatrix}
        (T')^{-1} & 0 \\
         0 & (T')^{t} 
    \end{pmatrix}\begin{pmatrix}
        0 & T' \\
        (T')^{-t} & 0
    \end{pmatrix}=\begin{pmatrix}
        0 & I_g \\
        I_g & 0
    \end{pmatrix}\]
    satisfies 
     \[ \begin{pmatrix}
         0 & I_g\\
        I_g  & 0
    \end{pmatrix} \begin{pmatrix}
        \mu T & 0\\
        0 & T^{-t}
    \end{pmatrix}\begin{pmatrix}
        0 & I_g\\
        I_g & 0
    \end{pmatrix}= \begin{pmatrix}
        T^{-t} & 0 \\
        0 & \mu T
    \end{pmatrix}= \begin{pmatrix}
        \mu (\mu T)^{-t} & 0 \\
        0 & (\mu T)
    \end{pmatrix}\]
    completes the proof.
\end{proof}

We shall now investigate the case where $m_\alpha$ is irreducible of degree 2.
Note that the characteristic polynomial of $\alpha$ is necessarily $m_\alpha^g$, since it is defined over $\FF_\ell$, has degree $2g$, and its only possible irreducible factor over $\FF_\ell$ is $m_\alpha$.
We shall then use the isomorphism $\FF_\ell[\alpha] \cong \FF_\ell[x]/\langle m_\alpha \rangle \cong \FF_{\ell^2}$ to view $V$ as a $g$-dimensional vector space over $\FF_{\ell^2}$.

The elements in $C_{\GSp(V)}(\alpha)$ preserve this structure on $V$.
We shall shortly prove that they in fact preserve a non-degenerate Hermitian form on $V$ up to similitude.
Our approach is inspired by \cite[Section 2]{chi_Mumford_Tate}, which was in turn inspired by  \cite[Lemmas 4.6, 4.7]{Deligne_Hodge_on_AVs}.

We let $\Tr$ denote the trace map $\Tr \colon \FF_{\ell^2} \rightarrow \FF_\ell$.

\begin{lemma}
\label{lem:unique_Hermitian_form}
    Suppose $\alpha \in \GSp(V)$ has an irreducible minimal polynomial $m_\alpha$ of degree 2 satisfying $m_\alpha(0) = \chi(\alpha)$.

    Let $\xi \in \FF_{\ell}[\alpha]^*$ be such that $\xi^{\ell-1}=-1$, i.e., $\Tr(\xi) = 0$.
    Then there exists a unique non-degenerate $\FF_{\ell}[\alpha]$-Hermitian form $\langle , \rangle_\# \colon V \times V \rightarrow \FF_{\ell}[\alpha]$ such that $\langle v , w \rangle = \Tr(\xi \langle v, w \rangle_{\#})$ for all $v,w \in V$.
\end{lemma}

\begin{proof}
We begin by noting that we may equip $V$ with a second structure of an $\FF_{\ell^2}$-vector space.
Indeed, the automorphism $\FF_{\ell^2} \rightarrow \FF_{\ell^2}$ given by $x \mapsto x^\ell$ allows us to equip $V$ with a $\FF_{\ell^2}$-vector space structure where $\alpha$ acts by $\overline{\alpha} \coloneqq \alpha^\ell$.
We write $\overline{V}$ for $V$ viewed with this alternative $\FF_{\ell^2}$ structure.

Let us write $\langle , \rangle'$ for $\langle , \rangle$ viewed as an $\FF_\ell$-bilinear form $ V \times \overline{V} \rightarrow \FF_\ell$.
Note that 
\[\langle \alpha v, w\rangle' =\langle \alpha v, w\rangle = \langle v, \chi(\alpha)\alpha^{-1} w \rangle = \langle v, \overline{\alpha} w\rangle = \langle v, \alpha w\rangle '\]
where in the above we have used that $m_\alpha$ is irreducible of degree 2 and $m_\alpha(0) = \chi(\alpha)$.
By the universal property of the tensor product, we find that $\langle, \rangle'$ determines an element of $\Hom_{\FF_\ell}(V \otimes_{\FF_{\ell^2}} \overline{V}, \FF_\ell)$.
Postcomposing by $\Tr$ yields an isomorphism $\Hom_{\FF_{\ell^2}}(V \otimes_{\FF_{\ell^2}} \overline{V}, \FF_{\ell^2}) \rightarrow\Hom_{\FF_\ell}(V \otimes_{\FF_{\ell^2}} \overline{V}, \FF_\ell)$ \cite[Lemma 4.3]{Deligne_Hodge_on_AVs}.
We thus deduce that there exists a unique non-degenerate $\FF_{\ell^2}$-bilinear form $\langle,\rangle'_0 \colon V \times \overline{V} \rightarrow\FF_{\ell^2}$ such that $\Tr\langle v,w\rangle'_0=\langle v,w\rangle'$ for all $v \in V, w \in \overline{V}$.
Correspondingly, we see that there is a unique non-degenerate sesquilinear form $\langle,\rangle_0 \colon V \times V \rightarrow\FF_{\ell^2}$ such that $\Tr\langle v,w\rangle_0=\langle v,w\rangle$ for all $v \in V, w \in V$.

Define $\langle, \rangle_\# \coloneqq \xi^{-1} \langle, \rangle_0$.
To prove $\langle, \rangle_\# \colon V \times V \rightarrow \FF_{\ell^2}$ is a non-degenerate Hermitian form it remains to prove that $\overline{\langle v, w \rangle }_\#=\langle w,v \rangle_\#$.
To this end, consider
\[\Tr (\xi \langle v,w \rangle_\# )=\langle v,w \rangle=-\langle w,v \rangle=\Tr (\xi \overline{\langle w,v \rangle}_\# ).\]
Replacing $v$ by $\xi^{-1} av$ with $a \in \FF_{\ell^2}$, we find $\Tr (a \langle v,w \rangle_\# )=\Tr (a\overline{\langle w,v \rangle}_\# )$.
Since multiplication followed by $\Tr$ yields a non-degenerate bilinear map $ \FF_{\ell^2} \times \FF_{\ell^2} \rightarrow \FF_\ell$ and the previous equality holds for all $a \in \FF_{\ell^2}$, we deduce $\langle v,w \rangle_\# = \overline{\langle w,v \rangle}_\#$ as required.

If there exist two such Hermitian forms, then their difference is also a Hermitian form $\langle,\rangle_!$, and it satisfies $0 = \Tr( \xi \langle v,w \rangle_!)$ for all $v,w \in V$. Since $\Tr(\xi) = 0$, this can only happen if $\langle , \rangle_!$ is $\FF_\ell$-valued. Since it is $\FF_{\ell^2}$-linear in the first variable, we deduce that $\langle , \rangle_! = 0$. This proves uniqueness.
\end{proof}

We let $\CU(V)$ denote the similitude group of $\langle, \rangle_{\#}$ and $\GU(V)$ its isometry group.
The similitude group $\CU(V)$ is generated by $\GU(V)$ and the invertible scalar matrices over $\FF_{\ell^2}$ \cite[(2.3.3)]{KleidmanLiebeck}.

\begin{proposition}
\label{prop:centralisers_grp_thy_inert_case}    
    Suppose $\alpha \in \GSp(V)$ has an irreducible minimal polynomial $m_\alpha$ of degree 2 satisfying $m_\alpha(0) = \chi(\alpha)$.
    Then the centraliser $C_{\GSp(V)}(\alpha)$ admits an isomorphism:
    \[C_{\GSp(V)}(\alpha) \cong \CU(V)\]
    sending $C_{\GSp(V)}(\alpha) \cap \Sp(V)$ to $ \GU(V)$ where the $\FF_{\ell^2}$-vector space structure on $V$ is given by the natural action of $\FF_\ell[\alpha]$.
\end{proposition}

\begin{proof}
    We shall keep the notation introduced in the statement of \cref{lem:unique_Hermitian_form}.
    Let $ \sigma \in \Aut_{\FF_\ell[\alpha]}(V) = C_{\GL(V)}(\alpha)$.
    We may thus define the map $(,) \colon V \times V \rightarrow\FF_{\ell^2}$ via $(v,w) \coloneqq \langle \sigma v, \sigma w \rangle_\#$ for any $v,w \in V$.
    It is easy to see that $(,)$ is a Hermitian pairing on $V$.
    Note that 
    \[\Tr(\xi \langle \sigma v, \sigma w \rangle_\#)=\Tr(\xi \langle v, w \rangle_\#) \text{ if and only if } \langle \sigma v, \sigma w \rangle =\langle v,  w \rangle.\]
    In particular, we see that $\GU(V) \hookrightarrow C_{\GSp(V)}(\alpha) \cap \Sp(V)$.
    Moreover, the uniqueness statement in Lemma \ref{lem:unique_Hermitian_form} shows that if $\sigma \in \Sp(V)$, then $\langle \sigma v, \sigma w \rangle_\# =\langle v, w \rangle_\#$.
    Hence $\GU(V) \cong C_{\GSp(V)}(\alpha) \cap \Sp(V)$.

    Now let $a \in \FF_{\ell^2}^*$ and by abuse of notation write $a$ again for the corresponding scalar matrix in $\CU(V)$.
    We have $\langle av,aw\rangle=\Tr(\xi \langle a v, a w \rangle_\#)=\Tr(\xi \Nm(a) \langle v, w \rangle_\#)=\Nm(a)\Tr(\xi \langle v, w \rangle_\#)=\Nm(a) \langle v, w \rangle$.
    We deduce that $\CU(V) \hookrightarrow C_{\GSp(V)}(\alpha)$.
    Moreover, since the norm is surjective, we find that this injection is in fact an isomorphism since the index of $\Sp(V)$ in $\GSp(V)$ is $\ell-1$.
\end{proof}

\section{Residual representations of IM abelian threefolds}
\label{sec:residual_reps}

\subsection{Restrictions on the residual images}
\label{subsec:restrictions_on_residual_images}
Let $A/K$ be a polarised abelian variety of dimension $g \geq 2$ with potential imaginary multiplication by $M$.
The preimage $\OO$ of $\End(A_{KM})$ under $\iota$ is well-known to be an order in $M$.
The restrictions on the action of Galois on $A[\ell]$ is influenced by whether or not $\ell$ divides the degree of our fixed polarisation $A \rightarrow A^\vee$ or the index of $\OO$ in the maximal order of $M$.
In \cref{subsection:inertia_at_ell}, it will be important to restrict to primes $\ell$ such that $A$ has semistable reduction at the primes lying over $\ell$ in $M$.
For these reasons, we introduce the following definition:

\begin{definition}
    We call a prime $\ell$ \emph{admissible} for $A$ if it satisfies all the conditions below:
    \begin{enumerate}
        \item \label{item:excludepolarisationdivisors} $\ell$ does not divide the degree of the fixed polarisation.
        \item \label{item:excluderamifiedprimes} $\ell$ does not divide the discriminant of $\OO \coloneqq \iota^{-1}\left(\End(A_{KM})\right)$.
    \end{enumerate}
    If in addition to the above, $A$ has semistable reduction at every prime above $\ell$ in $M$, we say $\ell$ is a \emph{semistable admissible} prime.
\end{definition}

For the rest of this section we assume $\ell$ is an admissible prime.
Recall that the absolute Galois group $G_{K} \coloneqq \Gal(\overline{K}/K)$ acts on the $\ell$-torsion of $A$, yielding a representation
\[\rho_\ell \colon G_{K} \rightarrow \Aut(A[\ell]).\]
Since $\ell$ is coprime to the degree of our fixed polarisation $A \rightarrow A^\vee$, there exists a non-degenerate alternating pairing on $A[\ell]$, commonly referred to as the Weil pairing.
The action of $\OO$ on $A[\ell]$ commutes with that of $\rho_\ell(G_{KM})$.

For a prime ideal $\lambda$ lying above $\ell$, we write \[A[\lambda] \coloneqq \{P \in A(\overline{K})| \iota(\alpha)P = 0 \,\forall \alpha \in \lambda\}.\]
Clearly $A[\lambda]$ is a vector space over $\FF_\lambda \coloneqq \OO/\lambda$ and by \cite[Prop. 2.2.1]{ribet_RM} its dimension over $\FF_\lambda$ is $g$.
The absolute Galois group $G_{KM}$ acts on $A[\lambda]$ again yielding a representation
\[\rho_\lambda \colon G_{KM} \rightarrow \Aut(A[\lambda]).\]
Note that as $G_{KM}$ is a normal subgroup of $G_K$, for any $\sigma \in G_K$ we may define a representation $^\sigma \rho_{\lambda}$ of $G_{KM}$ by $^\sigma \rho_{\lambda}(\tau) \coloneqq \rho_{\lambda}(\sigma \tau \sigma^{-1})$.
In the following theorem we describe $^\sigma \rho_{\lambda}$ and give upper bounds on the possible images of $\rho_\ell$ and $\rho_\lambda$.

Conserving the notation of \cref{subsec:group_theory_centralisers}, we further let $\GammaU_g(V)$ denote the group of semisimilitudes of $\langle, \rangle _\#$ on $V$ \cite[Pg. 11]{KleidmanLiebeck}.

\begin{theorem}
    \label{thm:modellimagesarecontainedin}
    Let $\ell$ be an admissible prime for $A$.
    Suppose that either $\ell$ is odd, or $2$ is inert in $M$. 
    
    If $\ell$ splits in $M$ as $\lambda \lambdabar$, then  $\rho_\ell(G_{KM}) \subseteq \FF_{\ell}^* \times \GL_g(\ell)$, $\rho_\lambda(G_{KM}) \subseteq \GL_g(\ell)$ and $\rho_\ell(G_{KM(\zeta_\ell)}) \subseteq \GL_g(\ell)$.
    If, furthermore, $M \cong \End^0(A_{KM})$, 
    then  $\rho_\ell(G_K) \subseteq \left( \FF_{\ell}^* \times \GL_g(\ell) \right) \rtimes \langle \gamma \rangle$ where $\gamma$ acts by $(\mu, T) \mapsto (\mu, \mu^{-1}T^{-t})$ and for any $\sigma \in G_K\setminus G_{KM}$ we have $^\sigma\rho_\lambda=\rho_{\Bar{\lambda}}$.

    If $\ell$ is inert in $M$, then $\rho_\ell(G_{KM}) \subseteq \CU_g(\ell)$ and $\rho_\ell(G_{KM(\zeta_\ell)}) \subseteq \GU_g(\ell)$.
    If, furthermore, $M \cong \End^0(A_{KM})$, then $\rho_\ell(G_K) \subseteq \GammaU_g(\ell)$.
\end{theorem}

\begin{proof}
    Since $\OO$ is an order in a quadratic field, we know that $\OO = \Z[f \omega]$ where $\OM = \Z[\omega]$ is the maximal order and $f = [\OM:\OO]$ . As $\ell$ does not divide the discriminant of $\OO$, we know that $\ell \nmid f$. Let $\alpha = f \omega$ and $m_\alpha$ denote the minimal polynomial of $\alpha$.
    If $m_\alpha(0)\equiv 0 \pmod{\ell}$, then $\ell$ splits in $M$ and as $\ell \neq 2$ there exists $n \in \ZZ$ such that $m_{\alpha+n}(0) \neq 0$.
    As $\ZZ[\alpha+n]=\ZZ[\alpha]$, we may assume $m_\alpha(0) \not \equiv 0 \pmod{\ell}$.

    Let $P,Q \in T_\ell(A)$ and denote by $\langle, \rangle \colon T_\ell(A) \times T_\ell(A) \rightarrow \ZZ_\ell$ the Weil pairing.
    Since $\ell$ is admissible, the pairing is perfect.
    Standard properties of the Rosati involution \cite[pgs. 189, 192]{Mumford_ab_vars} along with Albert's Classification \cite[Thm. 2, pg. 201]{Mumford_ab_vars} 
    give us the following equalities: 
    \[\langle \iota(\alpha)P,\iota(\alpha)Q \rangle = \langle P , \iota(\overline{\alpha}) \iota(\alpha)Q \rangle =\Nm_{M/\QQ}(\alpha)\langle P,Q\rangle .\]
    Since $\Nm_{M/\QQ}(\alpha)=m_\alpha(0) \not \equiv 0 \pmod{\ell}$, it follows that $\alpha$ induces a similarity on $A[\ell]$ with respect to the Weil pairing.
    The minimal polynomial of $\iota(\alpha)$ acting on the Tate module is given by $m_\alpha$ (combine \cite[Chapter II, Thm. 3.6]{Lang_CM} and \cite[Chapter II, Lemma 1]{Shimura_CM_book}).
    Thus the minimal polynomial of the induced similarity on $A[\ell]$ is $m_\alpha$ reduced modulo $\ell$.

    Since $\ZZ[\alpha]$ is $\ell$-maximal, we may apply Dedekind-Kummer factorisation to see $m_\alpha \pmod{ \ell}$ is irreducible if $\ell$ is inert in $M$ and is reducible with distinct roots if $\ell$ splits in $M$.
    In particular, Dedekind-Kummer factorisation tells us that if $\ell$ splits in $M$, then writing $m_\alpha = \phi \phi'$ with each of $\phi$ and $\phi'$ as degree 1 polynomials we have, up to exchanging $\phi$ and $\phi'$, $\lambda= \langle \ell, \phi(\alpha) \rangle$ and $\lambdabar= \langle \ell, \phi'(\alpha) \rangle$.
    Thus for $\ell$ split in $M$, we have $A[\lambda] = A[\ell] \cap \Ker(\phi(\alpha))$ and $A[\lambdabar] = A[\ell] \cap \Ker(\phi'(\alpha))$.
    As $\iota (M)$ is contained in $\End^0(A_{KM})$, 
    the action of $G_{KM}$ commutes with that of $\iota(\alpha)$ on $A[\ell]$.
    Thus the result for $G_{KM}$ follows from applying \cref{prop:centralisers_grp_thy_split_case} and \cref{prop:centralisers_grp_thy_inert_case}.

    If $\iota(M) \subseteq\End^0(A_{K})$, then the result for $G_{K}$ follows by the same reasoning as that of $G_{KM}$.
    Thus let us suppose $\iota(M) \not \subseteq\End^0(A_{K})$.
    Then the action of $G_{K}$ on $\End^0(A_{KM}) \cong M$ furnishes a surjection $\varrho \colon G_{K} \rightarrow \Gal(M/\QQ)$ with kernel $G_{KM}$.
    
    In particular, for $\sigma \in G_{K}$ and $\alpha \in M$ we have $\iota(\alpha)^\sigma=\iota(\alpha^{\varrho(\sigma)})$.
    For the action on $A[\ell]$, we thus have $ \rho_\ell(\sigma)\iota(\alpha)=\iota(\alpha^{\varrho(\sigma)})\rho_\ell(\sigma)$.
    The result for $G_K$ when $\ell$ is inert in $M$ is thus clear.

    Let us consider the case where $\ell$ is split in $M$.
    Since the minimal polynomial of $\alpha$ has degree two, we see that if $\sigma \in G_K \setminus G_{KM}$, then $\alpha^{\varrho(\sigma)}= \Nm_{M/\QQ}(\alpha)\alpha^{-1}$.
    Thus if $\sigma \in G_K \setminus G_{KM}$ and $v$ is an eigenvector of $\iota(\alpha)$ with eigenvalue $a$, then \[\iota(\alpha) \rho_\ell(\sigma)v=\rho_\ell(\sigma)\iota(\alpha^\varrho(\sigma))v= \Nm_{M/\QQ}(\alpha) \rho_\ell(\sigma)\iota(\alpha^{-1})v=\Nm_{M/\QQ}(\alpha)a^{-1}\rho_\ell(\sigma)v,\]
    from which we deduce that $\sigma$ exchanges the eigenspaces of $\alpha$.
    For $\tau \in G_{KM}$, we have $^\sigma\rho_\lambda(\tau) \coloneqq \rho_\lambda(\sigma \tau \sigma^{-1}) = \rho_\ell(\sigma \tau \sigma^{-1})|_{A[\lambda]}$. Since $\sigma$ gives an isomorphism $A[\Bar \lambda] \xrightarrow{\simeq} A[\lambda]$ of the underlying vector spaces, and the action of $\tau$ on $A[\Bar \lambda]$ is through $\rho_{\Bar \lambda}(\tau)$, we deduce that $^\sigma\rho_\lambda \cong \rho_{\Bar \lambda}$.
    
    The final result then follows from applying \cref{prop:centralisers_grp_thy_split_case}.
\end{proof}

As is common in group theory (see \cite[Pg. 6]{Bray_Holt_Roney-Dougal} or \cite[Pg. 15]{KleidmanLiebeck}), we will use the notation $\GL^+=\GL$ and $\GL^-=\GU$.
That is, $\GL^+$ and $\GL^-$ denote the isometry group of the zero pairing and a non-degenerate hermitian pairing respectively.
It is convenient for us to further extend this notation to the similitude groups, thus writing $\CL^+=\GL$ and $\CL^-=\CU$.

The following is a direct consequence of Theorem \ref{thm:modellimagesarecontainedin}. 

\begin{corollary}
    \label{thm:containment_for_rho_lambda}
    Let $\ell$ be an admissible prime for $A$, and $\lambda$ a prime above $\ell$ in $M$.
    Writing $\varepsilon=\left(\frac{\Delta_M}{\ell}\right)$, we have
     \[\rho_\lambda(G_{KM}) \subseteq \CL^\varepsilon(A[\lambda]) \]
    and \[\rho_\lambda(G_{KM(\zeta_\ell)}) \subseteq \GL^\varepsilon(A[\lambda]). \]
\end{corollary}

\subsection{Background on $\lambda$-adic representations}
\label{subsection:lambda-adic_reps}
From now on $A$ will denote an abelian threefold defined over $M$ with \emph{imaginary multiplication} by $(M, \iota)$ of signature $(2,1)$. 

Let $T_\ell(A) \coloneqq \varprojlim A[\ell]$ denote the Tate module of $A$, and $V_\ell(A) = T_\ell(A) \otimes_{\ZZ_\ell} \QQ_\ell$.
The endomorphism algebra $\End^0(A)$ acts faithfully on $V_\ell(A)$ and its action commutes with that of $\QQ_\ell$, making $V_\ell(A)$ into an $M \otimes \QQ_\ell \cong  \prod_{\lambda|\ell}M_\lambda$ module via $\iota$, where here $\lambda|\ell$ runs over all primes above $\ell$ in $M$ and $M_\lambda$ is the completion of $M$ at $\lambda$.
The decomposition $M \otimes \QQ_\ell = \prod_{\lambda|\ell}M_\lambda$ allows us to write $1 \in M \otimes \QQ_\ell$ as a sum of idempotents $1 = \sum_{\lambda | \ell} e_\lambda$ where $e_\lambda$ is the multiplicative identity in $M_\lambda$ and the additive identity in $M_{\lambda'}$ for $\lambda' \neq \lambda$.
The vector space $e_\lambda V_\ell(A)$ naturally inherits the structure of an $M_\lambda$-vector space, which we denote by $V_\lambda(A)$.
Each $V_\lambda(A)$ is a $M_\lambda$-vector space of dimension $3$ \cite{ribet_RM}.
Furthermore, since $\Res^{M_\lambda}_{\QQ_\ell} V_\lambda(A) = e_\lambda V_\ell(A)$ and $M \hookrightarrow \End_{\QQ_\ell}(V_\ell)$ is a unital algebra embedding, we have a decomposition
\[V_\ell(A) = \prod_{\lambda|\ell} \Res^{M_\lambda}_{\QQ_\ell}(V_\lambda(A)).\]

Thus the representation $\rho_{\ell^\infty}\colon \GM \rightarrow \Aut_{\QQ_\ell}(V_\ell(A))$ commutes with the action of $M\otimes \QQ_\ell$.
This in turn implies each $V_\lambda(A)$ is an $M_\lambda[\GM]$-module.
We denote the corresponding representation by
\[\rho_{\lambda^\infty} \colon \GM \rightarrow \Aut_{M_\lambda}(V_\lambda(A)).\]
If $\ell$ is an admissible prime, then we get a corresponding decomposition of the Tate module leading to a lattice $T_\lambda(A)$ inside $V_\lambda(A)$ for every prime $\lambda |\ell$.
We then have an isomorphism of $\FF_\lambda[\GM]$-modules $T_\lambda(A)/\lambda \cdot T_\lambda(A) \cong A[\lambda]$.

This family of representations $(\rho_{\lambda^\infty})_\lambda$ is known to form a strictly compatible system \cite[Thm 2.12]{ribet_RM}.
It follows that their determinants $(\det \circ \rho_{\lambda^\infty})_\lambda$ also form a strictly compatible system.
Such a system of $\lambda$-adic representations is known to arise from an algebraic Hecke character $\Omega$ of $M$ with values in $M$ and conductor dividing the conductor of $A$ \cite[Prop. 1.4, page 25]{Periods_of_Hecke_characters}.
That is, the compatible system of $\lambda$-adic representations $(\Omega_{\lambda^\infty})_\lambda$ associated to $\Omega$ as in \cite[Chapter 0, \S 5]{Periods_of_Hecke_characters} coincides with $(\det \circ \rho_{\lambda^\infty})_\lambda$.
Generalising the Shimura-Taniyama formula, Fité has shown that the infinity type of $\Omega$ is determined by the action of $\iota(M)$ on $\Omega^1(A)$ \cite[Prop. 14]{Fite_Ordinary_primes}.
In particular, as $A$ has imaginary multiplication by $M$, we deduce that the infinity type of $\Omega$ is either $2\varphi^{-1}+\overline{\varphi}^{-1}$ or $\varphi^{-1}+2\overline{\varphi}^{-1}$.

\subsection{Action of inertia at $\ell=p$}
\label{subsection:inertia_at_ell}
Let $\lambda$ be an unramified prime of $M$ above the rational prime $\ell$, $k_\lambda$ the residue field of $M_\lambda$ and $M_\lambda^{ur}$ be the maximal unramified extension of $M_\lambda$.
For any $d$ with $\ell \nmid d$, we denote by $\mu_d$ the $d$-th roots of unity in $M_\lambda^{ur}$, which we identify with their images in $\Bar{k}_\lambda$.

Let $\pi$ be a uniformiser of $M_\lambda^{ur}$.
Consider the map 
\[\theta_{\{d\}, \lambda} \colon \Gal(M_\lambda^{ur}(\pi^{1/d})/M_\lambda^{ur}) \rightarrow \mu_d\]
defined by sending the element $\sigma \in \Gal(M_\lambda^{ur}(\pi^{1/d})/M_\lambda^{ur})$ to the $d$-th root of unity $\theta_{\{d\}, \lambda}(\sigma)$ satisfying
\[\sigma(\pi^{1/d}) = \theta_{\{d\}, \lambda}(\sigma)\pi^{1/d}.\]
One may check that $\theta_{\{d\}, \lambda}$ does not depend on the choice of uniformiser $\pi$ nor its $d$-th root and that it is in fact a group isomorphism.

Clearly the $\theta_{\{d\}, \lambda} \colon I^t_\lambda \rightarrow \Bar{k}_\lambda^*$ give characters of the tame inertia group $ I^t_\lambda$.
In fact, they give rise to all characters.
Indeed, denoting the localisation of $\ZZ$ by $\ZZ \setminus (\ell)$ by $\ZZ_{(\ell)}$, we have:
\begin{proposition}{\cite[Prop. 5]{Serre72} }
\label{prop:classification_of_tame_characters}
    The map $\ZZ_{(\ell)}/\ZZ \rightarrow \Hom(I^t_\lambda , \Bar{k}_\lambda^* )$ defined by $\frac{b}{d} \mapsto \theta_{\{d\},\lambda}^b$ is an isomorphism of groups.
\end{proposition}

Note that characters $\theta_{\{\ell^n-1\}, \lambda} \colon I^t_\lambda \rightarrow \FF_{\ell^n}^*$ are surjective, as are their powers $\theta_{\{\ell^n-1\}, \lambda}^{\ell^s}$, for $s \in \ZZ/n\ZZ$.
We call such a character a fundamental character of level $n$ and write $\theta_{n, \lambda} \coloneqq\theta_{\{\ell^n-1\}, \lambda}$.
In fact, when what is meant is clear from context we shall write $\theta_{n}$ in place of $\theta_{n, \lambda}$.

Note also that $\chi_\ell|_{I_\lambda}$, that is, the restriction of the mod-$\ell$ cyclotomic character $\chi_\ell \colon \GM \rightarrow \FF_\ell^*$, post composed with $\FF_\ell^* \hookrightarrow \Bar{k}_\lambda^*$ gives rise to an element of  $\Hom(I^t_\lambda , \Bar{k}_\lambda^* )$ and hence equals $\theta_{1, \lambda}$ by \cref{prop:classification_of_tame_characters}.
We shall reserve the notation $\chi_\ell$ for the cyclotomic character when viewed as a global representation, i.e., of $\GM$ and write $\theta_{1, \lambda}$ (or $\theta_1$) when viewed as a character of $I_\lambda$.

\begin{proposition}
\label{prop:det_char_at_ell}
Let $\lambda$ be a prime of $M$ above an admissible prime $\ell$.
If $\ell$ splits in $M$, then either $\det \circ \rho_\lambda|_{I_\lambda} = \theta_1$ and $\det \circ \rho_{\lambda}|_{I_{\overline{\lambda}}} = \theta_1^{2}$, or $\det \circ \rho_{\lambda}|_{I_\lambda} = \theta_1^2$ and $\det \circ \rho_{\lambda}|_{I_{\overline{\lambda}}} = \theta_1$.

If $\ell$ is inert in $M$, then either $\det \circ \rho_{\lambda}|_{I_\lambda} = \theta_2^{2+\ell}$ or $\theta_2^{1+2\ell}$.
\end{proposition}

\begin{proof}
As summarised in \cref{subsection:lambda-adic_reps}, there exists an algebraic Hecke character $\Omega$ of $M$ with values in $M$, conductor dividing that of $A$ and infinity type either $2\varphi^{-1}+\overline{\varphi}^{-1}$ or $\varphi^{-1}+2\overline{\varphi}^{-1}$ such that its $\lambda$-adic avatars $\Omega_{\lambda^\infty}$ coincide with $\det \circ \rho_{\lambda^\infty}$.
It is explained in \cite[Section 3.4]{Serre72} (see also \cite[Section 6.1]{Goodman_superelliptic} for a summary) that we may reduce $\Omega_{\lambda^\infty}$ modulo $\lambda$ to obtain a representation $\Omega_{\lambda}$, which  coincides with $\det \circ \rho_{\lambda}$.
Furthermore, if $I$ is the inertia group at a prime above $\ell$ coprime to the conductor of $\Omega$, the restriction of $\Omega_{\lambda}$ to $I$ is shown to be the product of powers of fundamental characters of level $[\FF_\lambda \colon \FF_\ell]$, where the powers are given by the coefficients in the infinity type of $\Omega$ (see \cite[Eq (6.1)]{Goodman_superelliptic}). 
\end{proof}

By definition, a semisimple representation is determined by its simple constituents.
We shall make use of this in the following proposition to keep notation light.
More precisely, we shall write $(\rho_\lambda|_I)^{ss}$ equal to the sum of three characters for some group $I$ to mean that the simple factors of the semisimplification $(A[\lambda]|_{I} \otimes_{\FF_\lambda}\overline{\FF}_\lambda)^{ss}$ are given by these characters.

\begin{proposition}
\label{prop:inertia_at_l_description}
Let $\ell$ be a semistable admissible prime for $A$.
Let $I$ be an inertia subgroup corresponding to a prime above $\ell$ in $\GM$. 

 Suppose $\ell$ splits in $M$.
 Then either $\det \circ \rho_\lambda|_I = \theta_1$ or $\theta_1^2$.
 In the former case $(\rho_\lambda|_{I})^{ss}$ is given by one of $2 \mathbbm{1} + \theta_1, \mathbbm{1}+ \theta_2 + \theta_2^\ell$ and $ \theta_3+ \theta_3^\ell + \theta_3^{\ell^2}$.
In the latter case  $(\rho_\lambda|_{I})^{ss}$ is given by one of $ \mathbbm{1} + 2\theta_1, \theta_1+ \theta_2 + \theta_2^\ell$ and $ \theta_3^{1+\ell}+ \theta_3^{\ell+\ell^2} + \theta_3^{1+\ell^2}$.

 Suppose $\ell$ is inert in $M$. 
Then either  $\det \circ \rho_\lambda|_I  = \theta_2^{2+\ell}$ or $\theta_2^{1+2\ell}$.
In the former case $(\rho_\lambda|_{I})^{ss}$ is given by one of $2 \theta_2 + \theta_2^\ell$, $\mathbbm{1}+ \theta_1+ \theta_2$, or $\theta_6^{e}+\theta_6^{e\ell^2}+ \theta_6^{e\ell^4}$ where $e = 1+\ell+\ell^2$.
In the latter case, $(\rho_\lambda|_{I})^{ss}$  is given by one of $ \theta_2 + 2\theta_2^\ell$, $\mathbbm{1}+ \theta_1+ \theta_2^\ell$, or $\theta_6^{e}+\theta_6^{e\ell^2}+ \theta_6^{e\ell^4}$ where $e = 1+\ell +\ell^5$.
\end{proposition}

\begin{proof}
    Let $\psi_1$ be an irreducible $\overline{\FF}_\lambda[I]$-submodule of  $A[\lambda] \otimes_{\FF_\lambda}\overline{\FF}_\lambda$.
    As $\ell$-groups act trivially on irreducible modules in characteristic $\ell$, the wild inertia subgroup of $I$ acts trivially on $\psi_1$.
    Moreover, as the tame quotient of $I$ is pro-cyclic, the action of $I$ on $\psi_1$ is via a cyclic group.
    Hence we deduce that $\psi_1$ is one dimensional.
    Repeating the argument, we may write $(\rho_\lambda|_{I})^{ss} = \psi_1+\psi_2+\psi_3$.
    Our main goal is to determine the set $\Psi \coloneqq \{\psi_1,\psi_2,\psi_3\}$.

    We shall use the following restrictions on $\Psi$.
    First, we clearly have $\psi_1\psi_2\psi_3 = \det \circ \rho_\lambda|_I $.
    Second, as explained in the proof of \cite[Lemma 4.9]{Larson_Vaintrob}, it follows from \cite[Thm. 3.4.3]{Raynaud_type_p} and \cite[Exposé IX]{SGA7} that we may write each $\psi_i = \theta_n^e$ for some $n$ and $e = \sum_{j=0}^{n-1}e_j \ell^j$ with $e_j \in \{0,1\}$. 
    
    Suppose $\ell$ splits in $M$.
    As $A[\lambda]$ is a 3-dimensional vector space over $\FF_\ell$, each $\psi_i$ is a power of a fundamental character of level at most 3.
    By Proposition \ref{prop:det_char_at_ell}, we have either $\det \circ \rho_\lambda|_I=\theta_1$ or $\theta_1^2$.
    Let us suppose we are in the former case, i.e., $\det \circ \rho_\lambda|_I =\theta_1$.
    If one of $\psi_1,\psi_2,\psi_3$ equals $\theta_1$, then $\Psi \coloneqq \{\theta_1,\mathbbm{1},\mathbbm{1} \}$.
    Thus let us suppose in the following that $\theta_1 \not \in \Psi$.
    If $\theta_2 \in \Psi$, then so is its conjugate $\theta_2^\ell$, in which case $\Psi = \{\theta_2,\theta_2^\ell,\mathbbm{1}\}$.
    Thus we may suppose $\theta_3^{e} \in \Psi$ for some $e = e_0+e_1\ell+e_0\ell^2 \neq 1+\ell+\ell^2$, which implies $\Psi =\{\theta_3^e,\theta_3^{e\ell},\theta_3^{e\ell^2}\}$.
    As $\theta_1=\theta_3^e\cdot\theta_3^{e\ell}\cdot\theta_3^{e\ell^2}= \theta_3^{e(1+\ell+\ell^2)}= \theta_1^e$, we find $e \equiv 1 \pmod{\ell-1}$.
    Thus $\{e_0,e_1,e_2\}=\{1,0,0\}$, so $\Psi = \{\theta_3,\theta_3^{\ell},\theta_3^{\ell^2}\}$.
    An analogous argument when $\det \circ \rho_\lambda|_I =\theta_1^2$ determines the possibilities for $(\rho_\lambda|_{I})^{ss}$.

    Suppose $\ell$ is inert in $M$.
    As $A[\lambda]$ is a 3-dimensional vector space over $\FF_{\ell^2}$, the set $\Psi$ must be stable under the action of $G_{\FF_{\ell^2}}$, which implies each $\psi_i$ is a power of a fundamental character of level 1,2,3,4, or 6. 
    Moreover, by Proposition \ref{prop:det_char_at_ell}, we have either $\det \circ \rho_\lambda|_I =\theta_2^{2+\ell}$ or $\theta_2^{1+2\ell}$ thus $\Psi$ must contain a character which is a power of a fundamental character of even level that does not factor through a fundamental character of odd level.
    In particular, no $\psi_i$ can be a power of a fundamental character of level 3, unless it is $\theta_1$.
    Let us first consider the case $\det \circ \rho_\lambda|_I =\theta_2^{2+\ell}$.

    Let us first suppose $\Psi$ only contains fundamental characters of level at most 2.
    If $\theta_1 \in \Psi$, then clearly $\Psi = \{\theta_1, \theta_2, \mathbbm{1}\}$.
    If $\theta_2 \in \Psi$, but $\theta_1 \not \in \Psi$, then  $\Psi = \{\theta_2,\theta_2,\theta_2^\ell\}$.

    Let us now suppose $\theta_4^e \in \Psi$ and $\theta_4^e$ is not $\mathbbm{1}, \theta_1, \theta_2$ or $\theta_2^\ell$.
    Then $\Psi=\{\theta_4^e,\theta_4^{e\ell^2},\theta_2^{e'}\}$ for some $e'$.
    Raynaud's Theorem implies that $e = \sum_{i=0}^3 e_i \ell^i$ and $e' = e'_0+e'_1\ell$ for some $e_i, e_i' \in \{0,1\}$.
    We deduce that $\Psi$ may be equal to one of the following sets  $\{\theta_4^{1+\ell+\ell^2},\theta_4^{1+\ell^2+\ell^3}, \mathbbm{1}\}, \{\theta_4,\theta_4^{\ell^2}, \theta_1\}$,  $\{\theta_4^{1+\ell},\theta_4^{\ell^2+\ell^3}, \theta_2\}$,   or $\{\theta_4^{1+\ell^3},\theta_4^{\ell+\ell^2}, \theta_2\}$.
    It is easy to check in each of these cases that the order inertia at $\ell$ will contain a element of order $1+\ell^2$.
    We shall show this is incompatible with \cref{thm:containment_for_rho_lambda}.
    The order of $\CU_3(\ell)$ equals $\ell^3(\ell+1)(\ell^2-1)^2(\ell^3+1)$.
    Note that 4 does not divide $1+\ell^2$, since otherwise $\ell \pmod{4}$ would be a square root of $3 \pmod{4}$.
    The greatest common divisor of $1+\ell^2$ with any of $\ell+1,\ell^2-1,\ell^3+1$ is 2 if $\ell$ is odd and $1$ otherwise.
    Thus $1+\ell^2$ does not divide the order of $\CU_3(\ell)$.
    Since the image of $\rho_\lambda$ is contained in $\CU_3(\ell)$ by \cref{thm:containment_for_rho_lambda}, we deduce that $\theta_4^e \in \Psi$ if and only if $\theta_4^e$ is one of $\mathbbm{1}, \theta_1, \theta_2$ or $\theta_2^\ell$.

    We are left to consider the case where $\theta_6^e \in \Psi$ and $\theta_6^e$ is not $\mathbbm{1}, \theta_1, \theta_2$ or $\theta_2^\ell$, then $\Psi=\{\theta_6^e,\theta_6^{e\ell^2},\theta_6^{e\ell^4}\}$.
    By Raynaud's Theorem, we may write $e = \sum_{i=0}^5 e_i \ell^i$ for some $e_i \in \{0,1\}$.
    Since $\theta_2^{2+\ell}=\theta_6^e\cdot\theta_6^{e\ell^2}\cdot\theta_6^{e\ell^4}= \theta_6^{e(1+\ell^2+\ell^4)}= \theta_2^e$, we see $e \equiv (e_0+e_2+e_4) + (e_1+e_3+e_5)\ell \equiv 2+\ell \pmod{\ell^2-1}$.
    That is, $\{e_0,e_2,e_4\}=\{1,1,0\}$ and $\{e_1,e_3,e_5\}=\{1,0,0\}$.
    It follows that there are at $\binom{3}{2} \cdot \binom{3}{1} /3 =3$ possibilities for $\Psi$ in this situation.

    By considering the order of $\CU_3(\ell)$, we will, as above, discount 2 of these possibilities, leaving $\Psi= \{\theta^e_6, \theta^{e\ell^2}_6, \theta^{e \ell^4}_6\}$ with $e = 1+\ell+\ell^2$ as the unique possibility.
    To do so, it suffices to show the orders of $\theta_6^{1+\ell^2+\ell^3}(I)$ and $\theta_6^{1+\ell^2+\ell^5}(I)$ do not divide the order of $\CU_3(\ell)$.

    We begin by proving their orders are a multiple of $\frac{1}{9}(\ell^6-1)$.
    We provide the proof for $\theta_6^{1+\ell^2+\ell^3}(I)$, the other case being similar.
    Let $e = 1+\ell^2+\ell^3$.
    Again using modular arithmetic, we find $\gcd(e, \ell+1)=1$,$\gcd(e, \ell^2-\ell+1)=1$ and $\gcd(e,\ell-1)$ divides 3.
    Likewise, we obtain $e \equiv \ell(\ell+1)(\ell-1) \equiv 1-\ell \pmod{\ell^2+\ell+1}$.
    Thus $\gcd(e, 1+\ell+\ell^2)$ divides 3.
    If follows that $\gcd(e, \ell^6-1)$ divides 9 and thus the order of $\theta_6^e$ is a multiple of $\frac{1}{9}(\ell^6-1)$.

    The order of $\CU_3(\ell)$ equals $\ell^3(\ell+1)(\ell^2-1)^2(\ell^3+1)$.
    The gcd of $\# \CU_3(\ell)$ and $1+\ell+\ell^2$ equals 1 if $\ell \not \equiv 1 \pmod{3}$ and it is equal to 3 if $\ell \equiv 1 \pmod{3}$.
    For $\ell \geq 7$, we have that $1+\ell+\ell^2 \geq 57>3^3$, thus $\CU_3(\ell)$ cannot contain an element of order $\frac{1}{9}(\ell^6-1)$.
    Hence if $\Psi = \{\theta^e_6, \theta_6^{e\ell^2}\, \theta_6^{e\ell^4}\}$ we must have either $e = 1+\ell+\ell^2$.

    An analogous argument when $\det \circ \rho_\lambda|_I =\theta_2^{1+2\ell}$ determines the possibilities for $(\rho_\lambda|_{I})^{ss}$.
    We note that the factorisation $\ell^5+\ell+1=(\ell^2+\ell+1)(\ell^3-\ell^2+1)$ is used for the case when $\Psi$ contains a character which is a power of a fundamental character of level 6, not factoring through one of lower level.
\end{proof}

\begin{proposition}
\label{prop:det_surjects_and_GM_surjects_mod_lambda_iff_GMZetaell_surjects}
The homomorphism  $\Omega_\lambda =\det \circ \rho_{\lambda} \colon G_{M(\zeta_\ell)} \rightarrow \FF_\lambda^*$ is surjective if $\ell$ splits in $M$ and it surjects onto $\left(\FF_\lambda^*\right)^{\ell-1}$ if $\ell$ is inert in $M$.

 Moreover, for $\ell$ split in $M$, we have $\rho_\lambda(G_M)= \GL_3(\ell)$ if and only if $\rho_\lambda(G_{M(\zeta_\ell)})= \GL_3(\ell)$.
\end{proposition}

\begin{proof}
    Write $\Nm_{M/\QQ}(\lambda)=\ell^f$ where $f \in \{1,2\}$ is the inertia degree of $\lambda$.
    As $M$ is unramified at $\ell$, we have that $\theta_f \colon I_{\lambda'} \rightarrow \FF_\lambda^*$ is surjective, for $\lambda' \in \{\lambda,\overline{\lambda}\}$.
    In particular, there exists $\sigma' \in I_\lambda $ such that $\theta_f(\sigma')$ has order $\ell+(-1)^{f}$ and $\sigma'' \in I_{\overline \lambda}$ such that $\theta_f(\sigma'')=\theta_f(\sigma')^{-1}$.
    If $\ell$ is split in $M$, let $\sigma \coloneqq \sigma' \sigma''$.
    If $\ell$ is inert in $M$, let $\sigma \coloneqq \sigma'$.
    Using that $\theta_2^{\ell+1}=\theta_1$ and $\chi_\ell|_{I_\lambda}=\theta_1$, we have $\chi_\ell(\sigma)=1$.
    Thus $\sigma \in G_{M(\zeta_\ell)}$.

    We shall now apply \cref{prop:det_char_at_ell} to show $\Omega_\lambda(\sigma)$ has order $\ell+(-1)^f$ in $\FF_\lambda^*$.
    For $\ell$ split in $M$, consider $\Omega_\lambda(\sigma)=\Omega_\lambda(\sigma')\Omega_\lambda(\sigma'')=\chi_\ell(\sigma')^{\pm 1}$.
    For $\ell$ inert in $M$, we have $\Omega_\lambda(\sigma)=\theta_2(\sigma)^{\pm 1}$.
    This completes the proof of the first statement.

    Let us now suppose $\ell$ splits in $M$ and prove the final statement.
    If $\rho_\lambda(G_{M(\zeta_\ell)})= \GL_3(\ell)$ then it is clear by Theorem \ref{thm:containment_for_rho_lambda} that $\rho_\lambda(G_{M})= \GL_3(\ell)$.
    Thus let us suppose $\rho_\lambda(G_{M})= \GL_3(\ell)$.
    As $\Gal(M(\zeta_\ell)/M)$ is abelian, we see that $\rho_\lambda(G_{M(\zeta_\ell)})$ contains the commutator subgroup of $\rho_\lambda(G_{M})= \GL_3(\ell)$, that is, $\rho_\lambda(G_{M(\zeta_\ell)}) \supseteq \SL_3(\ell)$ since $\SL_3(\ell)$ is perfect \cite[Thm. 1.7]{Grove}.
    As $\det \circ \rho_{\lambda} \colon G_{M(\zeta_\ell)} \rightarrow \FF_\lambda^*$ is surjective, we deduce $\rho_\lambda(G_{M(\zeta_\ell)}) = \GL_3(\ell)$.
\end{proof}

For an admissible prime $\ell$ which splits in $M$ as $\lambda \lambdabar$, we will 
now show that surjectivity of the $\lambda$-torsion representation 
$\rho_\lambda$ implies surjectivity of the $\ell$-torsion representation 
over $M$. Note that this is a tautological statement for inert primes.

\begin{proposition}
    \label{prop:rholambdasurjectiveimpliesrhoellsurjective}
    Let $\ell$ be an admissible prime for an abelian threefold $A/M$ with 
    imaginary multiplication by $M$ of signature $(2,1)$, such that $\ell$ splits in $M$.
    Let $\lambda$ be a prime above $\ell$ in $M$.
    If $\rho_\lambda(\GM) = \GL_3(\ell)$, 
    then  
    $\rho_\ell(\GM) = \FF_\ell^* \times \GL_3(\ell)$.
\end{proposition}

\begin{proof}
    By \cref{prop:det_surjects_and_GM_surjects_mod_lambda_iff_GMZetaell_surjects}, we have $\rho_\lambda(G_{M(\zeta_\ell)})= \GL_3(\ell)$.
    Restricting the action of $G_{M(\zeta_\ell)}$ from $A[\ell]$ to $A[\lambda]$, we obtain a surjective homomorphism $\rho_\ell(G_{M(\zeta_\ell)}) \rightarrow \rho_\lambda(G_{M(\zeta_\ell)})=\GL_3(\ell)$.
    Moreover, owing to the mod-$\ell$ cyclotomic character $\chi_\ell$ being the similitude character of the Weil pairing on $A[\ell]$, we have a surjective homomorphism $\rho_\ell(G_M) \rightarrow \chi_\ell(G_M)$ with kernel $\rho_\ell(G_{M(\zeta_\ell)})$.
    As $M$ is unramified at $\ell$, $\chi_\ell(G_M)= \FF_\ell^*$.
    By Theorem \ref{thm:containment_for_rho_lambda}, $\rho_\ell(G_M) \subseteq \FF_\ell^* \times \GL_3(\ell)$.
    Hence for cardinality reasons, we deduce $\rho_\ell(G_M) = \FF_\ell^* \times \GL_3(\ell)$.
\end{proof}

\begin{theorem}
    \label{thm:big_lambda_image_gives_big_mod_ell_image_over_Q}
    Let $A/\QQ$ be an abelian threefold such that $\End^0(A_{\Qbar}) \cong M$.
    Let $\ell$ be an admissible prime for $A$ which is either odd, or inert in $M$.
    Suppose that $\rho_\lambda(G_M) = \CLe(A[\lambda]) \cong \CL^\varepsilon_3(\ell)$.
\begin{itemize}
    \item     If $\ell$ splits in $M$, then $\rho_\ell(G_\QQ) \cong \left( \FF_{\ell}^* \times \GL_3(\ell) \right) \rtimes \langle \gamma \rangle$ where $\gamma$ acts by $(\mu, T) \mapsto (\mu, \mu^{-1}T^{-t})$.
    \item  If $\ell$ is inert in $M$, then $\rho_\ell(G_\QQ) \cong \GammaU_3(\ell)$.
\end{itemize}
\end{theorem}

\begin{proof}
    By \cref{prop:rholambdasurjectiveimpliesrhoellsurjective}, for $\ell$ split in $M$, we have $\rho_\ell (G_M)\cong \FF_\ell^* \times \GL_3(\ell)$.
    For $\ell$ inert in $M$ we have $\rho_\ell(G_M)\cong \CU_3(\ell)$.
    By \cite[Thm. 2.4]{Silverberg} and \cite[Thm. 1.2]{Goodman:cyclic}, $M$ is contained in $\QQ(A[\ell])$.
    Thus for cardinality reasons, we have, by Theorem \ref{thm:modellimagesarecontainedin}, that $\rho_\ell(\GQ)$ is as claimed.
\end{proof}

\section{Criteria for large images}
\label{section:criteriaforfullimage}

Let $A$ be an abelian threefold defined over $M$.
\cref{thm:containment_for_rho_lambda} shows that for an admissible prime $\ell$, the image of $G_M$ under $\rho_\lambda$ lands in $\CLe_3(\ell)$ where $\varepsilon = (\frac{\Delta_M}{\ell})$.
In this section we will develop a practical method for proving $\rho_\lambda(G_M)\cong \CLe_3(\ell)$. 

To prove $\rho_\lambda(G_M)\cong\CLe_3(\ell)$, we need to rule out the possibility that $\rho_\lambda(G_M)$ is contained in any maximal subgroup of $\CLe(A[\lambda])$.
Thus in \cref{section_maximal_sbgps} we list the maximal subgroups of $\CLe_3(\ell)$.
In \cref{subsec:ContainmentTests}, we determine consequences of $\rho_\lambda(G_M)$ being contained in each maximal subgroup, and provide algorithms to check if these consequences are fulfilled. If $\End^0(A_{\overline{M}}) \cong M$, we show under \cref{assumption:surjective} that these algorithms always work. This assumption appears to be known to experts, but not written down anywhere.

\begin{assumption}
\label{assumption:surjective}
    Let $A/M$ be an abelian threefold with $\End^0(A_{\overline{M}}) \cong M$.
    Then there exists a rational prime $\ell$ split in $M$ and $\lambda$ a prime of $M$ lying over $\ell$ such that there is an isomorphism $\rho_{\lambda}(\GM) \cong \GL_3(\ell)$.
    
    Moreover, $\ell$ can be chosen to be as large as needed, and satisfying the congruence condition 
    $\ell \equiv 1 \pmod e$ for any integer $e$.
\end{assumption}

\subsection{Maximal subgroups of  $\CLe_3(\ell)$}
\label{section_maximal_sbgps}
Recall that we write $\CL^+=\GL$ and $\CL^- = \CU$, which are the similitude groups of the zero pairing and a non-degenerate Hermitian pairing respectively.

The classification of the maximal subgroups of $\CL^{\pm}_3(\ell)$ may be deduced from the tables found on pages 378 and 379 of \cite{Bray_Holt_Roney-Dougal}.
Let us briefly describe how one does so.
On page 378 (resp. 379) the maximal subgroups of $\SL^+_3(\ell) \coloneqq  \SL_3(\ell)$  (resp. $\SL^-_3(\ell) \coloneqq \SU_3(\ell)$) are listed.
The maximal subgroups of $\SL^\varepsilon_3(\ell)$ which extend to $\GL^\varepsilon_3(\ell)$ are identified by the presence of the diagonal automorphism $\delta$ (viewed as an element of the outer automorphism group of $\PSL^{\varepsilon}_3(\ell)$, see \cite[pages 33,34]{Bray_Holt_Roney-Dougal}) in the `stab' column.
The table also notes the order of the outer automorphism $\delta$ to be $d = \gcd(3, \ell - \varepsilon)$. If $d = 1$, then $\delta$ is trivial and all rows satisfy the criterion.

To deduce the maximal subgroups of $\CU_3(\ell)$ from those of $\GU_3(\ell)$ note that any maximal subgroup of either group not containing the centre must contain $\SU_3(\ell)$ (see for example \cite[Lemma 3.15]{Goodman_superelliptic}).
As is a general fact in group theory, the subgroups of $\CU_3(\ell)$ (resp. $\GU_3(\ell)$) containing the centre correspond bijectively to those of their projectivisations via the quotient and inverse image maps.
Thus the maximal subgroups of $\CU_3(\ell)$ either contain $\SU_3(\ell)$ or arise from the maximal subgroups of $\GU_3(\ell)$ via adjoining the group of scalar matrices over $\FF_{\ell^2}^*$.

One finds that the maximal subgroups of $\GL^\varepsilon_3(\ell)$ not containing $\SL^\varepsilon_3(\ell)$ fall into one of six families.
Namely, $\C_1$ reducible subgroups, $\C_2$ imprimitive subgroups, $\C_3$ field extension subgroups, $\C_6$ symplectic type subgroups, $\C_8/\C_5$-type orthogonal subgroups, or $\mathcal{S}$-type subgroups.
The last two families occur only when the order $d=\gcd(3,\ell-\varepsilon)$ of the automorphism $\delta$ is $1$, i.e., $\ell \not\equiv \varepsilon \pmod 3$.
So it is worth pointing out that if $M=\Q(\zeta_3)$, then $\varepsilon = \left( \frac{\Delta_M}{\ell} \right) = \left( \frac{-3}{\ell} \right)$, and no maximal subgroups of $\GLe_3(\ell)$ come from the last two families.

We shall now proceed to describe these maximal subgroups.
Whilst doing so, we will make use of a few basic facts from the theory of the classical groups, all of which may be found in \cite{KleidmanLiebeck}.

\subsubsection{$\C_1$ reducible subgroups}
\label{section:reducible_sbgps}
The groups in this case are stabilisers of minimal chains of submodules.
That is, they stabilise a chain \[0 \subseteq U \subseteq V\]
where $U$ has dimension one or two.

For $\GL_3(\ell)$ this leads to two conjugacy classes of maximal subgroups in $\GL_3(\ell)$ isomorphic to $\Fl^2 \rtimes \left( \GL_1(\ell) \times \GL_2(\ell) \right)$.
We note that these conjugacy classes cannot be distinguished by the characteristic polynomials of their elements.
However, the inverse transpose map exchanges these two classes.
The description given in \S \ref{subsec:restrictions_on_residual_images} thus shows that if $\GM$ stabilises a 2-dimensional subspace of $A[\lambda]$, then it must stabilise a 1-dimensional subspace of $A[\overline{\lambda}]$.
In this way, we only need to rule out the existence of 1-dimensional subrepresentations to prove the irreducibility of the action of $\GM$ on $A[\lambda]$.

Suppose $H \leq \GU_3(\ell)$ stabilises an irreducible submodule $U$ of $V$.
Then either $U \cap U^\perp =0$ or $U$, i.e., either $U$ is non-degenerate or totally isotropic.
Suppose the former holds.
Then $U^\perp$ is also non-degenerate and $V = U \oplus U^\perp$ as $H$-modules.
It follows that $H$ is contained in $\GU_1(\ell) \times \GU_2(\ell)$.

Otherwise $U$ is totally isotropic, and as $V$ has dimension three, maximal totally isotropic subspaces have dimension one and thus $U$ is one dimensional.
It follows $U^\perp$ is two dimensional, and $H$ stabilises a chain $0 \subseteq U \subseteq U^{\perp} \subseteq V$.
An application of Witt's Lemma shows $U^\perp/U$ may be equipped with a non-degenerate Hermitian form.
In fact, using further properties of the pairing, one can show $H$ is contained in a group isomorphic to $\ell^{1+2} \rtimes \left(\GU_1(\ell) \times \GL_1(\ell^2) \right)$ where $\ell^{1+2}$ denotes a non-abelian group of order $\ell^3$.
The eigenvalues of an element in this maximal subgroup are of the form $\alpha, \alpha^{-\ell},\beta$ where $\alpha, \beta \in \FF_{\ell^2}^*$ and $\beta^{1+\ell}=1$.
In particular, up to exchanging $\alpha$ with $\alpha^{-\ell}$, one may assume the eigenvector corresponding to $\alpha$ gives the fixed subspace.

We note that in either situation the semisimplification $V^{ss}$ contains a one dimensional, non-degenerate constituent. There is only one conjugacy class for each of the above types of maximal subgroup in $\GU_3(\ell)$.
For $\ell = 2$, the group $\GU_1(2) \times \GU_2(2)$ is not maximal in $\GU_3(2)$.

\subsubsection{$\C_2$ imprimitive subgroups}
Here we have stabilisers of orthogonal decompositions $V= \oplus_{i=1}^3 V_i$, that is the groups falling into this case permute the elements of the set $\{V_1,V_2,V_3\}$.

This leads to subgroups of the form $\GL^\varepsilon_1(\ell)^3 \rtimes S_3$ which are maximal in $\GLe_3(\ell)$ unless $\varepsilon = +1$ and $\ell = 2$.
They constitute a unique conjugacy class.

\subsubsection{$\C_3$ field extension subgroups}
\label{section:fieldextnsbgps}
Viewing $V$ as a one dimensional vector space over $\FF_{\ell^{3}}$ for $\varepsilon=+$ (resp. over $\FF_{\ell^{6}}$ for $\varepsilon=-$) gives us an embedding of $\GLe_1(\ell^3)$ into $\GLe_3(\ell)$.
Its normaliser, isomorphic to $\GL_1(\ell^3) \rtimes \Gal(\FF_{\ell^3}/\Fl)$ for $\varepsilon=+1$, and $\GU_1(\ell^3) \rtimes \Gal(\FF_{\ell^6}/\FF_{\ell^2})$ for $\varepsilon=-1$, is a maximal subgroup of $\GLe_3(\ell)$, except for $\GU_3(2)$ and $\GU_3(3)$.
Groups arising in this manor give a single conjugacy class of maximal subgroups.

\subsubsection{$\C_6$ symplectic type subgroups}
This case only occurs for odd $\ell \equiv 4\varepsilon, 7\varepsilon \pmod{9}$.
Here we have a single conjugacy class of maximal subgroups.
The image of such a group in $\PGLe_3(\ell)$ is isomorphic to $C_3^2 \rtimes \SL_2(3) \cong \ASL_2(3)$ which has order $6^3$.

\subsubsection{$\C_8/\C_5$-type orthogonal subgroups}
The maximal subgroups in this class stabilise a non-degenerate symmetric bilinear pairing up to scalars. They are isomorphic to the direct product of the center of $\GLe_3(\ell)$ and $\SO_3(\ell)$.
In the notation of \cite{Bray_Holt_Roney-Dougal}, these subgroups are of type $\C_8$ if $\varepsilon = +1$, and of type $\C_5$ if $\varepsilon = -1$.
For $\ell > 2$, there is a single conjugacy class of such maximal subgroups when $\gcd(3,\ell-\varepsilon)=1$, i.e., when $\ell \not\equiv \varepsilon \pmod 3$.

\subsubsection{$\mathcal{S}$-type subgroups}
\label{subsec:ClassS_maximalgroup}
There is a family of $\mathcal{S}$-type maximal subgroups, whose image in $\PGLe_3(\ell)$ is isomorphic to $\PSL_2(7)$, which has order $168$.
For $\ell > 2$, they form a single conjugacy class exactly when $\ell \not\equiv \varepsilon \pmod{3}$ and $\left(\frac{-7}{\ell}\right) = \varepsilon$. For $\varepsilon = +1$, this is equivalent to $\ell \equiv 2, 8, 11 \pmod{21}$. For $\varepsilon = -1$, this is equivalent to $\ell \equiv 10, 13, 19 \pmod{21}$ or $\ell=3$.

\subsection{Containment in maximal subgroups}
\label{subsec:ContainmentTests}

From now on,
let $A/\QQ$ be a polarised abelian threefold with potential imaginary multiplication by $M$ of signature $(2,1)$,
and $\ell$ be an admissible prime for $A$.
While all algorithms and results in this section apply for IM abelian threefolds of signature $(2,1)$ defined over $M$, we have chosen to write them specifically for $A/\Q$, because it simplifies the discussion of the algorithms and we have carried out these computations only for such abelian threefolds.

Recall that for a prime $\lambda$ of $M$ above $\ell$, the $\lambda$-torsion $A[\lambda]$ is a three dimensional $\FF_\lambda$-vector space \cite[Prop 2.2.1]{ribet_RM}, and we know by \cref{thm:containment_for_rho_lambda}, that the image $\rho_\lambda(\GM)$ of the mod-$\lambda$ representation $\rho_\lambda \colon \GM \rightarrow \Aut_{\FF_\lambda}(A[\lambda])$ is contained in $\CLe_3(\ell)$ where $\varepsilon = (\frac{\Delta_M}{\ell})$.
We introduce the following terminology to simplify the discussion ahead.

\begin{definition}
\label{def:C1prime}
    We say an admissible prime $\ell$ is a \emph{$\C_1$-prime} or a \emph{reducible prime} for $A$, if $A[\lambda]$ is not 
    absolutely irreducible as a $\GM$-module for some prime $\lambda$ of $M$ above $\ell$, i.e., 
    $A[\lambda] \otimes_{\FF_\lambda} \overline{\FF}_\lambda$ has a proper $\overline{\FF}_\lambda[\GM]$-submodule.
    
    In particular, $\ell$ is a \emph{$\C_1$-prime} if $\rho_\lambda(G_M)$ is contained in a $\C_1$-type maximal subgroup of $\CLe_3(\ell)$.
    The converse is not true though, as $\rho_\lambda(G_M)$ not being contained in a $\C_1$-type maximal subgroup only implies $A[\lambda]$ is irreducible over $\FF_\lambda$, not absolutely irreducible.
\end{definition}

\begin{definition}
\label{def:C2C3primes}
    We say an admissible prime $\ell$ is a \emph{$\C_2$-prime} (respectively $\C_3$-prime or $\C_6$-prime) for $A$, if it is not a $\C_1$-prime 
    and $\rho_\lambda(G_M)$ is contained in a $\C_2$-type (respectively $\C_3$-type or $\C_6$-type) maximal subgroup of $\CLe_3(\ell)$
    for some prime $\lambda$ of $M$ above $\ell$.
\end{definition}

\begin{definition}
\label{def:C8C5primes}
    We say an admissible prime $\ell$ is a \emph{$\C_8/\C_5$-prime} (respectively $\mathcal{S}$-prime) for $A$, if $\rho_\lambda(G_M)$ is contained in a $\C_8/\C_5$-type (respectively $\mathcal{S}$-type) maximal subgroup of $\CLe_3(\ell)$ for some prime $\lambda$ of $M$ above $\ell$.
\end{definition}

The main algorithms in this section are \cref{Alg:C1,Alg:C2,Alg:C3,Alg:CO}.
If $\End^0(A_{\overline{M}}) \cong M$, these algorithms return finite sets of rational primes which are guaranteed to contain every $\C_1$, $\C_2, \C_3$ and $\C_8/\C_5$-prime respectively. 
The returned sets could contain extraneous primes, and the goal of \cref{Alg:C2WeedOut,Alg:C3WeedOut} is to remove any such extraneous 
$\C_2$ or $\C_3$-prime.
We also show that the only possible $\C_6$-primes are $5$ and $7$, and
the only possible $\mathcal{S}$-prime is $3$.

We start by setting some notation and recalling simple facts about $L$-polynomials and characteristic 
polynomials of Frobenius elements in our setting of abelian threefolds over $\Q$ with potential imaginary multiplication by $M$ of signature $(2,1)$. 

\subsubsection{Shape of $L$-polynomial factorisations}
\label{subsec:LPolyShape}

For a polynomial $h$, and a positive integer $n$, we denote by $h^{(n)}$ the polynomial 
whose roots are $n$-th powers of the roots of $h$, i.e., $h^{(n)}(x) = \Res_t(h(t),x-t^n)$.
Consider a rational prime $p$ of good reduction for the abelian threefold $A$. Let
\begin{align}
    \label{eq:LPoly}
    L_{p,A}(x) = 1 + a_p x + b_p x^2 + c_p x^3 + p b_p x^4 + p^2 a_p x^5 + p^3 x^6
\end{align}
denote the $L$-polynomial of $A$ at $p$. Suppose $p$ is unramified in $M$ and let $\ell \ne p$ be any prime.
We define 
$f_p$ to be the characteristic polynomial of $\Frob_p$ (or $\Frob_p^2$) 
acting on $T_{\ell}(A)$ if $p$ splits (or is inert) in $M$.
So in either case, $f_p$ is the characteristic polynomial under $\rho_{\ell^\infty}$ of an element of $\GM$.
Then we know that 
$f_p$ is independent of $\ell$, and furthermore
\begin{align}
    \label{eq:charpolofFrobaction}
    f_p(x) = \begin{cases}
        x^6L_{p,A}(1/x) \quad \text{if } p \text{ splits in } M\\
        x^6L_{p,A}^{(2)}(1/x) \quad \text{if } p \text{ is inert in } M.
    \end{cases}
\end{align}
By taking the resultant of $L_{p,A}(t)$ and $x-t^2$ with respect to $t$, we note that
\begin{align*}
    L_{p,C}^{(2)}(x) = 1 - (a_p^2 - 2b_p)x + \cdots - (a_p^2 - 2b_p)p^4x^5 + p^6x^6.
\end{align*}
Since the Galois action of $\GM$ commutes with the action of 
$\iota(M) \otimes \QQ_\ell$
on $T_\ell(A)$, and we may take $\ell$ to split in $M$, i.e, $\iota(M) \otimes \QQ_\ell \simeq \QQ_\ell \times \QQ_\ell$, 
we deduce that $f_p$ factorises over $\OM$ as a product $f_{p,1}f_{p,2}$ 
of cubic polynomials $f_{p,i} = x^3 + a_{p,i} x^2 + b_{p,i} x + c_{p,i}$
that are Galois conjugate, i.e., $f_{p,2} = \overline{f_{p,1}}$.

\begin{remark}
    \label{rem:inertLPolyshape}
    For a prime $p$ that is inert in $M$, $L_{p,A}(x)$ is a polynomial in $x^2$, and 
    hence $f_p(x) = x^6L_{p,A}^{(2)}(1/x)$ is the square of a cubic polynomial 
    $f_{p,1}=f_{p,2} \in \Z[x]$ such that $x^6L_{p,A}(1/x) = f_{p,1}(x^2)$. 
    Furthermore $f_{p,1}(x)$ is reducible and has a linear factor $x+p$ 
    (see \cite[Lemma 3.8]{AFP22} for the case of Jacobians of Picard curves, and \cite[Lemma 4.3]{FG25}
    for the general case).
\end{remark}

\subsubsection{$\C_1$ reducible subgroups}
\label{subsec:C1test}

Our approach to ruling out $\C_1$-primes is inspired by Billerey's method for proving the reducibility of the $\ell$-torsion representations of
elliptic curves over number fields \cite{Billerey_irreductible}.
Forthcoming work \cite{Goo26_1dimsubreps}
of the second-named author further generalises this to rule out $1$-dimensional factors from residual representations 
of a large class of strictly compatible systems of Galois representations.

Suppose that $\ell$ is a $\C_1$-prime. Let $\lambda$ be a prime of $M$ above $\ell$ with inertia degree $f$. 
Let $\FF_\lambda$ denote the corresponding residue field, so that $f = [\FF_\lambda : \FF_\ell]$. 
As in \cref{subsection:inertia_at_ell}, let $\theta_n$ denote a fixed fundamental character of level $n$. 
As $A[\lambda] \otimes_{\FF_\lambda} \overline{\FF}_\lambda$ is a reducible $\overline{\FF}_\lambda[\GM]$-module of dimension $3$, 
the dimensions of its Jordan--H\"older factors are either $(1,1,1)$ or $(1,2)$, so one of the 
Jordan--H\"older factors is one dimensional. Let us call the corresponding character $\tau \colon \GM \rightarrow \overline{\FF}_\lambda^*$.

\begin{lemma}
\label{lem:C1local}
    Suppose $\tau \colon \GM \rightarrow \overline{\FF}_\lambda^*$ is a $1$-dimensional Jordan--H\"older factor of $A[\lambda]$ for a prime
    $\lambda$ of $M$ above $\ell$.
    Then there exist $a,b \in \{0,1\}$ such that 
    \begin{enumerate}
        \item $\tau|_{I_\lambda} = \theta_1^a$ and $\tau|_{I_{\lambdabar}} = \theta_1^b$ if $\ell$ is split in $M$,
        \item $\tau|_{I_\lambda} = \theta_2^{a+b\ell}$ if $\ell$ is inert in $M$.
    \end{enumerate}
    In case (1), the local characters $\calO_{M_\lambda}^*,\calO_{M_{\lambdabar}} ^* \rightarrow \overline{\FF}_\lambda^*$
    associated to $\tau|_{I_\lambda}$ and $\tau|_{I_{\lambdabar}}$ by local class field theory are
    $x \mapsto x^{-a} \pmod \lambda$ and $x \mapsto x^{-b} \pmod \lambdabar = \bar x^{-b} \pmod{\lambda}$ respectively.
    In case (2), the associated local character is $x \mapsto x^{-a}\bar x^{-b} \pmod{\lambda}$.
\end{lemma}

\begin{proof}
    Under local class field theory, the decomposition group at $\lambda$ corresponds to the profinite completion of 
    $M_\lambda^*$, and the inertia subgroup at $\lambda$ corresponds to $\OO_{M_\lambda}^*$.
    Since $\overline{\FF}_\lambda^*$ has no element of order $\ell$, the character $\tau$ is tamely ramified at $\lambda$, 
    and hence $\tau|_{I_\lambda}$ factors through $\OO_{M_\lambda}^*/1+\lambda \OO_{M_\lambda} \simeq \FF_\lambda^*$.
    As the residual degree of $\ell$ is at most 2, $\tau|_{I_\lambda}$ factors through $\theta_2$, that is, $\tau|_{I_\lambda}= \theta_2^{m}$ for some $m$.
    Applying \cref{prop:inertia_at_l_description}, we deduce the first statement for $\tau|_{I_\lambda}$. Replacing $\lambda$ by $\lambdabar$ proves the first statement for $\tau|_{I_{\lambdabar}}$.

    The last assertion follows from \cite[Chap 5, Prop 3.1 and 3.4]{Neu99} which say that
    $\theta_f(\Art_\lambda(x)) = x^{-1}$ for $x \in \OO_{M_\lambda}^*$, where $x$ is identified with its
    image in $\FF_\lambda$, under the normalisation that the local Artin map sends uniformisers to
    arithmetic Frobenius.
\end{proof}

Let $T$ denote the finite set of non-admissible primes for $A$. For a prime $q \geq 3, q \not\in T$,
let $e_q$ denote the exponent of the group $\CLe_3(q)$, where as usual
$\varepsilon = (\frac{\Delta_M}{q})$.
Set
\begin{align}
    \label{eq:semistablereductionover}
    n_M = \lcm\limits_{p \text{ prime}} \ \gcd\limits_{\substack{q \geq 3 \text{ prime} \\q \neq p, q \not\in T}} e_q.
\end{align}

\begin{remark}
    Note that even though this definition depends on $A$ through 
    the finite set of excluded primes $T$, the quantity $n_M$ only depends on the 
    field $M$. Excluding any finite set of primes from the $\gcd$ term does not 
    affect its value, and a simpler alternative definition is
    $n_M = \gcd\limits_{q \geq 3,~q \nmid \disc(M)} e_q$.
    The main cases of our computational interest here are $M = \QQ(i), \QQ(\sqrt{-2})$
    and $\QQ(\sqrt{-3})$, for which, the values of $n_M$ are $24, 24$ and $72$ respectively.
\end{remark}

\begin{lemma}
\label{lem:C1ram}
    Let $\tau \colon \GM \rightarrow \overline{\FF}_\lambda^*$ be a $1$-dimensional Jordan--H\"older factor of $A[\lambda]$ for a prime
    $\lambda$ of $M$ above $\ell$. Then $\tau^{n_M}$ is unramified at all primes $\p \nmid \ell$ of $M$.
\end{lemma}
\begin{proof}
    Let $\p \nmid \ell$ be a prime of $M$ with residue characteristic $p$.
    Let $q \geq 3, q \neq p, q \not\in T$ be a prime.
    Raynaud's theorem \cite[Prop. 4.7]{SGA7} says that $A$ attains semistable reduction over $M(A[q])$ at any prime not above $q$.
    By \cref{thm:modellimagesarecontainedin}, we know that the Galois group
    $\Gal(M(A[q])/M)$ is a subgroup of $\FF_q^* \times \CLe_3(q)$ or $\CLe_3(q)$, and hence its exponent 
    divides $e_q$.
    Let us fix an embedding $\Qbar \hookrightarrow \Qbar_\p$. Let $w$ 
    denote the prime of $M(A[q])$ above $\p$ induced by the chosen embedding.
    Then the inertia group $I_w$ at $w$ is $I_\p \cap G_{M(A[q])}$.
    So for any $\sigma \in I_\p$, we have $\sigma^{e_q} \in I_\p \cap G_{M(A[q])} = I_w$.
    Since $A$ has semistable reduction at $w$ and $w \nmid \ell$, 
    Grothendieck's inertial criterion \cite[Prop. 3.5]{SGA7} implies that 
    $I_w$ acts unipotently on the Tate module $T_\ell(A)$, and hence on 
    $A[\ell]$. Since $\tau$ is a $1$-dimensional subquotient, we deduce that 
    $\tau|_{I_w} = 1$ and hence $\tau^{e_q}|_{I_\p} = 1$. Therefore the order of $\tau|_{I_\p}$ divides the $\gcd$ term in 
    \cref{eq:semistablereductionover} and so it divides $n_M$.
\end{proof}

\begin{proposition}
    \label{prop:C1test}
    Let $\ell$ be a semistable admissible prime and $\lambda$ be a prime of $\OO_M$ above $\ell$.
    Suppose that the action of $G_M$ on $A[\lambda]$ is not absolutely irreducible, and let
    $\tau \colon \GM \rightarrow \overline{\FF}_\lambda^*$ be a $1$-dimensional Jordan--H\"older factor of $A[\lambda]$.
    Let $\p$ be a prime of $M$ of residue characteristic $p \neq \ell$, such that $A$ has good reduction
    at $\p$. Suppose $\p^h = (\gamma)$ for some $\gamma \in \OM$ and $h \geq 1$. Then there exists 
    $\alpha \in \{1, \gamma, \bar \gamma, \gamma \bar \gamma\}$
    such that
    \begin{align}
        \tau^{n_Mh}(\Frob_\p) \equiv \alpha^{n_M} \pmod \lambda
    \end{align}
\end{proposition}
\begin{proof}
    Let $\psi = \tau^{n_M}$.
    Since the global Artin reciprocity map is trivial on principal ideles, compatibility of the global and local 
    Artin maps gives $\prod_v \psi(\Art_v(\gamma)) = \psi(\Art(\gamma)) = 1$, where the product ranges over all 
    places $v$ of $M$ and $\Art_v$ is the local Artin reciprocity map at $v$. 
    Using the facts that the Artin map at the complex infinite place is trivial, $\ord_v(\gamma) = 0$ for all 
    $v \ne \p$, and $\psi$ is unramified outside $\ell$ by \cref{lem:C1ram}, 
    this equality simplifies to:
    \begin{align}
        \label{eq:artinreciprocitylocalglobal}
        \psi(\Art_\p(\gamma)) \prod\limits_{v|\ell} \psi(\Art_v(\gamma)) = 1
    \end{align}
    The first term is $\psi(\Art_\p(\gamma)) = \psi(\Frob_\p^h) = \tau^{n_Mh}(\Frob_\p)$. By \cref{lem:C1local},
    we get that $\prod\limits_{v|\ell} \tau(\Art_v(\gamma)) \equiv \gamma^{-a} \bar \gamma^{-b} \pmod \lambda$ for some $a,b \in \{0,1\}$. Raising the last congruence to the $n_M$-th power, and putting them back into 
    \cref{eq:artinreciprocitylocalglobal}, gives $\tau^{n_Mh}(\Frob_\p) \equiv (\gamma^a \bar \gamma^b)^{n_M} \pmod \lambda$, 
    which is the desired result.
\end{proof}

\cref{prop:C1test} readily yields the following algorithm for detecting all possibly $\C_1$-primes for an abelian threefold $A$, given sufficiently many $L$-polynomials for $A$.

\begin{algorithm}[Possibly $\C_1$-primes]
    \label{Alg:C1}
    \label{Alg:red}
    \hfill\\
    {\em Input:}
    \begin{itemize}
        \item an integer $N$ that is a multiple of all primes of bad reduction for a 
        polarised abelian threefold $A/\Q$.
        \item an imaginary quadratic field $M$ such that $A$ has potential imaginary multiplication by $M$ of signature $(2,1)$.
        \item $L$-polynomials of $A/\Q$ at each prime in the set $S$ of all primes less than $B > 0$, that split in $M$ and do not divide $N$.
    \end{itemize}
    {\em Output:} a finite superset of the semistable $\C_1$-primes, or $0$.
    \begin{enumerate}[leftmargin=!]
        \item \label{step:AlgC1ComputeMultipleOfn_M} Compute a positive integer $e$ divisible by $n_M$ 
        using \cref{eq:semistablereductionover}.
        \item \label{step:AlgC1Polys} For each $p \in S$, choose a prime $\p$ of $\OM$ above $p$. 
        Let $h_p$ be the order of $\p$ in the class group of $M$, and $\gamma_p$ a generator 
        of $\p^{h_p}$. Calculate $f_p(x)=x^6L_{p,A}(1/x)$ as in \cref{eq:charpolofFrobaction} from the given $L$-polynomial,
        and use resultants to compute $f_p^{(eh_p)}(x)$.
        Evaluate
        at $x=1,p^{eh_p},\gamma_p^e$ 
        to compute
        \begin{align*}
            R_p = f_p^{(eh_p)}(1) \cdot f_p^{(eh_p)}(p^{eh_p}) \cdot \Nm_{M/\QQ}\left(f_p^{(eh_p)}(\gamma_p^{e})\right).
        \end{align*}
        \item \label{step:AlgC1Return} Calculate $R = \gcd_{p \in S}\left(p R_p\right)$.
        If $R=0$, return $0$. Otherwise return the set of prime numbers dividing $R$.
    \end{enumerate}
\end{algorithm}

\begin{theorem}
    \label{thm:AlgC1works}
    Let $A/\Q$ be a polarised abelian threefold, and
    $M$ be an imaginary quadratic field 
    such that $M \cong \End^0(A_{\Qbar})$.
    Suppose \cref{assumption:surjective} holds.
    Then, for $B$ large enough, \Cref{Alg:C1} returns a finite superset of the semistable $\C_1$-primes for $A$.
\end{theorem}
\begin{proof}
    We work with the notation in \cref{Alg:C1}.
    To show that the algorithm does not return $0$ for $B$ large enough, we need to 
    produce a single prime $p$ with $R_p \neq 0$.
    Let $h_M$ denote the class number of $M$.
    By \cref{assumption:surjective}, 
    choose a rational prime $\ell_0 \equiv 1 \pmod e$ splitting in $M$, and large enough
    so that $\ell_0^2+\ell_0+1 > h_M$, and $\rho_{\lambda_0} \colon \GM \rightarrow \GL_3(\FF_{\ell_0})$ 
    is surjective for $\lambda_0$ above $\ell_0$.
    Consider a generator $g$ of $\FF_{\ell_0^3}^*$, and let $g$ also denote its image under 
    a fixed natural embedding $\FF_{\ell_0^3}^* \hookrightarrow \GL_3(\FF_{\ell_0})$.
    By applying the Chebotarev density theorem to $M(A[\lambda_0])/M$, we know that there 
    exists a prime $\p$ of $M$, lying over a rational prime $p \neq \ell_0$, which splits in $M$ and does not divide $N$, such that $\rho_{\lambda_0}(\Frob_\p)$ is conjugate to $g$.
    
    Fix a place of $\Qbar$ above $\lambda_0$.
    We claim that no root of $f_p$ has a reduction whose $h_p$-th power is in $\FF_{\ell_0}$.
    Since $f_p \pmod{\lambda_0}$ is the product of the characteristic polynomials of $\rho_{\lambda_0}(\Frob_\p)$ and of $\rho_{\Bar\lambda_0}(\Frob_\p)$, its roots 
    are $\{ g^{\ell_0^i}, p g^{-\ell_0^i} | 0 \leq i \leq 2\}$. Since the smallest power of $g$ to lie in $\FF_{\ell_0}$ is the $(\ell_0^2+\ell_0+1)$-th power, the same is true for all of these elements. Since $h_p \leq h_M < \ell_0^2+\ell_0+1$, the claim follows.

    Assume on the contrary that $R_p = 0$. Since every root of $f_p$ has complex absolute value $\sqrt{p}$ because of the Weil conjectures, $R_p$ can be zero only if 
    $f_p^{(eh_p)}(\gamma_p^{e}) = 0$. That is, there is a root $\alpha$ of $f_p$ such that
    $\alpha^{eh_p} = \gamma_p^e$. Thus $\alpha^{h_p} = \zeta \gamma_p$ for some root of unity $\zeta$ 
    with $\zeta^e=1$. Consider reduction at the chosen place above $\lambda_0$. Since 
    $\lambda_0$ has residue degree $1$, and $e | \ell_0-1$, we know that the reductions of
    both $\gamma_p$ and $\zeta$ lie in $\FF_{\ell_0}$. This implies that the reduction of 
    $\alpha^{h_p}$ also lies in $\FF_{\ell_0}$, contradicting the claim just proven. Therefore 
    $R_p \neq 0$.
    So, for large enough $B$, the algorithm returns a finite set of primes.

    Let $\ell$ be a semistable $\C_1$-prime for $A$. So $A[\lambda]$ is not absolutely irreducible for 
    some prime $\lambda$ above $\ell$, and let $\tau$ be a $1$-dimensional Jordan--H\"older factor of $A[\lambda]$.
    Let $p \in S$. If $p = \ell$, then $\ell$ clearly divides $pR_p$. If $p \neq \ell$,
    then $\tau(\Frob_\p)$ is an eigenvalue of $\rho_\lambda(\Frob_\p)$, thus 
    $f_p(\tau(\Frob_\p)) \equiv 0 \pmod \lambda$. \cref{prop:C1test} says that  
    $\tau^{eh_p}(\Frob_\p) \equiv \alpha^{e} \pmod \lambda$ for some 
    $\alpha \in \{1,\gamma_p,\Bar\gamma_p,\gamma_p\Bar\gamma_p = \Nm(\p)^{h_p} = p^{h_p}\}$.
    Hence $f_p^{(eh_p)}(\alpha^{e}) \equiv 0 \pmod \lambda$, which shows that $\ell | R_p$.
    So $\ell$ divides $R$.
\end{proof}

\subsubsection{$\C_2$ imprimitive subgroups}
\label{subsec:C2test}

We start with the following lemma about characteristic polynomials of certain elements in $\C_2$-type maximal subgroups.  
\begin{lemma}
    \label{lem:imprimitivity_min_polys}
    Let $V$ be a three dimensional vector space over a finite field $\FF$ of characteristic $\ell \geq 5$.
    Suppose $\sigma \in \GL(V)$ preserves a decomposition $V = \langle v_1 \rangle \oplus \langle v_2 \rangle \oplus \langle v_3 \rangle$.
    Write $D(V)= \{\langle v_1 \rangle , \langle v_2 \rangle , \langle v_3 \rangle\}$.
    The following hold:
    \begin{enumerate}
        \item \label{item:order3elementcharpol} if the action of $\sigma$ on $D(V)$ has order 3, then its characteristic polynomial equals $x^3+a$ for some $a \in \FF^*$;
        \item \label{item:order2elementcharpol} if the action of $\sigma$ on $D(V)$ has order 2, then its characteristic polynomial equals $(x+a)(x^2+b)$ for some $a,b \in \FF^*$.
    \end{enumerate}
    Moreover if $\sigma$ acts non-trivially on $D(V)$ and its minimal polynomial  equals 
    \begin{enumerate}[(a)]
        \item either $(x-1)(x-d)$ for some $  d \in \FF^*$;
        \item or $(x-1)(x-d)(x-\gamma)$ for some $d \in \FF_\ell^*$ and $\gamma  \in \FF \setminus  \FF_\ell$;
    \end{enumerate}
     then $\sigma$ acts as a 2-cycle on $D(V)$ and $d=-1$.
\end{lemma}

\begin{proof}
    The group of elements which preserve $D(V)$ is isomorphic to $\GL_1(\FF)^3 \rtimes S_3$. As $\ell \geq 5$, the elements in $\GL_1(\FF)^3 \rtimes S_3$ act semisimply by Maschke's Theorem, so their minimal polynomials are squarefree.
    In particular, $d \neq 1$.
    
    Let us suppose $\sigma \in \GL(V)$ acts on $D(V)$ as a 3-cycle.
    Then there exists $a_1, a_2,a_3 \in \FF^*$ such that $\sigma v_{i} = a_{i}v_{i+1} $  where the indices are taken modulo 3.
    In particular, $\sigma$ has trace zero and the minimal polynomial of $\sigma^3$ is  $t-a_1a_2a_3$, that is, $\sigma^3$ is a scalar matrix.
    Moreover, as $(1,2,3)$ is an even permutation, $\det(\sigma)=a_1a_2a_3$.
    
    Let us now show that the minimal polynomial of $\sigma$ cannot be as in $(a)$ or $(b)$.
    We shall suppose for a contradiction that it is.
    Then $\sigma$ has an eigenvalue equal to 1.
    As $\sigma^3$ is a scalar matrix, we find that $a_1a_2a_3 = 1$.
    Thus $\det(\sigma) = 1$ and $\sigma$ has order 3.
    However, this is incompatible with the minimal polynomials in $(a)$ and $(b)$.
    Indeed, if the minimal polynomial of $\sigma$ was as in $(a)$, the determinant of $\sigma$ would either equal $d$ or $d^2$, both of which are not equal to 1 since $d$ is forced to have order $3$ in $\FF^*$.
    The minimal polynomial cannot be as in $(b)$ either, since in this case the determinant is $d \gamma$, which does not belong to $\FF_\ell$.

    Let us now suppose $\sigma \in \GL(V)$ acts on $D(V)$ as a 2-cycle.
    Without loss of generality, we may suppose $\sigma$ fixes $\langle v_3\rangle$.
    Then there exists $a_1, a_2,a_3 \in \FF^*$  such that $\sigma v_{1} = a_{1}v_{2},\sigma v_{2} = a_{2}v_{1},\sigma v_{3} = a_{3}v_{3}$.
    We immediately verify that the characteristic polynomial of $\sigma$ is $x^3-a_{3}x^2-a_1a_2x +a_1a_2a_3$, which factors as $(x-a_{3})(x^2-a_1a_2)$.

    Let us now show that if the minimal polynomial of $\sigma$ is as given in $(a)$ or $(b)$, then $d=-1$.
    As the trace of $\sigma$ equals $a_3$ and $a_3$ is an eigenvalue of $\sigma$, we see that the eigenvalues of $\sigma$ are of the form $c,-c, a_3$ for some $c$.
    Since $-\gamma$ is not an eigenvalue of $\sigma$, we deduce $\{1,d\}=\{\pm c\}$.
\end{proof}

We now describe consequences of $\rho_\lambda(G_M)$ being contained in a $\C_2$-type maximal subgroup, using \cref{lem:imprimitivity_min_polys}, and the characterisation in \cref{prop:inertia_at_l_description} of the action of inertia at $\ell$ on $A[\lambda]$.

\begin{proposition}
    \label{prop:C2test}
    Let $\lambda$ be a prime of $\OO_M$ lying above an admissible prime $\ell$ for $A$.
    Suppose that $\rho_\lambda(G_M)$ acts absolutely irreducibly on $A[\lambda]$ and is contained in a $\C_2$-type maximal subgroup of
    $\CL^\varepsilon_3(A[\lambda])$.\\
    Then there exists a non-trivial abelian extension $L/M$ of degree $\leq 3$ unramified outside $\ell$ and the primes of bad reduction for $A$, such that
    \begin{itemize}
        \item if $[L \colon M]=2$ and the image of $\sigma \in \GM$ in $\Gal(L/M)$ is non-trivial then the minimal polynomial of $\rho_\lambda(\sigma)$ equals $x^3+ax^2+bx+ab$ for some $a,b \in \FF_\lambda^*$;
        \item if $[L \colon M]=3$ and the image of $\sigma \in \GM$ in $\Gal(L/M)$ is non-trivial then $\Tr \rho_\lambda(\sigma) \equiv 0 \pmod{\lambda}$.
    \end{itemize}
    If moreover, $A$ has semistable reduction at every prime lying above $\ell$ and $\ell \geq 5$, then $L/M$ is unramified at $\ell$.
\end{proposition}

\begin{proof}
    Being contained in a $\C_2$-type subgroup means $\rho_\lambda(G_M)$ stabilises a decomposition $A[\lambda] = \bigoplus^{3}_{i=1} \langle v_i \rangle$.
    Moreover, as $\rho_\lambda(G_M)$ acts irreducibly on $A[\lambda]$, its action on $D(\lambda) \coloneqq \{\langle v_1 \rangle , \langle v_2 \rangle , \langle v_3 \rangle\}$ contains a 3-cycle.
    The extension $L'$ of $M$ cut out by the kernel of the action of $\rho_\lambda(G_M)$ on $D(\lambda)$ has Galois group either $C_3$ or $S_3$.
    If $\Gal(L'/M) \cong C_3$ we write $L=L'$.
    If $\Gal(L'/M) \cong S_3$,  we let $L$ denote the unique quadratic extension of $M$ contained in $L'$.
    Applying the first part of Lemma \ref{lem:imprimitivity_min_polys}, we see that it only remains to prove the final statement.

    As $\rho_\lambda(G_M)$ is contained in a $\C_2$-type subgroup, its order is coprime to $\ell$ and thus acts semisimply by Maschke's Theorem.
    In particular, the restriction of $\rho_\lambda$ to an inertia group $I$ of a prime above $\ell$ is fully determined by Proposition \ref{prop:inertia_at_l_description}.
    Let us suppose for a contradiction that some $\sigma \in I$ acts non-trivially on $D(\lambda)$.
    We shall first consider the case that $\ell$ splits in $M$ and then afterwards that it is inert in $M$.

    If $\ell$ splits in $M$ then  $\rho_\lambda|_{I} = 2 \mathbbm{1} + \theta_1$ or $ \mathbbm{1} + 2\theta_1$.
    Indeed, the other possibilities described in \cref{prop:inertia_at_l_description} lead to $\rho_\lambda(I)$ containing a cyclic subgroup of order $\geq \ell^2-1$, which is incompatible with the fact that the largest cyclic subgroup of the $\C_2$-type maximal subgroup of $\GL_3(\ell)$ has order $3(\ell-1)$.
    Thus lemma \ref{lem:imprimitivity_min_polys} implies $\theta_1(\sigma)=-1$.
    But $\theta_1$ surjects onto $\FF_\ell^*$ which has order $\ell -1 \geq 4$.
    Thus $L/M$ is unramified at $\ell$ if $\ell$ splits in $M$.

    If $\ell$ is inert in $M$ then 
    $\rho_\lambda|_I = 2 \theta_2 + \theta_2^\ell$ or $\mathbbm{1} + \theta_1 + \theta_2$ for some fundamental character $\theta_2$ of level $2$.
    Indeed, the other possibilities described in \cref{prop:inertia_at_l_description} lead to $\rho_\lambda(I)$ containing a cyclic subgroup of order $\geq (\ell-1)(\ell^3+1)$, which is incompatible with the fact that the largest cyclic subgroup of the $\C_2$-type maximal subgroup of $\CU_3(\ell)$ has order $3(\ell^2-1)$.
    If $\rho_\lambda|_I = \mathbbm{1} + \theta_1 + \theta_2$, \cref{lem:imprimitivity_min_polys} implies that the image of $\theta_1$ is contained in $\{\pm 1\}$, contradicting $\ell\geq 5$.
    Thus suppose $\rho_\lambda|_I = 2 \theta_2 + \theta_2^\ell$.
    We may write $\rho_\lambda(\sigma) = \theta_2(\sigma) \sigma'$ for some $\sigma' \in \GL(A[\lambda])$.
    The action of $\sigma'$ on $D(\lambda)$ coincides with that of $\sigma$.
    As the minimal polynomial of $\sigma'$ is $(x-1)(x-\theta_2^{\ell-1}(\sigma))$, applying Lemma \ref{lem:imprimitivity_min_polys} we deduce $\theta_2^{\ell-1}(\sigma)=-1$.
    However, this contradicts the fact that $\theta_2$ surjects onto $\FF_{\ell^2}^*$ which has order $\ell^2-1 \geq 24$.
\end{proof}

\begin{remark}
\label{rem:ConductorExponentBound}
    Let us now bound the possible conductors of such quadratic or cyclic cubic extensions $L/M$. Since $L \subseteq M(A[\lambda])$, we know that this conductor divides the conductor of $A$. But we can say more. 
    Let $\p$ be a prime ramifying in the extension $L/M$ with residual characteristic $p$.
    If $p$ does not divide $[L:M]$, then $\p$ is tamely ramified and hence the exponent of $\p$ in the conductor of $L/M$ is $1$.
    If $p$ divides $[L:M]$, then $p = [L:M] \in \{2,3\}$ and the exponent of $\p$ in the conductor of $L/M$ is at most $1+\frac{p}{p-1}\ord_{\p}(p) \leq 5$, by \cite[Corollaire pg 215]{Serre_modularity_conjecture}.
\end{remark}

We now adapt \cref{prop:C2test} to turn it into an algorithm for detecting all possibly $\C_2$-primes for an abelian threefold $A$. The first step is to enumerate, using class field theory and \cref{rem:ConductorExponentBound}, the finite set of quadratic or cyclic cubic extensions $L'/M$, that are unramified outside the set of primes of semistable reduction for $A$. The field extension $L/M$ appearing in \cref{prop:C2test} must be one of these $L'$. So, for each of these extensions $L'$, and for sufficiently many primes $p$ inert in $L'$, if the $L$-polynomial of $A$ at $p$ is known, the congruences of \cref{prop:C2test} will lead to non-trivial divisibility conditions on the $\C_2$-primes. In \cref{Alg:C2,Alg:C3}, we try to optimize 
the number of primes to be considered for ensuring the existence of one that is inert in any possible extension $L'/M$.
We do this by working with a basis of homomorphisms from $G_M$ to $\FF_2$ or $\FF_3$, and making use of linear algebra.

\begin{algorithm}[Possibly $\C_2$-primes]
    \label{Alg:C2}
    \hfill\\
    {\em Input:} 
    \begin{itemize}
        \item an integer $N$ that is a multiple of all primes of bad reduction for a polarised abelian threefold $A/\Q$.
        \item an imaginary quadratic field $M$ such that $A$ has potential imaginary multiplication by $M$ of signature $(2,1)$.
        \item $L$-polynomials of $A/\Q$ at each prime $p$, not dividing $N$, less than some bound $B$.
    \end{itemize}
    {\em Output:} a finite superset of the $\C_2$-primes, or $0$.
    \begin{enumerate}[leftmargin=!]
        \item \label{step:AlgC2HomsToF2} Use the conductor exponent bounds from \cref{rem:ConductorExponentBound} to compute a basis $\{\chi_j : 1 \leq j \leq d\}$ for the $\FF_2$-vector space $H$ of Hecke characters of $M$ unramified outside $N$, and having order dividing $2$. We treat these characters as homomorphisms $G_M \rightarrow \FF_2$.
        \item \label{step:AlgC2ComputeLPolys} For each prime $p < B, p \nmid N$, use \cref{eq:charpolofFrobaction} to construct the polynomial $f_p(x) \in \Z[x]$ from the $L$-polynomial.
        Write $f_p(x) = f_{p,1}f_{p,2}$ for certain cubic polynomials $f_{p,i}(x) = x^3 + a_{p,i} x^2 + b_{p,i} x + c_{p,i} \in \OM[x]$, with $\overline{f_{p,1}} = f_{p,2}$, as follows:
        \begin{itemize}
            \item If $p$ splits in $M$, factorise $f_p(x)$ over $\OM$. If it has exactly two irreducible factors of degree $3$, call them $f_{p,1}(x)$ and $f_{p,2}(x)$. If not, move to the next prime.
            \item If $p$ is inert in $M$, let $f_{p,1}$ be the cubic polynomial in $\Z[x]$ such that $f_p(x) = f_{p,1}(x)^2$ and let $f_{p,2} = f_{p,1}$.
        \end{itemize}
        \item \label{step:AlgC2ConstructMatrix} 
        Let $S$ denote the set of primes $p$ such that $v_p \coloneqq a_{p,1} b_{p,1} - c_{p,1} \ne 0$.
        Let $S = \{p_i : 1 \leq i \leq m\}$. For each prime $p_i \in S$, choose a prime ideal $\frakp_i$ above $p_i$. 
        Construct the $m \times d$ matrix $A \in M_{m,d}(\FF_2)$ 
        whose $(i,j)^{th}$ entry is $\chi_j(\frakp_i) = \chi_j(\Frob_{\frakp_i}) \in \FF_2$.
        \item \label{step:AlgC2NotFullRankRepeat} If $\rank(A) < d$,
        return $0$.
        \item \label{step:AlgC2Return} Return the union of the set of all rational primes outputted by Algorithm \ref{Alg:C3}, and the set of rational primes dividing
        \begin{align*}
            6N\prod\limits_{p \in S} \Nm_{M/\Q}(v_p).
        \end{align*}
    \end{enumerate}
\end{algorithm}

\begin{theorem}
    \label{thm:AlgC2works}
    Let $A/\Q$ be a polarised abelian threefold, and
    $M$ be an imaginary quadratic field 
    such that $M \cong \End^0(A_{\Qbar})$.
    Suppose \cref{assumption:surjective} holds.
    Then, for $B$ large enough, \Cref{Alg:C2} returns a finite superset of the $\C_2$-primes for $A$.
\end{theorem}
\begin{proof}
    By \cref{assumption:surjective}, choose a rational prime $\ell$ splitting in $M$, and $\lambda$ above $\ell$ so that $\rho_\lambda(G_M) = \GL_3(\ell)$.
    Let $M(A[\lambda])$ denote the fixed field of $\ker(\rho_\lambda)$.
    It contains a unique quadratic extension of $M$, which corresponds to the quotient $\GL_3(\ell) \xrightarrow{\det} \FF_\ell^* \twoheadrightarrow \FF_\ell^*/\FF_\ell^{* 2}$.
    
    Let $I$ be the inertia group at $\lambda$. By \cref{prop:inertia_at_l_description}, we know that $\det \rho_\lambda|_I = \theta_1$ or $\theta_1^2$. If it is the latter, we deduce from $\det \rho_\lambda \det \rho_{\lambdabar} = \det \rho_\ell = \chi_\ell^3$ that $\det \rho_{\lambdabar}|_I = \theta_1$. 
    By the last part of \cref{thm:modellimagesarecontainedin}, we know that $\rho_{\lambdabar} = ^\sigma \!\! \rho_\lambda$ where $\sigma \in G_\Q \setminus G_M$.
    Therefore $\theta_1 = \det ^\sigma \!\! \rho_{\lambda}|_I = \det \rho_{\lambda}|_{\sigma I \sigma^{-1}}$.
    Hence the restriction of $\det \rho_\lambda$ to the inertia group at $\lambdabar$ is $\theta_1$.
    Hence this quadratic extension is ramified at either $\lambda$ or $\lambdabar$.
    
    Let $M_2$ denote the compositum of the quadratic extensions cut out by the order $2$ Hecke characters $\chi_j$ from \cref{step:AlgC2HomsToF2} of \cref{Alg:C2}.
    Since $M_2$ is unramified outside $N$,
    it is disjoint from $M(A[\lambda])$.
    Therefore, Chebotarev density theorem applied to the Galois extension $M_2 \cdot M(A[\lambda])$ of $M$ with Galois group $(\Z/2\Z)^d \times \GL_3(\ell)$, guarantees the existence of primes $\frakp_i|p_i$ in $M$ such that the characteristic polynomial of $\rho_\lambda(\Frob_{\frakp_i})$ is irreducible over $\FF_\ell$ (so that $f_{p_i,1}$ is irreducible over $\Z_\ell$ and hence $v_{p_i} \ne 0$) and $\chi_j(\frakp_i) = \delta_{i,j}$.
    So for $B$ large enough, the matrix $A$ constructed in \cref{step:AlgC2ConstructMatrix} of \cref{Alg:C2} will have full rank. This shows that the algorithm returns a finite set of primes, and not $0$.

    Let $\ell \ge 5$ be a $\C_2$-prime for $A$. If $A$ does not have semistable reduction at $\ell$, then $\ell \mid N$ and we are done. So we may assume that $\ell$ is a semistable prime for $A$. Then $\rho_\lambda(G_M)$ acts irreducibly on $A[\lambda]$ and is contained in a $\C_2$-type maximal subgroup for a prime $\lambda|\ell$.
    Let $L/M$ be the non-trivial abelian extension from \cref{prop:C2test} unramified outside the bad primes for $A$.
    
    If $[L:M]=2$, there exists a non-trivial character $\chi \in H$ which cuts out $L$. Since the matrix $A$ has full rank, we deduce that there exists a prime $p \in S$ and $\frakp | p$ such that $\chi(\frakp) \ne 0$. This means that $\Frob_\frakp$ is non-trivial in $\Gal(L/M)$. Hence \cref{prop:C2test} implies that either $v_p \equiv 0 \pmod{\lambda}$ or $v_p \equiv 0 \pmod{\lambdabar}$, and therefore $\Nm_{M/\Q}(v_p) \equiv 0 \pmod{\ell}$.

    If $[L:M]=3$, \cref{prop:C2test} says that $\Tr \rho_\lambda(\sigma) \equiv 0 \pmod{\lambda}$ for every non-trivial $\sigma \in \Gal(L/M)$. \cref{thm:AlgC3works} then guarantees that $\ell$ is in the set of primes returned by \cref{Alg:C3}.
\end{proof}

While \cref{Alg:C2} turns the problem of determining $\C_2$-primes into a finite problem,
the output can include false positives, i.e., extraneous primes $\ell$
for which the image $\rho_\lambda(\GM)$ is not contained in a $\C_2$-type maximal subgroup.
So it is useful to have an algorithm, that can decide for a given prime $\ell$, whether $\ell$ can possibly be a $\C_2$-prime by inspecting the reductions modulo $\ell$ of characteristic polynomials of Frobenius elements. \cref{Alg:C2WeedOut} serves this purpose. In particular, it is useful to treat the primes of non-semistable reduction for $A$, for which \cref{Alg:C2} says nothing. 

\begin{algorithm}[Weeding out extraneous primes from possibly $\C_2$-primes]
    \label{Alg:C2WeedOut}
    \hfill\\
    {\em Input:}
    \begin{itemize}
        \item an imaginary quadratic field $M$ such that the polarised abelian threefold $A/\Q$ has potential imaginary multiplication by $M$ of signature $(2,1)$.
        \item $L$-polynomials of $A/\Q$ at each prime $p < B$ of good reduction for $A$ that splits in $M$, for some bound $B$.
        \item a rational prime $\ell$ that is unramified in $M$.
    \end{itemize}
    {\em Output:} true or false depending on whether $\ell$ is possibly a $\C_2$-prime
    or not.
    \begin{enumerate}[leftmargin=!]
        \item \label{step:WeedC2ComputeLPolys} For each prime $p$, using the given $L$-polynomial, follow \cref{step:AlgC2ComputeLPolys} of \cref{Alg:C2} to construct $f_{p,i}(x) = x^3 + a_{p,i} x^2 + b_{p,i} x + c_{p,i}$, with $\overline{f_{p,1}} = f_{p,2}$.
        \item \label{step:WeedC2ReturnFalseC2} Choose a prime $\lambda$ above $\ell$ in $M$, and consider the reduction of $f_{p,1}(x)$ modulo $\lambda$.
        If the reduction does not fully factor in $\FF_\lambda[x]$ into linear polynomials, and is not of the form $x^3-a$, and $v_p \pmod{\lambda} \ne 0$ and $v_p \pmod{\overline{\lambda}} \ne 0$ where $v_p = a_{p,1}b_{p,1}-c_{p,1}$ then return False.
        \item \label{step:WeedC2ReturnTrue} 
        If no return statement has been executed while running \crefrange{step:WeedC2ComputeLPolys}{step:WeedC2ReturnFalseC2},
        then return True.
    \end{enumerate}
\end{algorithm}

Now we want to show that every extraneous prime $\ell$
will be weeded out by \cref{Alg:C2WeedOut}, if sufficiently many $L$-polynomials are available. We prove this in \cref{prop:Alg:C2WeedOutworks}. The main ingredient for this is an understanding of the intersection of the $\C_2$ and $\C_3$-type maximal subgroups, which \cref{lem:PreserveDecomposition} presents in some generality.

\begin{lemma}
    \label{lem:PreserveDecomposition}
    Let $\FF$ be a finite field of size $q$ and $\FF'/\FF$ be an extension of odd prime degree $s$.
    Suppose $H$ is a subgroup of $\GL_1(\FF') \rtimes \Gal(\FF'/\FF)$ such that $H \cap \GL_1(\FF') \leq \{ x \in \GL_1(\FF') |  x^s \in \GL_1(\FF) \}$. Then there exists a decomposition $\FF' = \bigoplus\limits_{i=0}^{s-1} V_i$ of $\FF$-vector spaces such that $H$ preserves the set $D = \{ V_i ~|~ 0 \leq i \leq s-1\}$.
\end{lemma}
\begin{proof}
    Let $\sigma = \Frob_q$ be a generator of the cyclic group $\Gal(\FF'|\FF) \simeq \Z/s\Z$. Since $H/(H \cap \GL_1(\FF')) \leq \Gal(\FF'|\FF)$ and $s$ is prime, we know that $H \leq \langle H \cap \GL_1(\FF'), \alpha \sigma \rangle$ for some $\alpha \in \GL_1(\FF')$. Since $\GL_1(\FF')$ is cyclic, its intersection with $H$ is cyclic as well. Let $h$ denote a generator of $H \cap \GL_1(\FF')$. We are given that $h^s \in \GL_1(\FF)$.

    \textbf{Case 1: $h \in \GL_1(\FF)$.} Then the action of $H \cap \GL_1(\FF') = \langle h \rangle$ preserves any decomposition of $\FF'$ into $\FF$-linear subspaces. So, all we need to find is a decomposition that is preserved by $\alpha \sigma$. If $\alpha \in \FF$, a decomposition preserved by $\sigma$ is also preserved by $\alpha \sigma$, and the normal basis theorem yields a decomposition preserved by $\sigma$. Suppose $\alpha \not\in \FF$. Then the element $\alpha$ generates the field $\FF'$ over $F$ because $s$ is prime, and $(\alpha \sigma)^s = \Nm(\alpha)$. So the minimal polynomial of the $\FF$-linear operator $\alpha \sigma$ is equal to $T^s - \Nm(\alpha)$ and $\FF' \simeq \FF[T]/(T^s - \Nm(\alpha))$. This isomorphism yields a decomposition $\FF' = \bigoplus\limits_{i=0}^{s-1} \FF \cdot T^i$, which is preserved by the operator $\alpha \sigma$.
    
    \textbf{Case 2: $h \not\in \GL_1(\FF)$.} Let $a = h^s \in \GL_1(\FF)$. Since raising to the $s^{th}$ power is not an isomorphism on $\FF$, we deduce that $q \equiv 1 \pmod{s}$. The element $h$ generates the field $\FF'$ over $\FF$ because $s$ is prime, so that $\FF' = \FF[h] \cong \FF[T]/(T^s-a)$ and $\Nm(h) = a$. Let $\omega \in \FF$ be a primitive $s^{th}$ root of unity, such that $\sigma(h) = \omega h$.
    
    Since $(\alpha \sigma)^s = \Nm(\alpha) \in H \cap \GL_1(\FF) = \langle h^s \rangle = \langle a \rangle$, we deduce that $\Nm(\alpha) = a^n$ for some $n$. Let $n = k s + m$ for integers $k, m$. Then the element $\sigma_1 = h^{-m} \alpha \sigma \in H$ satisfies $\sigma_1^s = \Nm(h^{-m}\alpha) = a^{n-m} = a^{ks}$. So, the minimal polynomial of $\sigma_1$ over $\FF$, which is equal to $T^s - a^{ks}$, completely factors as $\prod\limits_{i=0}^{s-1} (T-a^k \omega^i)$. Let $\FF' = \bigoplus\limits_{i=0}^{s-1} V_i$ be the corresponding decomposition into eigenspaces. Since $\sigma_1 (hx) = h^{-m} \alpha \sigma(hx) = \sigma(h) \sigma_1(x) = \omega h\sigma_1(x)$, we deduce that the action of $h$ permutes the eigenspaces of $\sigma_1$. Therefore, we conclude that $H \leq \langle h, \alpha \sigma \rangle = \langle h, \sigma_1 \rangle$ preserves the set $D = \{ V_i ~|~ 0 \leq i \leq s-1\}$.
\end{proof}

\begin{proposition}
    \label{prop:Alg:C2WeedOutworks}
    Let $A/\Q$ be a polarised abelian threefold, and $M$ be an imaginary quadratic field such that $M \cong \End^0(A_{\Qbar})$.
    Suppose $\ell > 7$ is a rational prime that is not a $\C_1$-prime or a $\C_8/\C_5$-prime for $A$.
    
    Suppose $B$ is large enough, i.e., sufficiently many $L$-polynomials for $A$ are given as input to \Cref{Alg:C2WeedOut}.
    Then $\ell$ is not a $\C_2$-prime for $A$ if and only if \Cref{Alg:C2WeedOut} returns false.
\end{proposition}
\begin{proof}
    Let $\lambda|\ell$ be a prime of $\OM$, and $\FF = \FF_\lambda$. The if direction follows in a straightforward way from \cref{lem:imprimitivity_min_polys}, which implies 
    that the characteristic polynomials of elements of the $\C_2$-type maximal subgroup have one of the following shapes:
    $\prod\limits_{i=1}^3 (x-a_i)$,
    $x^3 - a_1$ or
    $(x-a_1)(x^2-a_2)$, for some $a_i \in \FF$.

    Now we prove the converse.
    Since $\ell$ is not a $\C_1$-prime, we know that $\rho_\lambda(G_M)$ acts irreducibly on $\FF^3$.
    We are given that $\ell$ is not a $\C_2$-prime or a $\C_8/\C_5$-prime.
    Since $\ell > 7$, we know by \cref{subsec:C6test,subsec:ClassStest} that $\rho_\lambda(G_M)$ is not contained in a $\C_6$-type or $\mathcal{S}$-type maximal subgroup.
    Therefore, by \cref{prop:det_surjects_and_GM_surjects_mod_lambda_iff_GMZetaell_surjects} and the classification in \cref{section_maximal_sbgps} of maximal subgroups of $\CLe_3(\ell)$ not containing $\SL^\varepsilon_3(\ell)$,  we get that either $\rho_\lambda(G_M) = \CLe_3(\ell)$ 
    or $\rho_\lambda(G_M)$ is contained in a $\C_3$-type maximal subgroup. 
    
    If $\rho_\lambda(G_M) = \CLe_3(\ell)$, the Chebotarev density theorem guarantees the existence of primes $\frakp|p$ such that the characteristic polynomial of $\rho_\lambda(\Frob_\frakp)$ is an irreducible cubic polynomial over $\FF$ not of the form $x^3-a$. In other words, the reduction in $\FF[x]$ of the polynomial $f_{p,1}$ fits the conditions being checked in \cref{step:WeedC2ReturnFalseC2} of \cref{Alg:C2WeedOut}. So for $B$ large enough, i.e., when sufficiently many $L$-polynomials for $A$ are given as input, \cref{Alg:C2WeedOut} will return False.

    Now suppose $\rho_\lambda(G_M)$ is contained in a $\C_3$-type maximal subgroup. Let $\FF'|\FF$ be the degree $3$ extension in a fixed algebraic closure.\\
    \textbf{Case 1: $\ell$ is split in $M$.}
    Fix a natural inclusion $\psi \colon  \FF'^* \rtimes \Gal(\FF'/\FF) \xhookrightarrow{} \GL_3(\FF)$ such that $\rho_\lambda(G_M) \subseteq \im (\psi)$.
    Let $H = \{x \in \FF'^* | x^3 \in \FF^* \}$. If $\rho_\lambda(G_M) \cap \psi(\FF'^*) \leq \psi(H)$, then \cref{lem:PreserveDecomposition}
    say that $\rho_\lambda(G_M)$ preserves a decomposition $\FF' = \bigoplus\limits_{i=1}^3 ~\langle v_i \rangle$. This will imply that $\rho_\lambda(G_M)$ is contained in a $\C_2$-type maximal subgroup, contradicting the assumption that $\ell$ is not a $\C_2$-prime. So we deduce that $\rho_\lambda(G_M) \cap \psi(\FF'^*) \not\leq \psi(H)$, i.e., there exists $\alpha \in \FF'^*$ such that $\alpha^3 \not\in \FF^*$ and $\psi(\alpha) \in \rho_\lambda(G_M)$.\\
    \textbf{Case 2: $\ell$ is inert in $M$.}
    Then $[\FF:\FF_\ell] = 2$ and $[\FF':\FF_\ell] = 6$. Let $\FF'' \subset \FF'$ be the degree $3$ extension of $\FF_\ell$. For any quadratic field extension $l|k$, let $l_1$ denote the subgroup $\ker( \Nm : l^* \rightarrow k^*)$ of $l^*$, so that $\CL^-_1(\FF'') = \CU_1(\FF'') \cong \FF'_1$.
    Fix an inclusion $\psi \colon  \FF'_1 \rtimes \Gal(\FF'/\FF) \xhookrightarrow{} \FF'^* \rtimes \Gal(\FF'/\FF) \xhookrightarrow{} \CU_3(\ell)$ such that $\FF^* \im (\psi)$ is a maximal subgroup of $\CU_3(\ell)$ containing $\rho_\lambda(G_M) $.
    Let $H = \{x \in \FF'_1 | x^3 \in \FF_1 \}$. \Cref{lem:PreserveDecomposition}
    again implies that the intersection $\rho_\lambda(G_M) \cap \psi(\FF'_1)$ is not contained in $\psi(H)$, i.e., there exists $\alpha \in \FF'_1$ such that $\alpha^3 \not\in \FF_1$ and $\psi(\alpha) \in \rho_\lambda(G_M)$.
    
    The characteristic polynomial of $\psi(\alpha)$ is the minimal polynomial over $\FF$ of $\alpha \in \FF'$, and hence it is irreducible over $\FF$ and not of the form $x^3-a$. Chebotarev density theorem guarantees the existence of primes $\frakp|p$ such that $\rho_\lambda(\Frob_\frakp)$ is conjugate to $\psi(\alpha)$. Therefore the reduction in $\FF[x]$ of the polynomial $f_{p,1}$ fits the conditions being checked in \cref{step:WeedC2ReturnFalseC2}. So for $B$ large enough, \cref{Alg:C2WeedOut} will return False.
\end{proof}

\subsubsection{$\C_3$ field extension subgroups}
\label{subsec:C3test}

\begin{lemma}
\label{lem:C3_type_large_order_elements}
    Let $\FF$ be a finite field and $\FF'/\FF$ an extension of prime degree $s$.
    Then any element in $\GL_1(\FF') \rtimes \Gal(\FF'/\FF)$ of order greater than $s|\FF^*|$ is contained in the normal subgroup $\GL_1(\FF')$.
\end{lemma}

\begin{proof}
    We may write $\GL_1(\FF') \rtimes \Gal(\FF'/\FF) \cong (\FF')^* \rtimes \Gal(\FF'/\FF)  \cong \langle a \rangle \rtimes \langle f \rangle $ where $a$ has order $|\FF'|-1$, $f$ has order $s$ and $faf^{-1}= a^{|\FF|}$.
    Let $b \in \GL_1(\FF') \rtimes \Gal(\FF'/\FF)$.
    The relation $fa= a^{|\FF|} f$ allows us to write $b=a^m f^n$ for some $m,n$.
    
    If $n \equiv 0 \pmod{s}$, then $b$ is a power of $a$.
    Thus we may suppose that $n \not \equiv 0 \pmod{s}$ and show that the order of $b$ is at most $s|\FF^*|$.
    It will then follow that any element of order great than $s|\FF^*|$ is contained in $\langle a \rangle $.
    Considering $a$ as an element of $\FF'$, we have
    \[b^s = (a^mf^n)^s = \prod_{i=0}^{s-1} f^{in}a^m f^{-in} = \prod_{i=0}^{s-1} f^{i}a^m f^{-i} = \Nm(a^m) \in \FF^*\]
    where the penultimate equality follows from the fact that multiplication by $n$ induces an automorphism of $\ZZ/s\ZZ$.
    We deduce that the order of $b$ divides $s|\FF^*|$.
    Hence the result follows as explained above.
\end{proof}

\begin{proposition}
\label{prop:C3test}
    Let $\ell \geq 5$ be a semistable admissible prime for $A$ and $\lambda$ a prime of $M$ lying over $\ell$.
    Suppose $\rho_\lambda(G_M)$ is an absolutely irreducible subgroup contained in a $\C_3$-type subgroup.
    
    Then there exists an abelian extension $L/M$ of degree 3, unramified outside the primes of bad reduction for $A$, such that if the image of $\sigma \in \GM$ in $\Gal(L/M)$ is non-trivial then $\Tr \rho_\lambda(\sigma) \equiv 0 \pmod{\lambda}$.
\end{proposition}

\begin{proof}
    Suppose $\ell$ splits in $M$.
    The description given in Section \ref{section:fieldextnsbgps} of $\C_3$-type subgroups of $\GL_3(\ell)$ implies that there exists a injective homomorphism $\psi \colon  \FF_{\ell^3}^* \rtimes \Gal(\FF_{\ell^3}/\FF_\ell) \rightarrow \GL_3(\ell)$ such that $\rho_\lambda(G_M)$ is contained in the image of $\psi$.
    Since $\rho_\lambda(G_M)$ is absolutely irreducible by assumption, it cannot be contained in $\psi(\FF_{\ell^3}^*)$.
    Thus $Q \coloneqq \rho_\lambda(G_M)/(\rho_\lambda(G_M) \cap \psi(\FF_{\ell^3}^*))$ is a cyclic quotient of $\rho_\lambda(G_M)$ of order 3.
    Let $L$ denote the fixed field of $\rho_\lambda(G_M) \cap \psi(\FF_{\ell^3}^*)$.

    Let $a$ be a generator of $\FF_{\ell^3}^*$.
    Then the eigenvalues of $\psi(a)$ equal $a, a^\ell, a^{\ell^2}$ and any generator $f$ of $\psi (\Gal(\FF_{\ell^3}/\FF_\ell))$ permutes them cyclically.
    In particular, we find that the trace of such an $f$ must equal zero.

    Let $I$ be an inertia group corresponding to a prime above $\ell$.
    As $\ell \neq 3$, the action of $I$ on $A[\lambda]$ factors through its tame quotient and thus acts semisimply.
    In particular, its action is described by Proposition \ref{prop:inertia_at_l_description}.

    Suppose for a contradiction that $L/M$ is ramified at $\ell$. 
    Then there exists $\sigma \in I$ with non-trivial image in $\Gal(L/M)$.
    If the order of $\rho_\lambda(\sigma)$ is greater than $3(\ell-1)$, then $\rho_\lambda(\sigma) \in \psi(\FF_{\ell^3}^*)$ by Lemma \ref{lem:C3_type_large_order_elements}.
    As $L$ is the fixed field of $\rho_\lambda(G_M) \cap \psi(\FF_{\ell^3}^*)$, we deduce that the order of $\rho_\lambda(\sigma)$ is less than or equal to $3(\ell-1)$.
    Hence \cref{prop:inertia_at_l_description} combined with $\ell \geq 5$, implies $\rho_\lambda|_I = 2 \mathbbm{1}+\theta_1$ or $\mathbbm{1}+2\theta_1$.
    We deduce that the minimal polynomial of $\rho_\lambda(\sigma)$ is of the form $(x-1)(x-d)$ for some $d \in \FF_\ell^*$ of order divisible by 3.
    Applying Lemma \ref{lem:imprimitivity_min_polys} with $V = A[\lambda] \otimes \FF_{\ell^3}$ and $D(V)$ the decomposition given by the eigenspaces of $\psi(a)$ gives that $\rho_\lambda(\sigma)$ has order 2, a contradiction.

    Suppose now $\ell$ is inert in $M$.
    The maximal $\C_3$-type subgroups of $\CU_3(\ell)$ are isomorphic to $\FF_\lambda^*(\GU_1(\ell^3) \rtimes \Gal(\FF_{\ell^6}/\FF_{\ell^2}))$.
    More precisely, any such maximal subgroup of $\CU_3(\ell)$ is the image of an injective homomorphism $\psi \colon \FF_\lambda^*(\GU_1(\ell^3) \rtimes \Gal(\FF_{\ell^6}/\FF_{\ell^2})) \rightarrow \CU_3(\ell)$ such that $\psi(\FF_\lambda^*)$ is the scalar matrices in $\CU_3(\ell)$, the image of a generator $a$ of $\GU_1(\ell^3)$ has eigenvalues $a,a^{\ell^2},a^{\ell^4}$ which are permuted cyclically by the non-trivial elements in $\psi(\Gal(\FF_{\ell^6}/\FF_{\ell^2}))$.
    In particular, any non-trivial element in $\psi(\Gal(\FF_{\ell^6}/\FF_{\ell^2}))$ has trace zero.

    Since $\rho_\lambda(G_M)$ is absolutely irreducible by assumption, it cannot be contained in $\psi(\FF_\lambda\GU_1(\ell^3))$.
    Thus $Q \coloneqq \rho_\lambda(G_M)/(\rho_\lambda(G_M) \cap \psi(\FF_\lambda\GU_1(\ell^3))$ is a cyclic quotient of $\rho_\lambda(G_M)$ of order 3.
    Let $L$ denote the fixed field of $\rho_\lambda(G_M) \cap \psi(\FF_\lambda\GU_1(\ell^3))$.

    Let $I$ be an inertia group corresponding to a prime above $\ell$.
    As $\ell \neq 3$, the action of $I$ on $A[\lambda]$ factors through its tame quotient and thus acts semisimply.
    In particular, its action is described by Proposition \ref{prop:inertia_at_l_description}.

    Suppose for a contradiction that $L/M$ is ramified at $\ell$. 
    Then there exists $\sigma \in I$ with non-trivial image in $\Gal(L/M)$.
    If the order of $\rho_\lambda(\sigma)$ is greater than $3(\ell^2-1)$, then $\rho_\lambda(\sigma) \in \psi(\FF_\lambda\GU_1(\ell^3))$ by Lemma \ref{lem:C3_type_large_order_elements}.
    As $L$ is the fixed field of $\rho_\lambda(G_M) \cap \psi(\FF_\lambda\GU_1(\ell^3))$, we deduce that the order of $\rho_\lambda(\sigma)$ is less than or equal to $3(\ell^2-1)$.
    Hence \cref{prop:inertia_at_l_description} combined with $\ell \geq 5$, implies $\rho_\lambda|_I = \mathbbm{1}+\theta_1+\theta_2$ or $2 \theta_2 + \theta_2^{\ell}$.
    In the first case, the minimal polynomial of $\rho_\lambda(\sigma)$ is of the form $(x-1)(x-d)(x-\gamma)$ for some $d \in \FF_\ell^*$ and $\gamma \in \FF_\lambda \setminus \FF_\ell$.
    Applying Lemma \ref{lem:imprimitivity_min_polys} with $V = A[\lambda] \otimes \FF_{\ell^3}$ and $D(V)$ the decomposition given by the eigenspaces of $\psi(a)$ gives that $\rho_\lambda(\sigma)$ has order 2, a contradiction. 
    In the second case, i.e., $\rho_\lambda|_I = 2 \theta_2 + \theta_2^{\ell}$, consider $\gamma = \theta_2(\sigma)^{-1} \in \FF_\lambda^*$. The minimal polynomial of $\psi(\gamma) \rho_\lambda(\sigma)$ is of the form $(x-1)(x-d)$ for some $d \in \FF_\lambda^*$.
    Applying Lemma \ref{lem:imprimitivity_min_polys} again for the decomposition $D(V)$ gives that $\psi(\gamma) \rho_\lambda(\sigma)$ has order 2, i.e., $\rho_\lambda(\sigma)^2 = \psi(\gamma)^2$. This is a contradiction since $\rho_\lambda(\sigma)^2$ has non-trivial image in $\Gal(L/M)$, while $\psi(\gamma)^2$ does not.
\end{proof}

We now adapt \cref{prop:C3test} to turn it into an algorithm for detecting all possibly $\C_3$-primes for an abelian threefold $A$. The exact same idea as in \cref{Alg:C2}, of working with a basis of  homomorphisms from $G_M$ to $\FF_3$, is used to optimize the number of primes that need to be considered.

\begin{algorithm}[Possibly $\C_3$-primes]
    \label{Alg:C3}
    \hfill\\
    {\em Input:} Same as in \cref{Alg:C2}.\\

    {\em Output:} a finite superset of the $\C_3$-primes, or $0$.

    \begin{enumerate}[leftmargin=!]
        \item \label{step:AlgC3HomsToF3} Use the conductor exponent bounds from 
        \cref{rem:ConductorExponentBound} to compute a basis 
        $\{\chi_j : 1 \leq j \leq d\}$ for the $\FF_3$-vector space $H$ of Hecke 
        characters of $M$ unramified outside $N$, and having order dividing $3$. 
        Fix a primitive cube root of unity, and treat these characters as 
        homomorphisms $G_M \rightarrow \FF_3$.
        \item \label{step:AlgC3ComputeLPolys} For each prime $p < B, p \nmid N$, follow 
        \cref{step:AlgC2ComputeLPolys} of \cref{Alg:C2} to construct 
        $f_{p,i}(x) = x^3 + a_{p,i} x^2 + b_{p,i} x + c_{p,i}$, with 
        $\overline{f_{p,1}} = f_{p,2}$, from the $L$-polynomial.
        \item \label{step:AlgC3ConstructMatrix}
        Let $S$ denote the set of primes $p$ such that 
        $a_{p,1} \ne 0$.
        Let $S = \{p_i : 1 \leq i \leq m\}$. For each prime $p_i \in S$, choose a prime ideal $\frakp_i$ above $p_i$. Construct the $m \times d$ matrix $A \in M_{m,d}(\FF_3)$ 
        whose $(i,j)^{th}$ entry is 
        $\chi_j(\frakp_i) = \chi_j(\Frob_{\frakp_i}) \in \FF_3$.
        \item \label{step:AlgC3NotFullRankRepeat} If $\rank(A) < d$, return $0$.
        \item \label{step:AlgC3Return} Return the set of all rational primes dividing
        \begin{align*}
            6 N \cdot \prod\limits_{p \in S} \Nm_{M/\Q}(a_{p,1})
        \end{align*}
    \end{enumerate}
\end{algorithm}

\begin{theorem}
    \label{thm:AlgC3works}
    Let $A/\Q$ be a polarised abelian threefold, and $M$ be an imaginary quadratic field such that $M \cong \End^0(A_{\Qbar})$.
    Suppose \cref{assumption:surjective} holds.
    Then, for $B$ large enough, \Cref{Alg:C3} returns a finite superset of the $\C_3$-primes for $A$.
\end{theorem}
\begin{proof}
    By \cref{assumption:surjective}, choose a rational prime $\ell$ splitting in $M$, and $\lambda$ above $\ell$ so that $\rho_\lambda(G_M) = \GL_3(\ell)$.
    Let $M(A[\lambda])$ denote the fixed field of $\ker(\rho_\lambda)$. The maximal abelian sub-extension of $M(A[\lambda])/M$ corresponds to the quotient $\GL_3(\ell) \xrightarrow{\det} \FF_\ell^*$.
    So either there is a unique cyclic cubic sub-extension corresponding to the further quotient $\GL_3(\ell) \twoheadrightarrow \FF_\ell^*/\FF_\ell^{* 3}$ (when $\ell \equiv 1 \pmod 3$), or there are none (when $\ell \not\equiv 1 \pmod 3$).
    Let $I$ denote the inertia group at a prime above $\ell$.
    By \cref{prop:inertia_at_l_description}, we know that $\det \rho_\lambda|_I = \theta_1$ or $\theta_1^2$.
    Hence this cyclic cubic sub-extension of $M(A[\lambda])/M$, if it exists, is ramified at both $\lambda$ and $\lambdabar$.
    
    Let $M_3$ denote the compositum of the cyclic cubic extensions cut out by the order $3$ Hecke characters $\chi_j$ from \cref{step:AlgC3HomsToF3} of \cref{Alg:C3}.
    Since $M_3$ is unramified outside $N$, it is disjoint from $M(A[\lambda])$. 
    Therefore, Chebotarev density theorem applied to the Galois extension $M_3 \cdot M(A[\lambda])$ of $M$ with Galois group $(\Z/3\Z)^d \times \GL_3(\ell)$, guarantees the existence of primes $\frakp_i|p_i$ in $M$ such that $\Tr(\rho_\lambda(\Frob_{\frakp_i})) \ne 0 \in \FF_\ell$ (so that $a_{p_i,1} \ne 0 \in \Z_\ell$) and $\chi_j(\frakp_i) = \delta_{i,j}$.
    So for $B$ large enough, the matrix $A$ constructed in 
    \cref{step:AlgC3ConstructMatrix} of \cref{Alg:C3} will have full rank. This shows that the algorithm returns a finite set of primes, and not $0$.

    Let $\ell \ge 5$ be a $\C_3$-prime for $A$. If $A$ does not have semistable reduction at $\ell$, then $\ell \mid N$ and we are done. So we may assume that $\ell$ is a semistable prime for $A$. Then $\rho_\lambda(G_M)$ is 
    an absolutely irreducible subgroup contained in a $\C_3$-type maximal subgroup 
    for a prime ideal $\lambda|\ell$. 
    Let $L/M$ be the cyclic cubic extension given by \cref{prop:C3test}.
    Since $L/M$ is unramified outside $N$, there exists a non-trivial character 
    $\chi \in H$ which cuts out $L$. Since the matrix $A$ has full rank, we deduce that 
    there exists a prime $p \in S$ and $\frakp | p$ such that $\chi(\frakp) \ne 0$. 
    This means that $\Frob_\frakp$ is non-trivial in $\Gal(L/M)$. Hence 
    \cref{prop:C3test} implies that 
    $\Tr\rho_\lambda(\Frob_\frakp) \equiv 0 \pmod{\lambda}$, 
    i.e., $a_{p,i} \equiv 0 \pmod{\lambda}$ for $i = 1$ or $2$, which implies that
    $\Nm_{M/\QQ}(a_{p,1}) = a_{p,1}\overline{a_{p,1}} = a_{p,2}\overline{a_{p,2}} \equiv 0 \pmod{\ell}$.
\end{proof}

While \cref{Alg:C3} turns the problem of determining $\C_3$-primes into a finite problem,
the output can include false positives,
and furthermore it does not say anything about primes of non-semistable reduction for $A$.
\cref{Alg:C3WeedOut} is a simple test to decide for any given prime $\ell$, whether it can possibly be a $\C_3$-prime by inspecting the reductions modulo $\ell$ of Frobenius polynomials.

\begin{algorithm}[Weeding out extraneous primes from possibly $\C_3$-primes]
    \label{Alg:C3WeedOut}
    {\em Input:} Same as in \cref{Alg:C2WeedOut}.\\
    {\em Output:} true or false depending on whether $\ell$ is possibly a $\C_3$-prime
    or not.
    \begin{enumerate}[leftmargin=!]
        \item \label{step:WeedC3ComputeLPolys} For each prime $p$, using the given $L$-polynomial, follow \cref{step:AlgC2ComputeLPolys} of \cref{Alg:C2} to construct $f_{p,i}(x) = x^3 + a_{p,i} x^2 + b_{p,i} x + c_{p,i}$, with $\overline{f_{p,1}} = f_{p,2}$.
        \item \label{step:WeedC3ReturnFalse} Choose a prime $\lambda$ above $\ell$ in $M$, and consider the reduction of $f_{p,1}(x)$ modulo $\lambda$.
        If this polynomial is reducible in $\FF_\lambda[x]$, but not of the form $x^3-a$ or $(x-a)^3$, then return False.
        \item \label{step:WeedC3ReturnTrue} If no return statement has been executed while running \crefrange{step:WeedC3ComputeLPolys}{step:WeedC3ReturnFalse},
        then return True.
    \end{enumerate}
\end{algorithm}

\begin{proposition}
    \label{prop:Alg:C3WeedOutworks}
    Let $A/\Q$ be a polarised abelian threefold, and $M$ be an imaginary quadratic field such that $M \cong \End^0(A_{\Qbar})$.
    Suppose $\ell > 7$ is a rational prime that is not a $\C_1$-prime.
    Then $\ell$ is not a $\C_3$-prime for $A$ if \Cref{Alg:C3WeedOut} returns false.
\end{proposition}
\begin{proof}
    Let $\lambda|\ell$ be a prime of $\OM$, and $\FF = \FF_\lambda$. The characteristic polynomial of an element of the $\C_3$-type maximal subgroup has one of the following shapes:
    an irreducible cubic,
    $(x-a)^3$, or
    $x^3 - a$ for some $a \in \FF^*$. So, if the reduction of $f_{p,1}$ modulo $\lambda$ is not of these shapes, we can deduce that $\rho_\lambda(\Frob_\frakp)$ does not belong to the $\C_3$-type maximal subgroup for some prime $\frakp|p$. Hence $\ell$ is not a $\C_3$-prime for $A$.
\end{proof}

\begin{remark}
    It is worth noting that unlike \cref{Alg:C2WeedOut} which can eliminate any extraneous $\C_2$-prime given sufficiently many $L$-polynomials as input (\cref{prop:Alg:C2WeedOutworks}), \cref{Alg:C3WeedOut} is not guaranteed to eliminate extraneous $\C_3$-primes.
\end{remark}

\subsubsection{$\C_6$ symplectic type subgroups}
\label{subsec:C6test}

The description of inertia at primes above $\ell$ given in Proposition \ref{prop:inertia_at_l_description} shows us that for $\ell>7$ the projective image of $\rho_\lambda(G_{M})$ contains a cyclic element with order greater than $6$.
As $C_3^2 \rtimes \SL_2(3)$ does not contain such an element, we deduce
that $\rho_\lambda(G_M)$ cannot be contained in a $\C_6$-type maximal subgroup of $\CLe_3(\ell)$
unless $\ell =5$ and $\ell$ is inert in $M$ or $\ell=7$ and $\ell$ splits in $M$.

Characteristic polynomials of elements of the $\C_6$-type maximal subgroup are either reducible or of the form $x^3-a$ for some $a \in \FF$.
So exhibiting an irreducible Frobenius characteristic polynomial not of the form $x^3-a$ is sufficient to rule out 
$\ell \in \{5,7\}$ as $\C_6$-primes.
If $\ell \in \{5,7\}$ is a $\C_1$-prime, this simple strategy does not work.

\subsubsection{$\C_8/\C_5$-type orthogonal subgroups}
\label{subsec:Orthogonaltest}

Let $\ell > 2$ be an admissible prime for $A$, $\lambda$ be a prime of $M$ above $\ell$, and $\varepsilon = (\frac{\Delta_M}{\ell})$, and assume $\ell \not\equiv \varepsilon \pmod 3$.
Let $\SO_3(\ell)$ denote the subgroup of $\CLe_3(\ell)$ consisting of the determinant $1$ isometries of a non-degenerate symmetric bilinear pairing on the underlying vector space.
Then the similitude group of the pairing is $\FF_\lambda^* \SO_3(\ell)$, 
and it is a maximal subgroup of $\CLe_3(\ell)$. 
We first discuss 
the shape of characteristic polynomials of elements of this maximal subgroup.

\begin{lemma}
\label{lemma:CO_charpol}
Let $\sigma$ be any element in the maximal subgroup $\FF_\lambda^* \SO_3(\ell)$ of $\CLe_3(\ell)$, and $\phi$ be the characteristic polynomial of $\sigma$. Then there exists $\mu$ in $\FF_\ell^*$ if $\varepsilon = +1$ (or in $\FF_{\ell^2}^*$ if $\varepsilon = -1$) such that 
$\phi(x) = \frac{-x^3}{\mu^3} \phi\left(\frac{\mu^2}{x}\right)$, i.e., $\phi(x) = x^3 + a x^2 - a \mu x - \mu^3 = (x-\mu)(x^2+(a+\mu)x+\mu^2)$ for some $a$.
\end{lemma}
\begin{proof}
Let $\mu$ denote the scalar, considered as a scalar matrix, such that $\sigma = \mu \tau$ for some $\tau \in \SO_3(\ell)$. Let $X$ be the matrix of the symmetric bilinear pairing. Then $^t\tau X \tau = X$, and therefore $^t\tau$ and $\tau^{-1}$ are conjugate. Therefore the set of eigenvalues of $\tau$ is stable under inversion. Since $\det(\tau) = 1$, we deduce that the set of eigenvalues of $\tau$ is equal to $\{1, \nu, \nu^{-1} \}$ for some $\nu \in \Bar \FF_\ell$. Therefore the eigenvalues of $\sigma$ are $\{ \mu, \mu \nu, \mu \nu^{-1} \}$, from which the shape of its characteristic polynomial is evident.  
\end{proof}

\begin{algorithm}[Possibly $\C_8/\C_5$-primes]
    \label{Alg:CO}
    \hfill\\
    {\em Input:} Same as in \cref{Alg:C2}.\\
    {\em Output:} a finite superset of the $\C_8/\C_5$-primes, or $0$.
    \begin{enumerate}[leftmargin=!]
        \item \label{step:AlgCOComputeLPolys} For each prime $p < B, p \nmid N$, follow \cref{step:AlgC2ComputeLPolys} of \cref{Alg:C2} to construct 
        $f_{p,i}(x) = x^3 + a_{p,i} x^2 + b_{p,i} x + c_{p,i}$, with 
        $\overline{f_{p,1}} = f_{p,2}$, from the $L$-polynomial.
        \item \label{step:AlgCOGetValues}
        Let $S$ denote the set of primes $p$ such that 
        $v_p \coloneqq b_{p,1}^3 - c_{p,1} a_{p,1}^3 \ne 0$.
        \item \label{step:AlgCO_Fail} If $S$ is empty, return $0$.
        \item \label{step:AlgCO_Success} Return the set of all rational primes dividing $\gcd\limits_{p \in S} \Nm_{M/\Q}(v_p)$.
    \end{enumerate}
\end{algorithm}

\begin{proposition}
\label{prop:COtest}
    Let $A/\Q$ be a polarised abelian threefold, and $M$ be an imaginary quadratic field such that $M \cong \End^0(A_{\Qbar})$.
    Then, for $B$ large enough, \cref{Alg:CO} returns a finite superset of the $\C_8/\C_5$-primes for $A$.
\end{proposition}
\begin{proof}
    We first show that \cref{Alg:CO} does not return $0$, if $B$ is large enough, i.e., if sufficiently many $L$-polynomials are given as input. It is enough to show that there exists some prime $p$ such that the quantity $v_p$ in \cref{step:AlgCOGetValues} is not zero.
    By \cref{assumption:surjective}, choose a rational prime $\ell$ splitting in $M$, and $\lambda$ above $\ell$ so that $\rho_\lambda(G_M) = \GL_3(\ell)$. By Chebotarev density theorem there exists a prime $\frakp|p$ in $M$ such that the characteristic polynomial of $\rho_\lambda(\Frob_{\frakp})$ is irreducible over $\FF_\ell$. Therefore $f_{p,1}$ is irreducible over $\Z_\ell$, and hence $v_p = b_{p,1}^3 - c_{p,1}a_{p,1}^3 \ne 0$, so we are done.
    
    Let $\ell$ be an admissible prime for $A$, and $\lambda$ be a prime of $M$ above $\ell$. Suppose that $\rho_\lambda(G_M)$ is contained the $\C_8/\C_5$-type orthogonal maximal subgroup of $\CLe_3(\ell)$. 
    Recall that the reduction of $f_{p,1}(x)$ modulo $\lambda$ is the characteristic polynomial of $\rho_\lambda(\Frob_\frakp)$ for some prime $\frakp$ of $M$ above $p$. So it should have the shape given in \cref{lemma:CO_charpol}. This implies that either $a_{p,1} \equiv b_{p,1} \equiv 0 \pmod{\lambda}$, or $(b_{p,1}/a_{p,1})^3 \equiv c_{p,1} \pmod{\lambda}$. Both cases imply that $v_p \coloneqq b_{p,1}^3 - c_{p,1} a_{p,1}^3 \equiv 0 \pmod{\lambda}$, which means that $\ell$ divides $\Nm_{M/\Q} v_p$. Therefore, $\ell$ is in the set of primes outputted by \cref{Alg:CO}.
\end{proof}

\subsubsection{$\mathcal{S}$-type subgroups}
\label{subsec:ClassStest}

In \cref{subsec:ClassS_maximalgroup}, we mentioned necessary congruence conditions for the existence of an $\mathcal{S}$-type maximal subgroup of $\CLe_3(\ell)$ with projective image $\PSL_2(7)$. They force $\ell=3$ and $\varepsilon=-1$, or $\ell \geq 11$.
The description of inertia at primes above $\ell$ given in Proposition \ref{prop:inertia_at_l_description} shows us that for $\ell \geq 11$ the projective image of $\rho_\lambda(G_{M})$ contains a cyclic element of order at least $10$.
As $\PSL_2(7)$ does not contain such an element, we deduce that $\rho_\lambda(G_M)$ cannot be contained in such an $\mathcal{S}$-type maximal subgroup of $\CLe_3(\ell)$
unless $\ell = 3$ and $\ell$ is inert in $M$.

If $3$ is inert in $M$, to rule out containment of $\rho_3(G_M)$ in the $\mathcal{S}$-type  maximal subgroup, one needs to exhibit a Frobenius characteristic polynomial that does not come from an element of this maximal subgroup.

\subsubsection{Semistability criterion}
In this section, as before, $A$ denotes an abelian threefold defined over $M$ such that there exists a unital algebra embedding $\iota \colon M \hookrightarrow \End^0(A_{\overline{M}})$.
For a prime $\p$ of $M$, let $M_\p$ denote the completion of $M$ at $\p$ and $\OO_\p$ its ring of integers.
Let $A_\p$ denote the special fibre of the Néron model of $A$ over $\OO_\p$.
Let $A^0_\p$ denote the connected component of the identity of $A_\p$.
Recall that there exists a short exact sequence
\[0 \rightarrow A_\p^0 \rightarrow A_\p \rightarrow\Phi_{A,\p} \rightarrow 0\]
where $\Phi_{A,\p}$ is a finite étale group scheme over $k_\p \coloneqq \OO_\p/\p\OO_\p$.
Recall that $A$ is said to have semistable reduction at $\p$ if there exists an exact sequence
\[0 \rightarrow T_\p \rightarrow A_\p^0\rightarrow B_\p \rightarrow 0 \]
where $T_\p$ a torus  and $B_\p$  is an abelian variety.
The dimension of $T_\p$ is known as the toric rank of $A_\p$.
We say that $A$ has bad semistable reduction at $\p$ if it has semistable reduction and the toric rank of $A_\p$ is positive.

\begin{theorem}
    \label{thm:semistable_rules_out_C_2,C_3,C_6}
    Suppose $A$ has bad semistable reduction at a prime $\p$ of $M$.
    Let $\ell \not \in \p$ be an odd rational admissible prime not dividing the order of $\Phi_{A,\p}$ and $\lambda$ a prime of $M$ lying above $\ell$.

    Then either $G_M$ acts reducibly on $A[\lambda]$ or $\rho_\lambda(G_M) \cong \CLe_3(\ell)$.
\end{theorem}

\begin{proof}
    By \cite[Prop. 3.5]{SGA7}, for any $\sigma \in I_\p$, we have that $(\sigma-1)^2$ annihilates $T_\ell(A)$.
    By \cite[Lemma 2]{Serre_Tate}, we have that $A[\ell]^{I_\p}=A_\p[\ell]$.
    As the order of $\Phi_{A,\p}$ is coprime to $\ell$, we have $\dim_{\FF_\ell}A_\p[\ell] = 6-t<6$ where $t$ is the toric rank of $A_\p$ by \cite[Lemma 1]{Serre_Tate}.
    Thus $I_\p$ acts non-trivially on $A[\ell]$ and hence also non-trivially on $A[\lambda]$.
    In particular, there exists $\sigma \in I_\p$ such that $\rho_\lambda(\sigma) \neq 1$ and $(\rho_\lambda(\sigma)-1)^2=0$.
    Note this means the minimal characteristic polynomial of $\rho_\lambda(\sigma)$ is $(x-1)^2$.
    
    Suppose $\rho_\lambda(G_M)$ is contained in a $\C_2$ or $\C_3$-type subgroup.
    Then $\rho_\lambda(G_M)$ preserves a decomposition $A[\lambda] \otimes_{\FF_\lambda} \overline{\FF}_\lambda = \bigoplus_{i=1}^3 V_3$.
    Applying \cite[Lemma 5.8]{Goodman_superelliptic}, we see this is only possible if $\ell=2$.
    Note that for $\ell \geq 5$, order considerations directly make this impossible.

    For $\ell =3$, there are no maximal $\C_6$-type subgroups in $\CL^\varepsilon_3(\ell)$.
    Thus let us suppose $\ell \geq 5$.
    The relation $(\rho_\lambda(\sigma)-1)^2=0$ implies $0=(\rho_\lambda(\sigma)-1)^\ell=\rho_\lambda(\sigma)^\ell-1$, coupling this with $\rho_\lambda(\sigma) \neq 1$, we deduce $\rho_\lambda(\sigma)$ has order $\ell$.
    Since the $\C_6$-type subgroups do not contain any elements of order $\ell$ for $\ell \geq 5$, we deduce that $\rho_\lambda(\GM)$ is not contained in a $\C_6$-type subgroup.

    Since the minimal polynomial of $\rho_\lambda(\sigma)$ is $(x-1)^2$ and $A[\lambda]$ has dimension $3$ over $\FF_\lambda$, we deduce from its Jordan Normal Form that $\ker(\rho_\lambda(\sigma)-1)$ is of dimension 2.
    Thus \cite[Cor. 6.10]{Grove} implies $\rho_\lambda(\sigma)$ does not preserve a non-degenerate symmetric pairing.

    If $\rho_\lambda(G_M)$ is contained in a $\mathcal{S}$-type subgroup, then $\rho_\lambda(G_M) $ is contained in a maximal subgroup isomorphic to $\FF^*_\lambda \times  \PSL_2(7)$.
    By \cref{subsec:ClassS_maximalgroup}, this can only happen if $\ell =3$ or $\ell \geq 11$.
    Since $\rho_\lambda(\sigma)$ has order $\ell$ and the order of $\FF^*_\lambda \times  \PSL_2(7)$ is coprime to $\ell$ for $\ell \geq 11$, we deduce $\ell =3$ is the only possibility.
    Examining the representations of $\PSL_2(7)$ over a field of characteristic 3, we find this maximal subgroup does not contain an element with minimal polynomial equal to $(x-1)^2$.
    Thus $\rho_\lambda(G_M)$ is not contained in a $\mathcal{S}$-type subgroup.

    It follows from the classification in \cref{section_maximal_sbgps} of maximal subgroups of $\CLe_3(\ell)$ not containing $\SL_3^\varepsilon(\ell)$, that either $G_M$ acts reducibly on $A[\lambda]$ or $\rho_\lambda(G_M) \cong \CLe_3(\ell)$.
\end{proof}

\section{Results}
\label{sec:results}
In this section we present the results of running the algorithms developed in \cref{section:criteriaforfullimage} on various families of curves, and a dataset of Picard curves with good reduction outside $\{2,3,5,7\}$.

For hyperelliptic curves, $L$-polynomials are computed efficiently using the algorithms devised and implemented in \cite{HS14,HS16,KS08}.
For Picard curves, $L$-polynomials are computable in average polynomial time, by combining the work \cite{Sut20} of Sutherland on computing $L_p(C,x) \pmod p$ for superelliptic curves, and the work \cite{AFP22} of Asif, Fite and Pentland to lift this to the true integral $L$-polynomial $L_p(C,x)$ when $C$ is a Picard curve.
We use these implementations to obtain $L$-polynomials efficiently for our computations. For hyperelliptic and Picard curves, we work with $L$-polynomials at all good primes $p < 2^{12}$. For other smooth plane quartic curves, we work with $L$-polynomials at all good primes $p < 2^{10}$.

Recall that $B_d$ denotes the quaternion algebra over $\QQ$ of discriminant $d$.
\subsection{Hyperelliptic Curves}

Consider the family of smooth genus $3$ hyperelliptic curves
\begin{align}
    \label{eq:Qi-family}
    y^2 = f(x) = x^7 + a x^5 + b x^3 + c x.
\end{align}
The map $(x,y) \mapsto (-x,iy)$ is an automorphism defined over $\Q(i)$. Hence the geometric endomorphism algebra of the Jacobian of a curve in this family contains $\Q(i)$, with equality for a generic curve. We considered the finite set of separable polynomials $f(x) \in \Z[x]$ in the box $\calB$ defined by $|a|, |b|, |c| \leq 100$. Note that multiple polynomials may define the same curve upto twist, or even upto isomorphism, and in these cases the results will be identical. But we have not reduced the dataset up to isomorphism/twists.

Among the \numprint{8079794} polynomials in this box, \cref{Alg:C1} fails for \numprint{1052}. These define curves that have a bigger geometric endomorphism algebra of Jacobian than $\Q(i)$. Most of these curves are twists of the curve $y^2 = x^7 + x$, or a curve in the family $y^2 = x^7 + ax^5 + ax^3 + x$. Heuristic computations indicate that the generic geometric endomorphism ring in this family is an order of index $2^5$ in $\Q(i) \times M_2(\Q)$.
\cref{table:otherQiCurveEndAlgs} displays several other non-CM endomorphism algebras, that were observed for a very few curves. 

\begin{center}
\begin{table}[h!]
\begin{center}
\begin{tabular}{|c|c||c|c|}
    \hline
    $\End^0(A_{\Qbar})$ & Index & $\End^0(A_{\Qbar})$ & Index\\
    \hline
    $\Q(i) \times M_2(\Q)$ & $5^3$ & 
    $\Q(i) \times M_2(\Q)$ & $2^2 \cdot 5^3$\\
    $\Q(i) \times M_2(\Q(\sqrt{-2}))$ & $2^4$ & 
    $\Q(i) \times M_2(\Q(\sqrt{-3}))$ & $2^{11}$\\
    $\Q(i) \times M_2(\Q(\sqrt{-5}))$ & $5^4$ & 
    $\Q(i) \times B_6$ & $2 \cdot 3^2$\\
    $M_3(\Q(i))$ & $3^8$ & & \\
    \hline
\end{tabular}
\caption{Some geometric endomorphism algebras $\End^0(A_{\Qbar})$, along with the index of $\End(A_{\Qbar})$ in a maximal order, 
that are observed numerically for Jacobians $A$ of a few hyperelliptic curves as in \cref{eq:Qi-family}.}
\label{table:otherQiCurveEndAlgs}
\end{center}
\end{table}
\end{center}

For the curves defined by the remaining \numprint{8078742} polynomials, \cref{Alg:C1,Alg:C2,Alg:C3,Alg:CO} succeed. After weeding out extraneous $\C_2$ and $\C_3$-primes using \cref{Alg:C2WeedOut,Alg:C3WeedOut}, the set of primes returned is contained in $\{3\}$. There are no $\C_6$-type maximal subgroups to worry about since $M = \Q(i)$.
We did not rule out containment of $\rho_3(G_M)$ in the $\mathcal{S}$-type maximal subgroup of $\CU_3(3)$.
Thus we find that, for all these curves, the set of non-surjective semistable
admissible primes is contained in 
$\{3\}$.
\cref{table:C3C2reducible_hypell} gives refined information about the individual outputs of 
these algorithms, i.e.,
the admissible primes $\ell$ at which the mod-$\lambda$ Galois image for $\lambda|\ell$ is 
possibly contained in one of the maximal subgroups from \cref{section_maximal_sbgps}.

\begin{center}
\begin{table}[h!]
\begin{center}
\begin{tabular}{|c|c|c|}
    \hline
    Superset of & Output & Number of polynomials\\
    \hline
    Semistable $\C_1$-primes & $\{\}$ & \numprint{8078648}\\
     & $\{3\}$ & \numprint{94}\\
    \hline
    $\C_3$-primes & $\{\}$ & \numprint{8078742}\\
    \hline
    $\C_2$-primes & $\{\}$ & \numprint{8078700}\\
     & $\{3\}$ & \numprint{42}\\
    \hline
    $\C_8/\C_5$-primes & $\{\}$ & \numprint{8078742}\\
    \hline
\end{tabular}
\caption{Output of \cref{Alg:C1,Alg:C2,Alg:C3,Alg:CO}, after removing extraneous primes, for polynomials in the box $\calB$ defining smooth curves 
with endomorphism ring $\Z[i]$.}
\label{table:C3C2reducible_hypell}
\end{center}
\end{table}
\end{center}

For $d=2,3,4$, there are several $1$-dimensional families of hyperelliptic curves $C_a$ over $\Q$ with generic geometric endomorphism ring equal to $\Z[\sqrt{-d}]$. In forthcoming work, we study these families in detail and in particular certify the extra endomorphisms. We consider the \numprint{1216767} rational numbers of height up to $1000$ for values of $a$.
For all but two of these curves, the set of 
admissible 
primes returned by \Cref{Alg:C1,Alg:C2,Alg:C3,Alg:CO} is 
empty.
The two exceptions give $3$ as a possibly reducible prime.
\cref{table:reducible_hypell_sqrtminusd} shows these families and summarises the data.

For the curves in the $\Z[\sqrt{-2}]$ and $\Z[\sqrt{-3}]$ families, we thus find that the mod-$\lambda$ image is $\CLe_3(\ell)$, for all $\lambda|\ell$ for all 
semistable
admissible primes $\ell$.
For the curves in the $\Z[\sqrt{-4}]$ families, we did not rule out containment of $\rho_3(G_M)$ in the $\mathcal{S}$-type maximal subgroup of $\CU_3(3)$, so the mod-$\lambda$ image is $\CLe_3(\ell)$, for all $\lambda|\ell$ for all 
semistable
admissible primes $\ell \ne 3$.

\begin{center}
\begin{table}[h!]
\begin{center}
\begin{tabular}{|p{5cm}|p{2.5cm}|p{1.7cm}|p{2cm}|}
    \hline
    Hyperelliptic family & Generic $\End(\Jac(C_a)_{\Qbar})$ & Semistable $\C_1$-primes & Number of curves\\
    \hline
    $y^2=(x+2)(x^2-2)$\newline$(x^4-4x^2+a)$ & $\Z[\sqrt{-2}]$ & $\{\}$ & \numprint{1216763}\\
    \hline
    $y^2=x((x^2-3)^3+$\newline$a(x+2)(x+1)^4)$ & $\Z[\sqrt{-4}]$ & $\{\}$ & \numprint{1216764}\\
    \hline
    $y^2=(x^4 +ax^2-\frac{a}{2})$\newline$(x^4 +2ax^2-2ax +\frac{a}{2})$ & $\Z[\sqrt{-4}]$ & $\{\}$ & \numprint{1216762}\\
    & & $\{3\}$ & \numprint{2}\\
    \hline
    $y^2 = x(x^2-x-1/12)$ \newline $((x^2-x-1/12)^2+ax^3)$ & $\Z[\sqrt{-3}]$ & $\{\}$ & \numprint{1216764}\\
    \hline
    $y^2 = (x^2-x+1)(x^2+x+1)$ \newline $(x(x^3+2)+a(2x^3+1))$ & $\Z[\sqrt{-3}]$ & $\{\}$ & \numprint{1216765}\\
    \hline
\end{tabular}
\caption{Output of \cref{Alg:C1} for smooth genus $3$ curves in $1$-parameter hyperelliptic families $C_a$, with $a \in \Q, \Ht(a) \le 1000$ and given endomorphism ring.}
\label{table:reducible_hypell_sqrtminusd}
\end{center}
\end{table}
\end{center}

In the theorem below, for $f \in \QQ[x]$ separable, we let $J_f$ denote the Jacobian of the smooth projective hyperelliptic curve defined by the affine model $y^2=f(x)$.
We also write $\rho_\ell \colon \GQ \rightarrow \Aut(J_f[\ell])$ for the associated mod-$\ell$ representation.
\begin{theorem}
\label{thm:inverse_Galois_explicit_curves}
       Let $\ell$ be a prime, not congruent to $1, 25, 121$ modulo $168$, then $\GammaU_3(\ell)$ is a Galois group over $\QQ$. 
        Furthermore, 
        \begin{itemize}
            \item for $\ell \equiv 2 \pmod{3}$, $\ell \neq 2$ and $f(x) = (x^2-x+1)(x^2+x+1)(x(x^3+2)+3(2x^3+1)))$, we have $\rho_\ell(\GQ) \cong \GammaU_3(\ell)$; 
            \item for $\ell \equiv 3 \pmod{4}$ and $f(x) = x((x^2-3)^3+5(x+2)(x+1)^4)$, we have $\rho_\ell(\GQ) \cong \GammaU_3(\ell)$;
            \item for $\ell \equiv 5,7 \pmod{8}$ and $f(x) = (x+2)(x^2-2)(x^4-4x^2+3)$, we have $\rho_\ell(\GQ) \cong \GammaU_3(\ell)$.
            \item for $\ell \equiv 3,5,6 \pmod{7}$, and the smooth plane quartic curve $C: 2x^3z+x^2y^2+x^2yz-xy^2z+2xyz^2+xz^3-y^4-2y^3z-2yz^3+z^4=0$, we have $\rho_\ell(\GQ) \cong \GammaU_3(\ell)$.
        \end{itemize}
\end{theorem}

\begin{proof}
    The group $\GammaU_3(2)$ is soluble and thus a Galois group over $\QQ$ by Shafarevich's Theorem for soluble groups \cite{Shafarevich_soluble_groups_are_Galois_groups}.
    Hence we may assume $\ell >2$.
    Thus to prove $\GammaU_3(\ell)$ is a Galois group over $\QQ$ for $\ell$ not congruent to $1, 25, 121$ modulo $168$, it suffices to prove the `furthermore' claim of the statement.

    The last claim is immediate from the computation described in \cref{subsec:other_smoothplanequartics}, using the fact that $C$ has good reduction outside $\{2,7\}$.
    To prove the given statement for any congruence class of primes, it suffices
    by 
    \cref{table:reducible_hypell_sqrtminusd} 
    and the computational results mentioned prior to \cref{table:reducible_hypell_sqrtminusd}
    to check that $J_f$ has good reduction at all primes in the given congruence class.
    To do so, it is enough verify that the given primes $\ell$ do not divide the discriminant of $f$.
    The discriminant of the polynomial $(x^2-x+1)(x^2+x+1)(x(x^3+2)+3(2x^3+1))$ is $-2^{12} \cdot 3^9 \cdot 7^6$, thus for $\ell \equiv 2 \pmod{3}$, $\ell \neq 2$ we have $\rho_\ell(\GQ) \cong \GammaU_3(\ell)$.
    The discriminant of the polynomial $x((x^2-3)^3+5(x+2)(x+1)^4)$ is $2^{14} \cdot 5^4 \cdot 17^5$, thus for $\ell \equiv 3 \pmod{4}$ and $\ell > 3$ we have $\rho_\ell(\GQ) \cong \GammaU_3(\ell)$. For $\ell=3$, we also compute enough Frobenius characteristic polynomials to rule out containment of $\rho_3(G_{\Q(i)})$ in the $\mathcal{S}$-type maximal subgroup, thus showing that $\rho_3(G_{\Q(i)}) \cong \CU_3(3)$ and hence $\rho_3(\GQ) \cong \GammaU_3(3)$.
    The discriminant of the polynomial $(x+2)(x^2-2)(x^4-4x^2+3)$ is $2^{13} \cdot 3^3$, thus for $\ell \equiv 5,7 \pmod{8}$ we have $\rho_\ell(\GQ) \cong \GammaU_3(\ell)$.
\end{proof}

\subsection{Picard curves}

We considered the finite set of separable quartic polynomials $f(x) = x^4 + a x^2 + b x + c \in \Z[x]$ lying in the box $\calB$ given by $|a|, |b|, |c| \leq 100$ and $b > 0$. Among the \numprint{4039833} polynomials in $\calB$, \cref{Alg:C1} fails for exactly \numprint{27}, indicating that the Jacobians of the corresponding Picard curves $y^3 = f(x)$ have extra geometric endomorphisms besides $\Z[\zeta_3]$. Heuristic computations confirm that these Jacobians are not simple. All but two of these \numprint{27} curves fall into the first family in \cref{table:newPicardCurveFamilies} up to twist. The remaining \numprint{2} curves fall into the second and fourth families in \cref{table:newPicardCurveFamilies} upto twist.

For the remaining \numprint{4039806} curves, \cref{Alg:C1,Alg:C2,Alg:C3} succeed. Recall that since $M=\Q(\zeta_3)$, there are no $\C_8/\C_5$-type or $\mathcal{S}$-type maximal subgroups to worry about. After weeding out using \cref{Alg:C2WeedOut,Alg:C3WeedOut}, 
we find that the set of non-surjective 
semistable
admissible primes is contained in $\{2,7\}$.
\cref{table:C3C2reducible_PicardBoxSearch} gives refined information.

\begin{center}
\begin{table}[h!]
\begin{center}
\begin{tabular}{|c|c|c|}
    \hline
    Superset of & Output & Number of curves\\
    \hline
    Semistable $\C_1$-primes & $\{2\}$ & \numprint{4039804}\\
    & $\{2,7\}$ & \numprint{2}\\
    \hline
    $\C_3$-primes & $\{\}$ & \numprint{4039806}\\
    \hline
    $\C_2$-primes & $\{\}$ & \numprint{2019805}\\
    & $\{2\}$ & \numprint{2019999}\\
    & $\{7\}$ & \numprint{2}\\
    \hline
    $\C_6$-primes & $\{7\}$ & \numprint{2}\\
    \hline
\end{tabular}
\caption{Output of \cref{Alg:C1,Alg:C2,Alg:C3}, after removing extraneous primes, for Picard curves in the box $\calB$.}
\label{table:C3C2reducible_PicardBoxSearch}
\end{center}
\end{table}
\end{center}

Now we report the results obtained from applying our algorithms on Sutherland's dataset of \numprint{3024913} Picard curves of $7$-smooth conductor. We note that the dataset includes curves that are twists. Our algorithms succeed for \numprint{2413173} curves. The failures consist of bielliptic Picard curves $y^3 = x^4 + a x^2 + c$, curves given by $y^3 = x^4 + bx$, and \numprint{13972} other curves for which heuristic computations indicate the presence of extra geometric endomorphisms besides $\Z[\zeta_3]$.

The \numprint{13972} failures correspond to \numprint{495} curves up to twist. For these curves, we used \cite{CMSV19} to heuristically compute the geometric endomorphism algebra $\End^0(\Jac(C)_{\Qbar})$ and the index of $\End(\Jac(C)_{\Qbar})$ in the maximal order. After sorting them according to this information, we discovered through interpolation of the invariants $a^3/b^2, ac/b^2$, interesting sub-families of Picard curves with generic geometric endomorphism ring strictly bigger than $\Z[\zeta_3]$ (see \cref{table:newPicardCurveFamilies}). While we had enough data points to obtain the first three families by simple interpolation, for the last one, we needed to look for a short vector in an integer lattice. \cref{table:otherPicardCurveEndAlgs} displays several other non-CM endomorphism algebras, that were observed for a very few curves. 

For the families in \cref{table:newPicardCurveFamilies}, we have heuristically verified the endomorphism algebra for all curves corresponding to rational numbers $a$ up to height $20$. Therefore, we believe that these represent special subvarieties of the Picard modular surface: the first two are certain rational modular curves $X_0(n)$
for some $n$ dividing $81$ and $192$ respectively, 
while the last two are quaternionic rational Shimura curves.

Let $C$ be a smooth Picard curve from one of the last two families in \cref{table:newPicardCurveFamilies}, corresponding to some generic value of $a \in \Q$.
Let $A \coloneqq \Jac(C)$, and fix an Abel-Jacobi map $C \hookrightarrow A$. Then $A$ contains an elliptic curve $E$ with $j$-invariant $0$, such that the induced map $C \hookrightarrow A \rightarrow E$ does not factor through an isogeny of elliptic curves.
Let $P$ denote the Prym variety corresponding to this map, so that $P$ is an abelian surface carrying a natural polarisation; the one obtained by restriction of the principal polarisation on $A$. 
The endomorphisms package of \cite{CMSV19} allows us to compute the period matrix of $P$, using idempotents in the endomorphism algebra. 
Using this, we verify that $\End(P_{\Qbar})$ is isomorphic to the maximal order in $B_{15}$ or an order of index $3$ in the maximal order in $B_6$ respectively. The latter order precludes principal polarisation. While the former does not, we numerically verify that the natural polarisation on $P$ is not principal.
\begin{center}
\begin{table}[h!]
\begin{center}
\begin{tabular}{|p{6.7cm}|p{2.4cm}|p{2.3cm}|}
    \hline
    $1$-parameter Picard curve family & Generic $\End^0(A_{\Qbar})$ heuristically & Index in maximal order of $\End(A_{\Qbar})$\\
    \hline
    $y^3 = x(x^3+12x^2+ax+64)$ & $\Q(\zeta_3) \times M_2(\Q)$ & $81 = 3^4$\\
    $y^3 = ((x+1)^2 - a(5a+1))((x-1)^2 + 27a(a+1))$ & $\Q(\zeta_3) \times M_2(\Q)$ & $192 = 2^6 \cdot 3$\\
    $y^3 = (x^2 - 27a(a-1)(3a^2-3a+32))^2 + 2916a^2(a-1)(x - (a-1)(7a-16))^2$ & $\Q(\zeta_3) \times B_{15}$ & $125 = 5^3$\\
    $y^3 = (x - 3a)\big( (x + a)^3 - a(17a^3 + 6a^2 + 24a + 8)(x + a) + a^2(27a^4 + 14a^3 + 60a^2 + 24a + 32) \big)$ & $\Q(\zeta_3) \times B_6$ & $108 = 2^2 \cdot 3^3$ \\
    \hline
\end{tabular}
\caption{Picard curve families with generic geometric endomorphism algebra of Jacobian $A$ strictly bigger than $\Q(\zeta_3)$.}
\label{table:newPicardCurveFamilies}
\end{center}
\end{table}
\end{center}

\begin{center}
\begin{table}[h!]
\begin{center}
\begin{tabular}{|c|c||c|c|}
    \hline
    $\End^0(A_{\Qbar})$ & Index & $\End^0(A_{\Qbar})$ & Index\\
    \hline
    $\Q(\zeta_3) \times M_2(\Q)$ & $3 \cdot 7^3$ & 
    $\Q(\zeta_3) \times M_2(\Q)$ & $3^7$\\
    $\Q(\zeta_3) \times M_2(\Q(\sqrt{-2}))$ & $3^7$ & 
    $\Q(\zeta_3) \times M_2(\Q(\sqrt{-6}))$ & $3^4$\\
    $\Q(\zeta_3) \times M_2(\Q(\sqrt{-15}))$ & $3^4$ & 
    $\Q(\zeta_3) \times M_2(\Q(\sqrt{-15}))$ & $5^4$\\
    $\Q(\zeta_3) \times M_2(\Q(\sqrt{-30}))$ & $5^4$ & 
    $M_3(\Q(\zeta_3))$ & $3^{13}$\\
    $M_3(\Q(\zeta_3))$ & $2^{16} \cdot 3^5$ & $M_3(\Q(\zeta_3))$ & $2^{24}$\\
    $\Q(\zeta_3) \times B_{10}$ & $2^2 \cdot 3 \cdot 5^2$ & 
    $\Q(\zeta_3) \times B_6$ & $2^8$\\
    \hline
\end{tabular}
\caption{Some geometric endomorphism algebras $\End^0(A_{\Qbar})$, along with the index of $\End(A_{\Qbar})$ in a maximal order, 
that are observed numerically for Jacobians $A$ of a few Picard curves.}
\label{table:otherPicardCurveEndAlgs}
\end{center}
\end{table}
\end{center}

Among the \numprint{2413173} curves, for which our algorithms succeed,
the largest possibly non-surjective semistable admissible prime found is $13$. It is obtained as a reducible prime for $5$ curves, all of which are twists of the curve:
\begin{align}
\label{eq:PicardCurve-13reducible}
    y^3 = 243 x^4 + 338 x^3 -147 x^2 - 387 x - 142
\end{align}
We also obtain a considerable number of curves having $5$ and $7$ as possibly a reducible prime. \cref{table:C3C2reducible_Picard} gives more detailed information.
\begin{center}
\begin{table}[h!]
\begin{center}
\begin{tabular}{|c|c|c|}
    \hline
    Superset of & Output & Number of curves\\
    \hline
    Semistable $\C_1$-primes & $\{2\}$ & \numprint{2412508}\\
    & $\{2,5\}$ & \numprint{279}\\
    & $\{2,7\}$ & \numprint{381}\\
    & $\{2,13\}$ & \numprint{5}\\
    \hline
    $\C_3$-primes & $\{\}$ & \numprint{2413173}\\
    \hline
    $\C_2$-primes & $\{\}$ & \numprint{2406496}\\
    & $\{2\}$ & \numprint{6559}\\
    & $\{7\}$ & \numprint{106}\\
    & $\{13\}$ & \numprint{5}\\
    & $\{2,7\}$ & \numprint{7}\\
    \hline
    $\C_6$-primes & $\{5\}$ & \numprint{279}\\
    & $\{7\}$ & \numprint{381}\\
    \hline
\end{tabular}
\caption{Output of \cref{Alg:C1,Alg:C2,Alg:C3}, after removing extraneous primes, for Picard curves from Sutherland's $7$-smooth  dataset.}
\label{table:C3C2reducible_Picard}
\end{center}
\end{table}
\end{center}

\subsection{Other smooth plane quartic curves.}
\label{subsec:other_smoothplanequartics}

As part of the work \cite{FKS25} on classification of Sato-Tate groups of abelian threefolds, Fite, Kedlaya and Sutherland conducted several large-scale enumerations of genus $3$ curves over $\Q$, with the aim of realising many Sato-Tate groups through such Jacobians. This was done by extending Sutherland's original computation \cite{Sut19} of the genus $3$ smooth plane quartic database, and retaining all smooth plane quartic models with integer coefficients of absolute value at most $9$, and whose discriminants are $7$-smooth (divisible only by primes $p \leq 7$), or of absolute value at most $10^9$. In their type B dataset, there are $17$ smooth plane quartic curves whose Jacobians appear to have geometric endomorphism algebra isomorphic to $\Q(\sqrt{-7})$. 
We certify the existence of extra endomorphisms for these $17$ curves using the endomorphisms package of \cite{CMSV19}, thus showing that there is a subring of $\End^0(\Jac(C)_{\Qbar})$ isomorphic to $\Z[\frac{1+\sqrt{-7}}{2}]$.
\cref{Alg:C2,Alg:C3,Alg:CO} return the empty set on each of these curves, while \cref{Alg:C1} returns $\{2\}$ for $5$ of these curves and the empty set for the remaining curves. 
Since \cref{Alg:C1,Alg:C2,Alg:C3,Alg:CO} succeed we conclude that the geometric endomorphism algebra does not merely contain $\Q(\sqrt{-7})$, but is rather isomorphic to it, in all these cases. 
The results of the algorithms inform us that the mod-$\ell$ Galois representations have as large an image as possible for every semistable admissible prime $\ell > 3$.

\printbibliography

\end{document}